\documentclass[10pt,reqno]{amsart}

\usepackage{amssymb}
\usepackage{amsfonts}
\usepackage{amsthm}
\usepackage{amsmath}
\usepackage{bbm}

\numberwithin{equation}{section}
\allowdisplaybreaks

\newtheorem{theorem}{Theorem}[section]
\newtheorem{proposition}[theorem]{Proposition}

\newtheorem{lemma}[theorem]{Lemma}

\theoremstyle{definition}

\theoremstyle{remark}

\newcommand{\R}{\mathbb{R}}

\newcommand{\N}{\mathbb{N}}
\newcommand{\Z}{\mathbb{Z}}
\newcommand{\C}{\mathbb{C}}

\renewcommand{\hat}{\widehat}
\renewcommand{\tilde}{\widetilde}
\newcommand{\eps}{\varepsilon}

\newcommand{\scriptA}{\mathcal{A}}
\newcommand{\scriptB}{\mathcal{B}}
\newcommand{\scriptC}{\mathcal{C}}

\newcommand{\scriptF}{\mathcal{F}}

\newcommand{\scriptH}{\mathcal{H}}

\newcommand{\scriptJ}{\mathcal{J}}
\newcommand{\scriptK}{\mathcal{K}}
\newcommand{\scriptL}{\mathcal{L}}
\newcommand{\scriptM}{\mathcal{M}}

\newcommand{\scriptP}{\mathcal{P}}

\newcommand{\scriptS}{\mathcal{S}}

\newcommand{\scriptW}{\mathcal{W}}
\newcommand{\scriptX}{\mathcal{X}}

\newcommand{\qtq}[1]{\quad\text{#1}\quad}
\newcommand{\ctc}[1]{\,\,\text{#1}\,\,}

\DeclareMathOperator*{\supp}{supp}
\DeclareMathOperator*{\diam}{diam}
\DeclareMathOperator*{\dist}{dist}

\begin{document}
\title[Endpoint Lebesgue estimates for averages on curves]{Endpoint Lebesgue estimates for weighted averages on polynomial curves}
\author{Michael Christ}
\address{Department of Mathematics, University of California, Berkeley}
\email{mchrist@berkeley.edu}
\author{Spyridon Dendrinos}
\address{School of Mathematical Sciences, University College Cork, Cork, Ireland}
\email{sd@ucc.ie}
\author{Betsy Stovall}
\address{Department of Mathematics, University of Wisconsin--Madison}
\email{stovall@math.wisc.edu}
\author{Brian Street}
\address{Department of Mathematics, University of Wisconsin--Madison}
\email{street@math.wisc.edu}

\begin{abstract}
We establish optimal Lebesgue estimates for a class of generalized Radon transforms defined by averaging functions along polynomial-like curves.  The presence of an essentially optimal weight allows us to prove uniform estimates, wherein the Lebesgue exponents are completely independent of the curves and the operator norms depend only on the polynomial degree.  Moreover, our weighted estimates possess rather strong diffeomorphism invariance properties, allowing us to obtain uniform bounds for averages on curves satisfying natural nilpotency and nonoscillation hypotheses.
\end{abstract}

\maketitle

%%%%%%%%%%%%%%%%%%%%%%%%%%%%%%%%%%%%%%%%
%%%%%%%%%%%%%%%%%%%%%%%%%%%%%%%%%%%%%%%%
%%%%%%%%%%%%%%%%%%%%%%%%%%%%%%%%%%%%%%%%

\section{Introduction}

%%%%%%%%%%%%%%%%%%%%%%%%%%%%%%%%%%%%%%%%
%%%%%%%%%%%%%%%%%%%%%%%%%%%%%%%%%%%%%%%%
%%%%%%%%%%%%%%%%%%%%%%%%%%%%%%%%%%%%%%%%

Let $(P_1,g_1)$ and $(P_2,g_2)$ be two smooth Riemannian manifolds of dimension $n-1$, with $n \geq 2$.  In \cite{TW}, Tao--Wright established near-optimal Lebesgue estimates for local averaging operators of the form
\begin{equation} \label{E:linear formulation}
Tf(x_2) := \int f(\gamma_{x_2}(t)) a(x_2,t) |\gamma_{x_2}'(t)|_{g_1} dt, \qquad f \in C^0(P_1),
\end{equation}
with $a$ continuous and compactly supported, under the hypothesis that the map $(x_2,t) \mapsto \gamma_{x_2}(t) \in P_1$ is a smooth submersion on the support of $a$.  

Our goal in this article is to sharpen the Tao--Wright theorem to obtain optimal Lebesgue space estimates, without the cutoff, under an additional polynomial-like hypothesis on the map $\gamma$.  We replace the Riemannian arclength with a natural generalization of affine arclength measure; this enables us to prove estimates wherein the Lebesgue exponents are independent of the manifolds and curves involved (provided $\gamma$ is polynomial-like), and operator norms for a fixed exponent pair and fixed polynomial degree are uniformly bounded.  Our results are strongest at the Lebesgue endpoints, where the generalized affine arclength measure is essentially the largest measure for which these estimates can hold and, moreover, the resulting inequalities are invariant under a variety of coordinate changes.  

By duality, bounding the operator $T$ in \eqref{E:linear formulation} is equivalent to bounding the bilinear form
\begin{equation} \label{E:bilinear formulation}
\scriptB(f_1,f_2) := \int_M f_1(\gamma_{x_2}(t)) f_2(x_2) a(x_2,t) |\gamma_{x_2}'(t)|_{g_1} d\nu_2(x_2) dt,
\end{equation}
where $M := P_2 \times \R$.  For the remainder of the article, we will focus on the problem of bounding such bilinear forms.  

\subsection{The Euclidean case}

The Tao--Wright theorem, being local, may be equivalently stated in Euclidean coordinates.  Though we will obtain more general results on manifolds (and also in Euclidean space) by applying diffeomorphism invariance of our operator and basic results from Lie group theory, the Euclidean version is, in some sense, our main theorem.

Let $\pi_1,\pi_2:\R^n \to \R^{n-1}$ be smooth mappings.  Define vector fields
\begin{equation} \label{E:def X}
X_j := \star(d\pi_j^1 \wedge \cdots \wedge d\pi_j^{n-1}),
\end{equation}
where $\star$ denotes the map from $n-1$ forms to vector fields obtained by composing the Riemannian Hodge star with the natural identification of 1 forms with vector fields given by the Euclidean metric.  The geometric significance of the $X_j$ is that they are tangent to the fibers of the $\pi_j$, and their magnitude arises in the coarea formula:
\begin{equation} \label{E:coarea}
|\Omega| = \int_{\pi_j(\Omega)} \int_{\pi_j^{-1}(y)} \chi_\Omega(t) |X_j(t)|^{-1}\, d\scriptH^1(t)\, dy, \qquad \Omega \subseteq \{X_j \neq 0\},
\end{equation}
where $\scriptH^1$ denotes 1-dimensional Hausdorff measure.  

We define a map $\Psi:\R^n \times \R^n \to \R^n$ by 
\begin{equation} \label{E:Psi}
\Psi_x(t) := e^{t_nX_n} \circ \cdots \circ e^{t_1 X_1}(x), 
\end{equation}
where we are using the cyclic notation $X_j := X_{j \!\!\mod 2}$, $j=3,\ldots,n$.  Given a multiindex $\beta$, we define 
\begin{gather}
\label{E:def b}
b = b(\beta) := (\sum_{j \, \rm{odd}}1+\beta_j, \sum_{j\, \rm{even}}1+\beta_j)\\
\label{E:def rho}
\rho_\beta(x) := \bigl|\bigl(\partial_t^\beta \det D_t \Psi_x\bigr)(0)\bigr|^{\frac1{b_1+b_2-1}}\\
\label{E:def p}
(p_1,p_2) = (p_1(b),p_2(b)) := \bigl(\tfrac{b_1+b_2-1}{b_1},\tfrac{b_1+b_2-1}{b_2}\bigr).
\end{gather}

Our main theorem is the following.   

\begin{theorem} \label{T:main}
Let $n \geq 3$, let $N$ be a positive integer, and let $\beta$ be a multiindex.  Assume that the maps $\pi_j$ and associated vector fields $X_j$, defined in \eqref{E:def X} satisfy the following:\\
\emph{(i)} The $X_j$ generate a nilpotent Lie algebra $\mathfrak g$ of step at most $N$, and for each $X \in \mathfrak g$, the map $(t,x) \mapsto e^{tX}(x)$ is a polynomial of degree at most $N$;\\
\emph{(ii)} For each $j=1,2$ and a.e.\ $y \in \R^{n-1}$, $\pi_j^{-1}(\{y\})$ is contained in a single integral curve of $X_j$.\\
Then with $\rho_\beta$ satisfying \eqref{E:def rho} and $p_1,p_2$ as in \eqref{E:def p},
\begin{equation} \label{E:main}
\bigl|\int_{\R^n} f_1 \circ \pi_1(x) \, f_2 \circ \pi_2(x)\, \rho_\beta(x)\, dx\bigr| \leq C_N \|f_1\|_{p_1}\|f_2\|_{p_2},
\end{equation}
 for some constant $C_N$ depending only on the degree $N$.
\end{theorem}

No explicit nondegeneracy (i.e.\ finite type) hypothesis is needed, because the weight $\rho_\beta$ is identically zero in the degenerate case.  

The weights $\rho_\beta$ were introduced in \cite{BSapde}, wherein local, non-endpoint Lebesgue estimates were proved in the $C^\infty$ case for a multilinear generalization.  In Section~\ref{S:Examples}, we give examples showing that the endpoint estimate \eqref{E:main} may fail in the multilinear case, and that it may also fail in the bilinear case when Hypothesis (i), Hypothesis (ii), or the dimensional restriction $n \geq 3$ is omitted.  

Theorem~\ref{T:main} uniformizes, makes global, and sharpens to Lebesgue endpoints the  Tao--Wright theorem for averages along curves, under our additional hypotheses.  (As the Tao--Wright theorem is stated in terms of the spanning of elements from $\mathfrak g$, not the non-vanishing of $\rho$, the relationship between the results will take some explanation, which will be given in Section~\ref{S:TW}.)  Moreover, our result generalizes to the fully translation non-invariant case the results of \cite{DLW, DSjfa, Laghi,  BSlms, BSjfa}, wherein endpoint Lebesgue estimates were established for convolution and restricted X-ray transforms along polynomial curves with affine arclength measure.  

\subsection{Averages on curves in manifolds and other generalizations}
As our results are global and uniform, it is natural to ask whether they lead to global results in the more general setting described at the outset, wherein operators are defined for functions on manifolds.  This is the content of our Theorem~\ref{T:mfold version}.  Roughly speaking, this theorem allows one to compute the $X_j$ and $\rho_\beta$ in local coordinates and removes the polynomial hypothesis in (i) of Theorem~\ref{T:main}.  We leave the precise statement for later because it requires some additional terminology.

Another natural question is the extent to which one can relax hypothesis (i).  In this article, we prove a local result (Proposition~\ref{P:local version}) for the mild generalization considered in \cite{GressmanPoly}, wherein it is only assumed that there exist vector fields tangent to the $X_j$ generating a nilpotent Lie algebra.  A number of counter-examples to other possible generalizations are given in Section~\ref{S:Examples}.

\subsection{Background and sketch of proof}
We turn to an outline of the proof of Theorem~\ref{T:main}, and a discussion of the context in the recent literature.  

We begin with the proof on a single torsion scale $\{\rho_\beta \sim 2^m\}$.  By uniformity, it suffices to consider the case when $m=0$, and thus the restricted weak type version of \eqref{E:main} is equivalent to the generalized isoperimetric inequality
\begin{equation} \label{E:RWT intro}
|\Omega| \lesssim |\pi_1(\Omega)|^{\frac1{p_1}}|\pi_2(\Omega)|^{\frac1{p_2}}, \qquad \Omega \subseteq \{\rho_\beta \sim 1\}.
\end{equation}
With $b$ and $p$ as in \eqref{E:def b} and \eqref{E:def p}, $b = \bigl(\tfrac{1/p_1}{1/p_1+1/p_2-1}, \tfrac{1/p_2}{1/p_1+1/p_2-1}\bigr)$, so \eqref{E:RWT intro} is, after a bit of arithmetic, equivalent to the lower bound 
\begin{equation} \label{E:alpha^b}
\alpha_1^{b_1} \alpha_2^{b_2} \lesssim |\Omega|, \qquad \alpha_j := \tfrac{|\Omega|}{|\pi_j(\Omega)|}.
\end{equation}

To establish \eqref{E:alpha^b}, Tao--Wright \cite{TW}, and later Gressman \cite{GressmanPoly}, used a version of the iterative approach from \cite{ChCCC}.  Roughly speaking, for a typical point $x_0 \in \Omega$, the measure of the set of times $t$ such that $e^{tX_j}(x_0) \in \Omega$ is $\alpha_j$.  Iteratively flowing along the vector fields $X_1,X_2$ gives a smooth map, $\Psi_{x_0}$ (recall \eqref{E:Psi}), from a measurable subset $F \subseteq \R^n$ into $\Omega$.  The containment $\Psi_{x_0}(F) \subseteq \Omega$ must then be translated into a lower bound on the volume of $\Omega$.  

Tao--Wright deduce from linear independence of a fixed $n$-tuple $Y_1,\ldots,Y_n \in \mathfrak g$ (the Lie algebra generated by $X_1,X_2$) a lower bound on some fixed derivative $\partial^\beta$ of the Jacobian determinant $\det D \Psi_{x_0}$.  For typical points $t \in \R^n$, we have a lower bound $|\det D\Psi_{x_0}(t)| \gtrsim |t^\beta||\partial^\beta \det D\Psi_{x_0}(0)|$, and this we should be able to use in estimating the volume of $\Omega$:  
$$
|\Omega| \geq |\Psi_{x_0}(F)| \,\text{`$\gtrsim$'}\, \int_F |\det D\Psi_{x_0}(t)|\, dt\, \text{`$\gtrsim$'}\, |F| |\partial^\beta \det D\Psi_{x_0}(0)| \max_{t \in F} |t^\beta| \gtrsim \alpha^b.
$$
Unfortunately, the failure of $\Psi_{x_0}$ to be polynomial in the Tao--Wright case and the fact that $F$ is not simply a product of intervals means that this deduction is not so straightforward; in particular, the inequalities surrounded by quotes in the preceding inequality are false in the general case.  More precisely, if $\Psi_{x_0}$ is merely $C^\infty$, we cannot uniformly bound the number of preimages in $F$ of a typical point in $\Omega$ (so the first inequality may fail), and even for polynomial $\Psi_{x_0}$, if $F$ is not an axis parallel rectangle, then the inequality $|\det D\Psi_{x_0}(t)| \gtrsim |t^\beta|$ may fail for most $t \in F$.  

In the nonendpoint case of \cite{TW}, it is enough to prove \eqref{E:alpha^b} with a slightly larger power of $\alpha$ on the left; this facilitates an approximation of $F$ by a small, axis parallel rectangle centered at 0, and (using the approximation of $F$) an approximation of $\Psi_{x_0}$ by a polynomial.  These approximations are sufficiently strong that $\Psi_{x_0}$ is nearly finite-to-one on $F$ (see also \cite{ChRL}) and $\det D\Psi_{x_0}$ grows essentially as fast on $F$ as its derivative predicts, giving \eqref{E:alpha^b}.  In \cite{GressmanPoly}, wherein the Lie algebra $\mathfrak g$ is assumed to be nilpotent, the map $\Psi_{x_0}$ is lifted to a polynomial map in a higher dimensional space, abrogating the need for the polynomial approximation.  This leaves the challenge of producing a suitable approximation of $F$ as a product of intervals, and Gressman takes a different approach from Tao--Wright, which avoids the secondary endpoint loss.

In Section~\ref{S:RWT}, we reprove Gressman's single scale restricted weak type inequality.  A crucial step is an alternate approach to approximating one-dimensional sets by intervals.
%Mike
This alternative approach gives us somewhat better lower bounds for the integrals of polynomials on these sets, and these improved bounds will be useful later on.  

An advantage of the positive, iterative approach to bounding generalized Radon transforms has been its flexibility, particularly relative to the much more limited exponent range that seems to be amenable to Fourier transform methods.  A disadvantage of this approach is that it seems best suited to proving restricted weak type, not strong type estimates.  Let us examine the strong type estimate on torsion scale 1.  By positivity of our bilinear form, it suffices to prove
$$
\sum_{j,k} 2^{j+k} \int_{\{\rho_\beta \sim 1\}} \chi_{E^j_1} \circ \pi_1(x) \chi_{E^k_2}\circ \pi_2(x)\, dx \lesssim (\sum_j 2^{jp_1}|E^j_1|)^{\frac1{p_1}}(\sum_k 2^{kp_2}|E^k_2|)^{\frac1{p_2}},
$$
for measurable sets $E^j_1,E^k_2 \subseteq \R^{n-1}$, $j,k \in \Z$.  Thus a scenario in which we might expect the strong type inequality to fail is when there is some large set $\scriptJ$ and some set $\scriptK$ such that the $2^j \chi_{E^j_1}$, $j \in \scriptJ$, evenly share the $L^{p_1}$ norm of $f_1$, the $2^k \chi_{E^k_2}$, $k \in \scriptK$ evenly share the $L^{p_2}$ norm of $f_2$, and the restricted weak type inequality is essentially an equality
\begin{equation} \label{E:heuristic quasiex}
\int_{\{\rho_\beta \sim 1\}} \chi_{E^j_1} \circ \pi_1(x) \chi_{E^k_2}\circ \pi_2(x)\, dx \sim |E^j_1|^{\frac1{p_1}}|E^k_2|^{\frac1{p_2}},
\end{equation}
for each $(j,k) \in \scriptJ \times \scriptK$.  

In \cite{ChQex} a technique was developed for proving strong type inequalities by defeating such enemies, and this approach was used to reprove Littman's bound \cite{Littman} for convolution with affine surface measure on the paraboloid.  This approach was later used \cite{DLW, DSjfa, DSjfa2, Laghi, BSlms, BSjfa} to prove optimal Lebesgue estimates for translation invariant and semi-invariant averages on various classes of curves with affine arclength measure.  Key to these arguments was what was called a `trilinear' estimate in \cite{ChQex}, which we now describe.  We lose if one $E^k_2$ interacts strongly, in the sense of \eqref{E:heuristic quasiex} with many sets $E^j_1$ of widely disparate sizes.  Suppose that $E^k_2$ interacts strongly with two sets $E^{j_i}_1$, $i=1,2$.  Letting 
$$
\Omega_i:=\pi_1^{-1}(E^{j_i}_1) \cap \pi_2^{-1}(E^k_2) \cap \{\rho_\beta \sim 1\},
$$
our hypothesis \eqref{E:heuristic quasiex} and the restricted weak type inequality imply that $\pi_2 (\Omega_i)$ must have large intersection with $E^k_2$ for $i=1,2$; let us suppose that $E^k_2=\pi_2(\Omega_1)=\pi_2(\Omega_2)$.  Assuming that every $\pi_2$ fiber is contained in a single $X_2$ integral curve, for a typical $x_0 \in \Omega_i$, the set of times $t$ such that $e^{tX_2}(x_0) \in \Omega_{i'}$ must have measure about $\alpha_2^{i'}:= \tfrac{|\Omega_{i'}|}{|E^k_2|}$; thus we have $\Psi_{x_0}(F_i) \subseteq \Omega_{i}$ for measurable sets $F_i$, which are not well-approximated by products of intervals centered at 0.  In all of the above mentioned articles \cite{ChQex, DLW, DSjfa, DSjfa2, Laghi, BSlms, BSjfa}, rather strong pointwise bounds on the Jacobian determinant $\det D\Psi_{x_0}$ were then used to derive mutually incompatible inequalities relating the volumes of the three sets, $E^{j_1}_1, E^{j_2}_1,E^k_2$ (whence the descriptor `trilinear').  In generalizing this approach, we encounter a number of difficulties.  First, we lack explicit lower bounds on the Jacobian determinant.  We can try to recover these using our estimate $1 \sim |\partial^\beta \det D\Psi_{x_0}(0)|$, but this is difficult to employ on the sets $F_i$, since it is impossible to approximate these sets using products of intervals centered at 0.  Finally, in the translation invariant case, it is natural to decompose the bilinear form in time,
$$
\scriptB(f_1,f_2) = \sum_j \int_{\R^{n-1}}\int_{t \in I_j} f_1(x-\gamma(t)) f_2(x)\, \rho_\beta(t)\, dt\, dx,
$$
and, thanks to the geometric inequality of \cite{DW}, there is a natural choice of intervals $I_j$ that makes the trilinear enemies defeatable.  It is not clear to the authors that an analogue of this decomposition in the general polynomial-like case is feasible.

Our solution is to dispense entirely with the pointwise approach.  In Section~\ref{S:Quasiex}, we prove that if the set $\Omega$ nearly saturates the restricted weak type inequality \eqref{E:RWT intro}, then $\Omega$ can be very well approximated by Carnot--Carath\'eodory balls.  Thus, if $E_1$ and $E_2$ interact strongly, then $E_1$ and $E_2$ can be well-approximated by projections (via $\pi_1,\pi_2$) of Carnot--Carath\'eodory balls.  The proof of this inverse result relies on the improved polynomial approximation mentioned above, as well as new information, proved in Section~\ref{S:CC balls}, on the structure of Carnot--Carath\'eodory balls generated by nilpotent families of vector fields.  In Section~\ref{S:ST scale 1}, we prove that it is not possible for a large number of Carnot--Carath\'eodory balls with widely disparate parameters to have essentially the same projection; thus one set $E^k_2$ cannot interact strongly with many $E^j_1$, and so the strong type bounds on a single torsion scale hold.  In Section~\ref{S:ST full}, we sum up the torsion scales.  In the non-endpoint case considered in \cite{BSapde}, this was simply a matter of summing a geometric series, but here we must control the interaction between torsion scales.  The crux of our argument is that many Carnot--Carath\'eodory balls at different torsion scales cannot have essentially the same projection.  

Section~\ref{S:Nilpotent} gives relevant background on nilpotent Lie groups which will be used in deducing from Theorem~\ref{T:main} more general results, including the above-mentioned result on manifolds.  The results of this section are essentially routine deductions from known results in the theory of nilpotent Lie groups, but the authors could not find elsewhere the precise formulations needed here.  In Section~\ref{S:Extensions}, we prove extensions of our result to the nilpotent case and other generalizations.  In Section~\ref{S:Examples}, we give counter-examples to a few ``natural'' generalizations of our main theorem, discuss its optimality at Lebesgue endpoints, and recall the impossibility of an optimal weight away from Lebesgue endpoints.  The appendix, Section~\ref{S:Polynomials}, contains various useful lemmas on polynomials of one and several variables.  Some of these results are new and may be useful elsewhere.  

\subsection*{Acknowledgements}  The authors would like to thank Melanie Matchett--Wood for explaining how one manipulates intersections of sets parametrized by polynomials into varieties, Daniel Erman for suggesting the reference for Theorem~\ref{T:quantitative ag}, and Terence Tao for advice and encouragement in the very early stages of this article.  
Christ was supported by NSF DMS-1363324.
Dendrinos was supported by DFG DE 1517/2-1.
Stovall was supported by NSF DMS-1266336 and DMS-1600458.  
Street was supported by NSF DMS-1401671 and DMS-1764265.
Part of this work was completed at MSRI during the 2017 Harmonic Analysis program, which was funded in part by NSF DMS-1440140.
Finally, the authors would like to thank the anonymous referee, whose thoughtful comments led to substantial improvements in the article.   

\subsection*{Notation}  Constants are allowed to depend on $N$ and may change from line-to-line.  Constants may depend on those that come logically before.  Thus constants in conclusions depend on those arising in proofs (or in lemmas used in the proofs), which in turn depend on $N$ and the constants in hypotheses.  Further subscripts will be used to denote other parameters on which constants depend.  Capital letters (usually $C$) will typically be used to denote large constants and lower case letters (usually $c$) to denote small ones.  

We will use the now-standard $\lesssim, \gtrsim, \sim$ and the non-standard $\lessapprox, \gtrapprox, \approx$.  We describe their use using two nonnegative quantities $A$ and $B$.  When found in the hypothesis of a statement, $A \lesssim B$ means that the conclusion holds whenever $A \leq C B$ for any $C$ (with constants in the conclusion allowed to depend on $C$).  In the conclusion, $A \lesssim B$ means that $A \leq C B$ for some $C$.  Later on, we will introduce a small parameter $0 < \eps \leq 1$, and many quantities depend on $\eps$ in some way as well.  We will use $A \lessapprox B$ to mean that $A \leq C \eps^{-C}B$ for $C$ quantified in the same way as the implicit constant in the $\lesssim$ notation.  (In Section~\ref{S:ST full}, this notation will depend instead on a small parameter $\delta >0$.)   Finally $A \sim B$ means $A \lesssim B$ and $B \lesssim A$, and $A \approx B$ means $A \lessapprox B$ and $B \lessapprox A$.  We will occasionally subscript these symbols to indicate their dependence on parameters other than $N$.

%%%%%%%%%%%%%%%%%%%%%%%%%%%%%%%%%%%%%%%%
%%%%%%%%%%%%%%%%%%%%%%%%%%%%%%%%%%%%%%%%
%%%%%%%%%%%%%%%%%%%%%%%%%%%%%%%%%%%%%%%%

\section{The restricted weak type inequality on a single scale} \label{S:RWT}

%%%%%%%%%%%%%%%%%%%%%%%%%%%%%%%%%%%%%%%%
%%%%%%%%%%%%%%%%%%%%%%%%%%%%%%%%%%%%%%%%
%%%%%%%%%%%%%%%%%%%%%%%%%%%%%%%%%%%%%%%%

This section is devoted to a proof, or, more accurately, a reproof, of the restricted weak type inequality on the region where $\rho_\beta \sim 1$.  The following result is due to Gressman in \cite{GressmanPoly}.  (Uniformity is not explicitly claimed in \cite{GressmanPoly}, but the arguments therein may easily be adapted.)

\begin{proposition} \label{P:RWT} \cite{GressmanPoly}
For each pair $E_1,E_2 \subseteq \R^{n-1}$ of measurable sets,
\begin{equation} \label{E:RWT}
|\{\rho_\beta \sim 1\} \cap \pi_1^{-1}(E_1) \cap \pi_2^{-1}(E_2)| \lesssim |E_1|^{1/p_1} |E_2|^{1/p_2}
\end{equation}
holds uniformly, with definitions and hypotheses as in Theorem~\ref{T:main}.
\end{proposition}

We give a complete proof of the preceding, using partially alternative methods from those in \cite{GressmanPoly}, because our approach will facilitate a resolution, in Section~\ref{S:Quasiex}, of a related inverse problem, namely, to characterize those pairs $(E_1,E_2)$ for which the inequality in \eqref{E:RWT} is reversed.  Our proof of Proposition~\ref{P:RWT} is based on the following proposition.

\begin{proposition} \label{P:poly refinement}
Let $S \subseteq \R$ be a measurable set.  For each $N$, there exists an interval $J = J(N,S)$ with $|J \cap S| \gtrsim |S|$ such that for any polynomial $P$ of degree at most $N$,
\begin{equation} \label{E:poly refinement}
\int_S|P| \, dt \gtrsim \sum_{j=0}^N \|P^{(j)}\|_{L^\infty(J)} \bigl(\tfrac{|J|}{|S|}\bigr)^{(1-\eps)j}|S|^{j+1}.
\end{equation}
\end{proposition}

The key improvement of this lemma over the analogous result in \cite{GressmanPoly} is the gain $(\frac{|J|}{|S|})^{(1-\eps)j}$ in the higher order terms.  This gain will allow us to transfer control over $\int_S |P|$ into control over the length of $J$.

\begin{proof}[Proof of Proposition~\ref{P:poly refinement}]  

If $S$ has infinite measure, the left hand side of \eqref{E:poly refinement} is infinite whenever it is nonzero.  Thus we may assume that $S$ has finite measure.  Replacing $S$ by a bounded subset with comparable measure, we may assume that $S \subseteq I$ for some finite interval $I$.  Now we turn to a better approximation.

\begin{lemma} \label{L:poly refinement1}
Given $c > 0$, there exist intervals $J,K \subseteq I$ with the following properties.
\begin{itemize}
\item[i.  ] $|J| \sim |K| \sim \dist(K,J)$
\item[ii.  ] $|S \cap J| \gtrsim |S|$
\item[iii.  ] $|S \cap K| \gtrsim (\frac{|S|}{|K|})^c  |S|$.
\end{itemize}
\end{lemma}

\begin{proof}
Let $c' > 0$ be a small constant, to be determined.  

Starting from $i=0$ and $I_0 = I$, we use the following stopping time procedure.  

Let $m_i := \lceil \log_{\frac43}(\frac{|I_i|}{|S|})\rceil$.  Divide $I_i = I_i^1 \cup I_i^2 \cup I_i^3 \cup I_i^4$ into four non-overlapping intervals of equal length, arranged in order of increasing index.  

If 
$$
|S \cap I_i^j| > c' 2^{-c m_i}|S \cap I_i|,
$$
for $j=1$ and $j=4$, then stop.  Set $J=I_i^j$, where $j$ is chosen to maximize $|S \cap I_i^j|$ and set $K=I_i^k$, where $k \in \{1,4\}$ is not adjacent to $j$.  Then we are done, provided $|S \cap I_i| \gtrsim |S|$.  

If (say) $|S \cap I_i^1| \leq c' 2^{-c m_i}|S \cap I_i|$ (the case where $|S \cap I_i^4| \leq c' 2^{-c m_i}|S \cap I_i|$ being handled analogously), discard $I_i^1$ and repeat the procedure on $I_{i+1} := I_i^2 \cup I_i^3 \cup I_i^4$.  Note that $m_{i+1} = m_i-1$.  

On the one hand, $|I_i| = 2 (\frac34)^i|I_0|$ tends to zero as $i \to \infty$, while on the other hand, 
$$
|I_i| \geq |S \cap I_i| \geq \bigl[\prod_{j=0}^{i-1}(1-c' 2^{-m_jc})\bigr]|S| \gtrsim |S|,
$$
where the last inequality is valid for $c'$ sufficiently small.  Thus the process terminates after finitely many steps.  
\end{proof}

We apply Lemma~\ref{L:poly refinement1} iteratively, $N$ times, to obtain a sequence of pairs of bounded intervals $K_1,J_1 \subseteq I$, $K_{i+1},J_{i+1} \subseteq J_i$, $1 \leq i \leq N-1$, satisfying
\begin{gather*}
|K_i| = |J_i| = \dist(K_i,J_i)\\
|S \cap J_i| \gtrsim |S|\\
|S \cap K_i| \gtrsim \bigl(\tfrac{|S|}{|K_i|}\bigr)^c |S|.
\end{gather*}
Let $m_i := \log_2\bigl(\frac{|K_i|}{|S|}\bigr)$.  We observe that $m_1 \geq m_2 \geq \cdots \geq m_N$.  

It remains to prove that if $P$ is any degree $N$ polynomial,
\begin{equation} \label{E:Int S P}
\int_S |P| \gtrsim  \sum_{j=0}^N \|P^{(j)}\|_{L^\infty(J_N)} 2^{j(1-c)m_N} |S|^{j+1}.
\end{equation}

We will repeatedly use, without comment, the equivalence of all norms on the finite dimensional vector space of polynomials of degree at most $N$.  (Examples of norms that we use are $\|P\|_{L^\infty([0,1])}$, $\sum_j |P^{(j)}(\zeta_0)|$ for a fixed $\zeta_0 \in \{|\zeta| < 1\}$, $\|P\|_{L^1([0,1])}$, $\|P\|_{L^\infty([\frac23,1])}$, $\|P\|_{L^\infty(\{|z| < 1\})}$, etc.)  By scaling and translation, we can map $[0,1]$ onto any closed interval, and the norms transform accordingly.   

Multiplying $P$ by a constant if needed, we may write $P(t) = \prod_{j=1}^N(z-\zeta_j)$, where the $\zeta_j$ are the complex zeros, counted according to multiplicity. 

First, suppose that $\dist(\zeta_j,J_N) \geq \tfrac1{100}|J_N|$ for all $j$.  Then $|P(t)| \sim |P(t_0)|$ throughout $J_N$, so
$$
\int_S |P| \geq \int_{S \cap J_N}|P| \gtrsim \|P\|_{L^\infty(J_N)} |S| \sim \sum_{j=0}^N |P^{(j)}(t_0)| |J_N|^j |S|,
$$
which dominates the right side of \eqref{E:Int S P}.  

Now suppose that $\dist(\zeta_1,J_N) < \tfrac1{100}|J_N|$.  We have that 
\begin{align*}
\|P\|_{L^\infty(J_N)} &\sim \sum_{j=0}^N |P^{(j)}(\zeta_1)||J_N|^j = |J_N| \sum_{j=0}^{N-1}|(P')^{(j)}(\zeta_1)| |J_N|^j\\
&\sim |J_N|\|P'\|_{L^\infty(J_N)} \sim \sum_{j=1}^N \|P^{(j)}\|_{L^\infty(J_N)} |J_N|^j.
\end{align*}
By construction, for each $j \geq 2$, $\dist(\zeta_j,K_i) < \frac1{100}|K_i|$ can hold for at most one value of $i$.  Thus there exists $1 \leq i \leq N$ such that $\dist(\zeta_j,K_i) \geq \tfrac1{100}|K_i|$ for all $j$, so $|P(t)| \sim |P(t_i)|$, for any $t,t_i \in K_i$.  Therefore
\begin{equation} \label{E:int S}
\begin{aligned}
\int_S |P| &\gtrsim \int_{S \cap K_i} |P| \sim \|P\|_{L^\infty(K_i)} |S \cap K_i| \sim \|P\|_{L^\infty(J_i)}|S \cap K_i| \\
&\gtrsim \sum_{j=1}^N\|P^{(j)}\|_{L^\infty(J_i)}|J_i|^j|S \cap K_i| \gtrsim \sum_{j=1}^N\|P^{(j)}\|_{L^\infty(J_N)}2^{(j-c)m_i}|S|^{j+1},
\end{aligned}
\end{equation}
which is again larger than the right side of \eqref{E:Int S P}.  
\end{proof}

\begin{proof}[Proof of Proposition~\ref{P:RWT}]
We may assume that $E_1,E_2$ are open sets.  We take the now-standard approach of iteratively refining the set 
$$
\pi_1^{-1}(E_1) \cap \pi_2^{-1}(E_2) \cap \{\rho_\beta \sim 1\}.
$$
Since $X_j \neq 0$ a.e.\ on $\{\rho_\beta \sim 1\}$, $\pi_j$ is a submersion a.e.\ on $\{\rho_\beta \sim 1\}$.  By the implicit function theorem and hypothesis (ii) of Theorem~\ref{T:main}, points $x,x'$ at which $\pi_j$ is a submersion that lie on distinct $X_j$ integral curves cannot have $\pi_j(x) = \pi_j(x')$.  Thus there exists an open set $\Omega \subseteq \R^n$ with
$$
\Omega \subseteq \pi_1^{-1}(E_1) \cap \pi_2^{-1}(E_2) \cap \{\rho_\beta \sim 1\}, \qquad |\Omega| \sim |\pi_1^{-1}(E_1) \cap \pi_2^{-1}(E_2) \cap \{\rho_\beta \sim 1\}|,
$$
such that for each $y \in \R^{n-1}$, $\pi_j^{-1}(\{y\}) \cap \Omega$ is contained in a single integral curve of $X_j$.  

Define
$$
\alpha_j := \tfrac{|\Omega|}{|E_j|}, \qquad j=1,2.
$$
Our goal is to prove that
$$
|\Omega| \gtrsim \alpha_1^{b_1}\alpha_2^{b_2},
$$
where $b=b(p)$ is as in \eqref{E:def b}; after some arithmetic, this implies  \eqref{E:RWT}.  We may thus assume that $\Omega$ is a bounded set.  

We may write the coarea formula as
$$
|\Omega'| = \int_{\pi_j(\Omega')} \int \chi_{\Omega'}(e^{tX_j}(\sigma_j(y)))\, dt\, dy, \qquad \Omega' \subseteq \Omega,
$$
and we use this formula to refine iteratively, starting with $j=n$ and $\Omega_n := \Omega$.  For $x \in \Omega_j$, we define
$$
S_j(x) := \{t : e^{tX_j}(x) \in \Omega_j\}.
$$
Let $\sigma_j:\pi_j(\Omega_j) \to \Omega_j$ be a measurable section of $\pi_j$, with further properties to be determined later.  If we define $t_j(x) \in \R$ by the formula $x =: e^{t_j(x)X_j}(\sigma_j(\pi_j(x)))$, we see that $S_j(x)+t_j(x) = S_j(\sigma_j(\pi_j(x)))$; in particular, both sides depend only on $\pi_j(x)$.  We further define
$$
J_j(x) := J(N,S_j(x)+t_j(x)) - t_j(x),
$$
where $J(N,S)$ is the interval whose existence was guaranteed in Proposition~\ref{P:poly refinement}.  We choose this somewhat cumbersome definition so that $J_j(x) + t_j(x)$ depends only on $\pi_j(x)$ and $|S_j(x) \cap J_j(x)| \gtrsim |S_j(x)|$.  Finally, we set
\begin{equation} \label{E:Omega j-1}
\Omega_{j-1} := \{x \in \Omega_j : |S_j(x)| \geq C_{j}^{-1} \alpha_j, \: 0 \in J_j(x)\},
\end{equation}
with $C_{j}$ sufficiently large.  Note that $0 \in J_j(x)$ if and only if $x \in \{e^{tX_j}(x):t \in S_j(x) \cap J_j(x)\}$.  

We claim that $|\Omega_{j-1}| \sim |\Omega_j|$.  Indeed, 
\begin{align*}
|\Omega_j| &= \int_{\pi_j(\Omega_j)} |S_j(\sigma_j(y))|\, dy
\sim \int_{\pi_j(\Omega_j^g)}|S_j(\sigma_j(y))|\, dy\\
&\sim \int_{\pi_j(\Omega_j^g)}|S_j(\sigma_j(y)) \cap J_j(\sigma_j(y))|\, dy
=|\Omega_{j-1}|,
\end{align*}
where $\Omega_j^g :=\{x \in \Omega_j: |S_j(x)| > C_{j}^{-1} \alpha_j\}$ (same constant as in \eqref{E:Omega j-1}), and the second `$\sim$' uses Proposition~\ref{P:poly refinement}.  

We claim that that each $\Omega_j$ is open (possibly after a minor refinement).  Since $\Omega$ is open, it suffices to prove that $\Omega_{j-1}$ is open whenever $\Omega_j$ is open.  By deleting a set of measure much smaller than $|\Omega_j|$, we may assume that
$$
\Omega_j = \bigcup_{\alpha \in \scriptA} \{e^{tX_j}(\sigma_j(y)) : y \in B_\alpha, t \in S_\alpha\},
$$
where the $B_\alpha$ are disjoint open subsets of $\R^{n-1}$, the $S_\alpha$ are open subsets of $\R$, and $(t,y) \mapsto e^{tX_j}(\sigma_j(y))$ is a diffeomorphism on $B_\alpha \times S_\alpha$.  (We make no hypotheses on $\#\scriptA$.)  Then $S_j(e^{tX_j}(\sigma_j(y))) = S_\alpha + t$ and $J_j(e^{tX_j}(\sigma_j(y))) = J(N,S_\alpha)+t$, for each $(y,t) \in B_\alpha \times S_\alpha$.  By construction, there exists a subset $\scriptA' \subseteq \scriptA$ such that we may write
$$
\Omega_{j-1} = \bigcup_{\alpha \in \scriptA'} \{e^{tX_j}(\sigma_j(y)) : y \in B_\alpha, t \in S_\alpha \cap J(N,S_\alpha)\},
$$
a union of open sets.  

Let $x_0 \in \Omega_0$, and for $t \in \R^n$, define
\begin{equation} \label{E:def Psi}
\Psi_{x_0}(t) := e^{t_n X_n} \circ \cdots \circ e^{t_1X_1}(x_0).
\end{equation}
Define $\scriptF_1 := S_1(x_0)$, and for each $j=2,\ldots,n$,
$$
\scriptF_j := \{(t',t_j) \in \R^j : t' \in \scriptF_{j-1}, \: t_j \in S_j(\Psi_{x_0}(t',0))\}.
$$
Thus for $t \in \scriptF_j$, $\Psi_{x_0}(t,0) \in \Omega_j$, so $0 \in J_j(\Psi_{x_0}(t,0))$.  

In particular, $\Psi_{x_0}(\scriptF_n) \subseteq \Omega$, so by Lemma~\ref{L:finite to one},
$$
|\Omega| \geq |\Psi_{x_0}(\scriptF_n)| \gtrsim \int_{\scriptF_n} |\det D\Psi_{x_0}(t)|\, dt.
$$
Since $0 \in J_j(\Psi_{x_0}(t',0))$ for each $t' \in \scriptF_j$, we compute
\begin{equation} \label{E:refined lb}
\begin{aligned}
&\int_{\scriptF_n} |\det D\Psi_{x_0}(t)|\, dt = \int_{\scriptF_{n-1}} \int_{S_n(\Psi_{x_0}(t',0))}|\det D\Psi_{x_0}(t',t_n)|\, dt_n\, dt'\\
&\qquad \gtrsim \alpha_n^{\beta_n+1} \int_{\scriptF_{n-1}} |\partial_{t_n}^{\beta_n} \det D \Psi_{x_0}(t',0)|\, dt' \\
&\qquad \gtrsim \alpha_n^{\beta_n+1}\cdots\alpha_1^{\beta_1+1} |\partial_t^\beta \det D\Psi_{x_0}(0)| \sim \alpha_1^{b_1}\alpha_2^{b_2}.
\end{aligned}
\end{equation}
After a little arithmetic, we see that \eqref{E:RWT} is equivalent to $|\Omega| \gtrsim \alpha_1^{b_1}\alpha_2^{b_2}$, so the proposition is proved.
\end{proof}

We have not yet used the gain in Proposition~\ref{P:poly refinement}; we will take advantage of that in Section~\ref{S:Quasiex} when we prove a structure theorem for pairs of sets for which the restricted weak type inequality \eqref{E:RWT} is nearly reversed.  Before we state this structure theorem, it will be useful to understand better the geometry of the image under $\Psi_{x_0}$ of axis parallel rectangles.

%%%%%%%%%%%%%%%%%%%%%%%%%%%%%%%%%%%%%%%%%%%%%%%%%
%%%%%%%%%%%%%%%%%%%%%%%%%%%%%%%%%%%%%%%%%%%%%%%%%
%%%%%%%%%%%%%%%%%%%%%%%%%%%%%%%%%%%%%%%%%%%%%%%%%

\section{Carnot--Carath\'eodory balls associated to polynomial flows} \label{S:CC balls}

%%%%%%%%%%%%%%%%%%%%%%%%%%%%%%%%%%%%%%%%%%%%%%%%%
%%%%%%%%%%%%%%%%%%%%%%%%%%%%%%%%%%%%%%%%%%%%%%%%%
%%%%%%%%%%%%%%%%%%%%%%%%%%%%%%%%%%%%%%%%%%%%%%%%%

In the previous section, we proved uniform restricted weak type inequalities at a single scale.  To improve these to strong type inequalities, we need more, namely, an understanding of those sets for which the inequality \eqref{E:RWT} is nearly optimal.  In this section, we lay the groundwork for that characterization by establishing a few lemmas on Carnot--Carath\'eodory balls associated to nilpotent vector fields with polynomial flows.  Results along similar lines have appeared elsewhere, \cite{CNSW, NSW, BrianFrobenius, TW} in particular, but we need more uniformity and a few genuinely new lemmas, and, moreover, our polynomial and nilpotency hypotheses allow for simpler proofs than are available in the general case.  

We begin by reviewing our hypotheses and defining some new notation.  We have vector fields $X_1,X_2 \in \scriptX(\R^n)$ that are assumed to generate a Lie subalgebra $\mathfrak g \subseteq \scriptX(\R^n)$ that is nilpotent of step at most $N$, and such that for each $X \in \mathfrak g$, the exponential map $(t,x) \mapsto e^{tX}(x)$ is a polynomial of degree at most $N$ in $t$ and in $x$.  

\begin{lemma}\label{L:div free}
The elements of $\mathfrak g$ are divergence-free.
\end{lemma}

\begin{proof}
Let $X \in \mathfrak g$.  Both $\det D e^{tX}(x)$ and its multiplicative inverse, which may be written $\det (De^{-tX})(e^{tX}(x))$, are polynomials, so both must be constant in $t$ and $x$.  Evaluating at $t=0$, we see that these determinants must equal 1, so the flow of $X$ is volume-preserving, i.e.\ $X$ is divergence-free.
\end{proof}

A \textit{word} is a finite sequence of 1's and 2's, and associated to each word $w$ is a vector field $X_w$, where $X_{(i)} := X_i$, $i=1,2$, and $X_{(i,w)} := [X_i,X_w]$.  We let $\scriptW$ denote the set of all words $w$ with $X_w \not \equiv 0$.  For $I \in \scriptW^n$, we define $\lambda_I := \det(X_{w_1},\ldots,X_{w_n})$, and we define $\Lambda := (\lambda_I)_{I \in \scriptW^n}$.  We denote by $|\Lambda|$ the sup-norm.  

As in the proof of Proposition~\ref{P:poly refinement}, we will repeatedly, and without comment, use the fact that all norms on the finite dimensional vector space of polynomials of degree at most (e.g.) $N$ are equivalent.  

Throughout this section, $c$ denotes a sufficiently small constant depending on $N$.  

\begin{lemma}\label{L:lambda t sim lambda 0}
Assume that $|\lambda_I(0)| \geq \delta|\Lambda(0)|$, for some $\delta > 0$.  Then for any $w \in \scriptW$,
$$
|\lambda_I(e^{tX_w}(0))| \sim |\lambda_I(0)|, \qquad |\lambda_I(e^{tX_w}(0))| \gtrsim \delta |\Lambda(e^{tX_w}(0))|,
$$
for all $|t| < c\delta$.  
\end{lemma}

\begin{proof}
By Lemma~\ref{L:div free}, $X_w$ is divergence-free.  By the formula for the Lie derivative of a determinant, for any $I' = (w_1',\ldots,w_n') \in \scriptW^n$,
$$
X_w \lambda_{I'} = \sum_{i=1}^n \lambda_{I'_i},
$$
where $I'_i$ is obtained from $I'$ by replacing the $i$-th entry with $[X_w,X_{w_i'}]$.  Thus for each $k$,
\begin{equation} \label{E:dkLambda}
|\tfrac{d^k}{dt^k} \Lambda(e^{tX_w}(0))| \lesssim |\Lambda(0)| \lesssim \delta^{-1}|\lambda_I(0)|.
\end{equation}
As $t \mapsto \Lambda \circ e^{tX_w}(0)$ is a polynomial of bounded degree, the first inequality in \eqref{E:dkLambda} implies that $|\Lambda(e^{tX_w}(0))| \sim |\Lambda(0)|$ for $|t| < c$.  Moreover, \eqref{E:dkLambda} implies that $|\tfrac{d}{dt} \lambda_I(e^{tX_w}(0))| \lesssim \delta^{-1}|\lambda_I(0)|$, so $|\lambda_I(e^{tX_w}(0))| \sim |\lambda_I(0)|$ for $|t| < c \delta$.  The conclusion of the lemma follows.
\end{proof}

For $I = (w_1,\ldots,w_n) \in \scriptW^n$, we define a map
$$
\Phi_{x_0}^I(t_1,\ldots,t_n) := e^{t_n X_{w_n}} \circ \cdots \circ e^{t_1 X_{w_1}}(x_0).
$$

\begin{lemma} \label{L:det D Phi I}
Let $I \in \scriptW^n$, and assume that $|\lambda_I(0)| \geq \delta |\Lambda(0)|$.  Then for all $|t| < c\delta$,
\begin{equation} \label{E:det D Phi I 0}
|\det D\Phi^I_0(t)| \sim |\lambda_I \circ \Phi^I_0(t)| \sim |\lambda_I(0)|,
\end{equation}
and $|\Lambda \circ \Phi^I_0(t)| \sim |\Lambda(0)|$.
\end{lemma}

\begin{proof}
By Lemma~\ref{L:lambda t sim lambda 0} and a simple induction, we have only to show that $|\det D\Phi_0^I(t)| \sim |\lambda_I(0)|$, for all $|t| < c\delta$.  Since the flow of each $X_w$ is volume-preserving, we may directly compute
$$
\det D\Phi^I_0(t) = \det(X_{w_1}(0),\phi_{t_1X_{w_1}}^* X_{w_2}(0),\ldots,\phi_{t_1X_{w_1}}^* \cdots \phi_{t_{n-1}X_{w_{n-1}}}^* X_{w_n}(0)),
$$
where 
$$
\phi_X^*Y(x) := De^{-X}(e^{X}(x))Y(e^X(x)).
$$
Since 
\begin{equation} \label{E:bracket XY}
\tfrac d{dt} \phi_{tX}^* Y = \phi_{tX}^*[X,Y],
\end{equation}
this gives
$$
|\partial_t^\beta \det D\Phi_0^I(0)| \lesssim |\Lambda(0)| \leq \delta^{-1} |\lambda_I(0)| = \delta^{-1}|\det D\Phi_0^I(0)|,
$$
for all multiindices $\beta$.  This gives us the desired bound on $|\det D\Phi_0^I(t)|$, for $|t| < c\delta$.  
\end{proof}

\begin{lemma} \label{L:Phi^I one to one}
Assume that $|\lambda_I(0)| \geq \delta|\Lambda(0)|$.  Then $\Phi_0^I$ is one-to-one on $\{|t| < c\delta\}$, and for each $w \in \scriptW$, the pullback $Y_w:= (\Phi_0^I)^* X_w$ satisfies $|Y_w(t)| \lesssim \delta^{-1}$ on $\{|t| < c\delta\}$.  
\end{lemma}

\begin{proof}
We write $D \Phi_0(t) = A(t,\Phi_0(t))$, where $A$ is the matrix-valued function given by
$$
A(t,x) := (\phi_{-t_nX_{w_n}}^* \cdots \phi_{-t_2X_{w_2}}^* X_{w_1}(x),\cdots,\phi_{-t_nX_{w_n}}^*X_{w_{n-1}}(x),X_{w_n}(x)).
$$
By the nilpotency hypothesis and \eqref{E:bracket XY}, each column of $A$ is polynomial in $t$, and thus may be computed by differentiating and evaluating at $t=0$.  Using the Jacobi identity, iterated Lie brackets of the $X_{w_i}$ may be expressed as iterated Lie brackets of the $X_i$, and so
\begin{equation} \label{E:multi phi as poly Xw}
\phi_{-t_nX_{w_n}}^* \cdots \phi_{-t_{i+1}X_{w_{i+1}}}^* X_{w_i} = X_{w_i} + \sum_{w \in \scriptW} p_{w_i,w}(t)X_w,
\end{equation}
where each $p_{w_i,w}$ is a polynomial in $(t_{i+1},\ldots,t_n)$, with bounded coefficients and $p_{w_i,w}(0) = 0$.  By Cramer's rule, for each $w$, 
\begin{equation} \label{E:Cramer Xw}
\lambda_I X_w = \sum_{i=1}^n \lambda_{I(w,i)} X_{w_i},
\end{equation}
where $I(w,i)$ is obtained from $I$ by replacing $X_{w_i}$ with $X_w$.  Combining \eqref{E:multi phi as poly Xw} and \eqref{E:Cramer Xw}, we may write
$$
A = (X_{w_1},\ldots,X_{w_n}) ( \mathbb I_n + \lambda_I^{-1} P),
$$
where $\mathbb I_n$ is the identity matrix and $P$ is a matrix-valued polynomial whose entries are linear combinations of the products $p_{w_i,w}\lambda_{I(w,i)}$.   Since $p_{w_i,w}$ has bounded coefficients and vanishes at zero, $|p_{w_i,w}(t)| \lesssim \delta$ on $\{|t| < c\delta\}$, and so by Lemma~\ref{L:det D Phi I}, 
\begin{equation} \label{E:P lesssim lambdaI}
|P \circ \Phi_0^I(t)| \lesssim |\lambda_I \circ \Phi_0^I(t)| \sim |\lambda_I(0)|,
\end{equation}
on $\{|t| < c\delta\}$.  

Recalling the definition of $Y_w$ in the statement of the lemma, $Y_{w_i}(0) = \tfrac{\partial}{\partial t_i}$, $1 \leq i \leq n$.  Let
\begin{equation} \label{E:def Yw tilde}
\tilde Y_w := \lambda_I(0)^{-1} (\det D\Phi_0^I ) Y_w.
\end{equation}
By Cramer's rule, $\tilde Y_w$ is a polynomial; we also have $\tilde Y_w(0) = Y_w(0)$.  We expand
\begin{align*}
Y_w(t) &= A(t,\Phi_0(t))^{-1}X_w \circ \Phi_0^I(t)\\
 &= (\mathbb I_n + \lambda_I^{-1}\Phi_0^I(t) P \circ \Phi_0^I(t))^{-1}(X_{w_1} \circ \Phi_0^I(t),\ldots,X_{w_n} \circ \Phi_0^I(t))^{-1}X_w \circ \Phi_0^I(t),
\end{align*}
which directly implies
\begin{equation} \label{E:Ywi perturbs ei}
Y_{w_i}(t) = (\mathbb I_n + \lambda_I^{-1} \circ \Phi_0^I(t) P \circ \Phi_0^I(t))^{-1} e_i.
\end{equation}

By \eqref{E:Ywi perturbs ei} and inequality \eqref{E:P lesssim lambdaI}, 
\begin{equation} \label{E:quant Ywi perturbs ei}
|Y_{w_i}-\tfrac{\partial}{\partial t_i}| \lesssim 1
\end{equation}
on $\{|t| < c\delta\}$.  By Cramer's rule, 
$$
Y_w = \sum_{i=1}^n \tfrac{\lambda_{I(i)} \circ \Phi_0^I}{\lambda_I \circ \Phi_0^I} Y_{w_i},
$$
while Lemma~\ref{L:Phi^I one to one} bounds the coefficients; combined with inequality \eqref{E:quant Ywi perturbs ei}, we obtain
$|Y_w| \lesssim \delta^{-1}$ on $\{|t|<c\delta\}$.  

By inequality \eqref{E:P lesssim lambdaI}, $|\tfrac{\det D\Phi_0^I}{\lambda_I(0)}-1| \lesssim 1$.  Therefore, 
\begin{equation} \label{E:tilde Ywi perturbs ddti}
|\tilde Y_{w_i}-\tfrac{\partial}{\partial t_i}| \lesssim 1
\end{equation}
 on $\{|t| < c\delta\}$.  The vector field $\tilde Y_{w_i}$ is a polynomial that satisfies $\tilde Y_{w_i}(0) = \tfrac{\partial}{\partial t_i}$, while \eqref{E:tilde Ywi perturbs ddti} (and the equivalence of norms) implies bounds on the coefficients of $\tilde Y_{w_i}$; taken together, these imply the stronger estimate
\begin{equation} \label{E:tildeYwi-ddti}
|\tilde Y_{w_i}(t) - \tfrac{\partial}{\partial t_i}| \lesssim \delta^{-1}|t|
\end{equation}
on $\{|t| < c\delta\}$.  Similarly, $|\tfrac{\det D\Phi_0^I(t)}{\lambda_I(0)}-1| \lesssim \delta^{-1}|t|$, whence, from the definition \eqref{E:def Yw tilde} of $\tilde Y_{w_i}$ and \eqref{E:tildeYwi-ddti},
$$
|Y_{w_i}(t)-\tfrac{\partial}{\partial t_i}| \lesssim \delta^{-1}|t|,
$$
on $\{|t|<c\delta\}$.  Therefore
$$
|D_s e^{s_nY_{w_n}} \circ \cdots \circ e^{s_1 Y_{w_1}}(0) - \mathbb I_n| \lesssim \delta^{-1}|t|, \qquad |t| <c\delta,
$$
which, by the contraction mapping proof of the Inverse Function Theorem, implies that
$$
(s_1,\ldots,s_n) \mapsto e^{s_nY_{w_n}} \circ \cdots \circ e^{s_1 Y_{w_1}}(0)
$$
is one-to-one on $\{|t| < c\delta\}$.  Finally, by naturality of exponentiation, $\Phi_0^I$ must also be one-to-one on this region.  
\end{proof}

\begin{lemma} \label{L:doubling}
Let $x_j \in \R^n$, $j=1,2$, and assume that $I_j \in \scriptW^n$ are such that $|\lambda_{I_j}(x_j)| \geq \delta |\Lambda(x_j)|$, $j=1,2$.  Let $0 < \rho < c\delta$.  If $\bigcap_{j=1}^2\Phi_{x_j}^{I_j}(\{|t| < c\delta \rho\}) \neq \emptyset$, then $\Phi_{x_1}^{I_1}(\{|t| < c\delta\rho\}) \subseteq \Phi_{x_2}^{I_2}(\{|t| < \rho\})$. 
\end{lemma}

\begin{proof}
By assumption, each element of $\Phi_{x_1}^{I_1}(\{|t| < c\delta\rho\})$ can be written in the form
$$
e^{t_{3n}X_{w_{3n}}} \circ \cdots \circ e^{t_1X_{w_1}}(x_2),
$$
with $w_j \in \scriptW$, $j=1,\ldots,3n$, and $|t| < 3c\delta\rho$.  Setting $Y_w:=(\Phi_{x_2}^{I_2})^*X_w$, Lemma~\ref{L:Phi^I one to one} (together with the Mean Value Theorem) implies that
$$
|e^{t_{3n}Y_{w_{3n}}} \circ \cdots \circ e^{t_1Y_{w_1}}(0)| < \rho,
$$
whenever $|t|<3c\delta \rho$, and so the containment claimed in the lemma follows by applying $\Phi^{I_2}_{x_2}$ to both sides.  
\end{proof}

We recall that $\Psi_{x_0} = \Phi_{x_0}^{(1,2,1,2,\ldots)}$, and we define $\tilde\Psi_{x_0}:= \Phi_{x_0}^{(2,1,2,1,\ldots)}$.  For $\beta \in \Z_{\geq 0}^n$ a multiindex, we define
\begin{equation} \label{E:def J}
J^\beta(x_0):= \partial^\beta \det D\Psi_{x_0}(0), \qquad \tilde J^\beta(x_0) := \partial^\beta \det D\tilde\Psi_{x_0}(0).
\end{equation}

\begin{lemma} \label{L:Lambda sim J}
\begin{equation} \label{E:Lambda sim J}
|\Lambda(0)| \sim \sum_\beta |J^\beta(0)| + |\tilde J^\beta(0)|.
\end{equation}
\end{lemma}

\begin{proof}
The argument that follows is due to Tao--Wright, \cite{TW}; we reproduce it to keep better track of constants to preserve the uniformity that we need.

Direct computation shows that the $J^\beta$ and $\tilde J^\beta$ are linear combinations of determinants $\lambda_I$, and it immediately follows that the left side of \eqref{E:Lambda sim J} bounds the right.  

To bound the left side, it suffices to prove that there exists $|t| \lesssim 1$ such that
$$
|\Lambda(0)| \lesssim |\det D\Psi_0(t)| + |\det D\tilde\Psi_0(t)|,
$$
which is equivalent (via naturality of exponentiation and Lemma~\ref{L:Phi^I one to one}) to finding a point $|s| \lesssim 1$ such that
$$
1 \lesssim |\det D_s e^{s_n Y_n} \circ \cdots \circ e^{s_1 Y_1}(0)| + |\det D_s e^{s_n Y_{n+1}} \circ \cdots \circ e^{s_1 Y_2}(0)|,
$$
where the vector fields $Y_i$ are those defined in Lemma~\ref{L:Phi^I one to one}, the $n$-tuple $I$ having been chosen to maximize $\lambda_I(0)$. 

By Lemma~\ref{L:Phi^I one to one}, $\|Y_w\|_{C^N(\{|t| < c\})} \lesssim 1$,  for all $w \in \scriptW$.  By induction, this implies that $|Y_w(0)| \lesssim (|Y_1(0)| + |Y_2(0)|)$.  Since $|Y_{w_i}(0)| = 1$, $|Y_1(0)| + |Y_2(0)| \sim 1$.  Thus \eqref{E:1 sim Jk} holds for $k=1$, $s=0$.  Without loss of generality, we may assume that $|Y_1(0)| \sim 1$. Now we proceed inductively, proving that for each $1 \leq k \leq n$, there exists a point $|(s_1,\ldots,s_{k+1})| < c$ such that
\begin{equation} \label{E:1 sim Jk}
1 \sim |\partial_1 e^{s_k Y_k} \circ \cdots \circ e^{s_1 Y_1}(0) \wedge \cdots \wedge \partial_k e^{s_k Y_k} \circ \cdots \circ e^{s_1 Y_1}(0)|;
\end{equation}
the case $k=1$, $s=0$ having already been proved.  Assume that \eqref{E:1 sim Jk} holds for some $k < n$, $|s| = |s^0| < c$.  Then $(s_1,\ldots,s_k) \mapsto e^{s_kY_k} \circ \cdots \circ e^{s_1 Y_1}(0)$, $|s-s^0| < c'$ parametrizes a $k$-dimensional manifold $M$, and the vector fields
$$
Z_i(e^{s_k Y_k} \circ \cdots \circ e^{s_1 Y_1}(0)) := \partial_{s_i} e^{s_k Y_k} \circ \cdots \circ e^{s_1 Y_1}(0), \qquad i=1,\ldots,k,
$$
form a basis for the tangent space of $M$ at each point.  

Let us suppose that the analogue of \eqref{E:1 sim Jk} for $k+1$ fails.  Then for all $|s-s^0| < c'$, we may decompose $Y_{k+1}$ as 
\begin{equation} \label{E:Yk+1 parallel}
Y_{k+1}(e^{s_kY_k} \circ \cdots \circ e^{s_1Y_1}(0)) = \sum_{i=1}^k a_i(s) Z_i(e^{s_kY_k} \circ \cdots \circ e^{s_1 Y_1}(0)) + Y^\perp(e^{s_kY_k} \circ \cdots \circ e^{s_1 Y_1}(0)),
\end{equation}
with $\|a_i\|_{C^N(\{|s-s^0| < 1\})} \lesssim 1$, and $|\partial_s^\alpha Y^\perp| =|Z^\alpha Y^\perp| < c''$, for $c''$ as small as we like and all $|\alpha| < N$; otherwise, by equivalence of norms, the analogue of \eqref{E:1 sim Jk} for $k+1$ would hold.  By construction, $Z_k = Y_k$, thus by induction and \eqref{E:Yk+1 parallel},
$$
|Z_1 \wedge \cdots \wedge Z_k \wedge Y_w(e^{s_k^0Y_k} \circ \cdots \circ e^{s_1^0Y_1}(0))| < c'',
$$
for any word $w$ with $\deg_i(w)>0$, where $i \equiv k+1 \mod 2$.  By \eqref{E:1 sim Jk} and boundedness of the $Y_w$,
$$
|\det(Y_{w_1},\ldots,Y_{w_n})(e^{s_k^0Y_k} \circ \cdots \circ e^{s_1^0Y_1}(0))| < c'',  
$$
for an (possibly different but) arbitrarily small constant $c''$.  Thus
$$
|\lambda_I(e^{s_k^0X_k} \circ \cdots \circ e^{s_1^0X_1}(0))| < c'' |\Lambda(0)|,
$$
which, by Lemma~\ref{L:lambda t sim lambda 0}, contradicts our assumption that $|\lambda_I(0)| \sim |\Lambda(0)|$.
\end{proof}

We say that a $k$-tuple $(w_1,\ldots,w_k) \in \scriptW$ is \textit{minimal} if $w_1,w_2 \in \{(1),(2)\}$, and for $i \geq 3$, $w_i = (j,w_l)$ for some $j=1,2$ and $l < i$.  It will be important later that a minimal $n$-tuple must contain the indices $(1)$, $(2)$, and $(1,2)$.  

\begin{lemma} \label{L:minimal I}
Under the assumption that $|\Lambda(0)| \lessapprox \sum_\beta| J^\beta(0)|$, there exists a minimal $n$-tuple $I^0 \in \scriptW^n$ such that for all $\eps > 0$, 
\begin{equation} \label{E:minimal I}
\bigl|\{x \in \Psi_0(\{|t| < 1\}) : |\lambda_{I^0}(x)| \gtrapprox |\Lambda(x)|\}\bigr| \geq (1-\eps)|\Psi_0(\{|t| < 1\})|.
\end{equation}
\end{lemma}

We recall that the implicit constants in the conclusion can depend on the implicit constants in the hypothesis.

The proof of Lemma~\ref{L:minimal I} will utilize the following simple fact.

\begin{lemma} \label{lemmette}
Let $P$ be a polynomial of degree at most $N$ on $\R^n$.  Then for each $\eps > 0$, 
$$
|\{t \in \mathbb I^n:|P(t)| < \eps \|P\|_{C^0(\mathbb I^n)}\}| \lesssim \eps^{1/C}, \qquad \mathbb I^n:=[-1,1]^n.
$$
\end{lemma}

\begin{proof}[Proof of Lemma~\ref{lemmette}]
Assume that $n=1$, $N \geq 1$, and $P(t) = \prod_{i=1}^N(t-w_i)$.  We may assume $|w_i|\leq 2$ for all $i$, as for other $i$, $|t-w_i| \sim |w_i|$ on $\mathbb I$.  Thus $A:=\|P\|_{C^0(\mathbb I)} \sim 1$.  The set $\{-\eps^2 A^2 < P\overline P < \eps^2 A^2\}$ is a union of at most $2N$ bad intervals $I$ on which $|P|<\eps A$.  For any interval $I \subseteq \mathbb I$, $\|P\|_{C^0(I)} \geq |P^{(N)}||I|^N \gtrsim |I|^N$, so a bad interval has length $|I| \lesssim \eps^{1/N}$.  

Now assume that the lemma has been proved for dimensions at most $n$.  Given $P$, a polynomial of degree at most $N$ on $\R^{n+1}$, set 
$$
Q(t') :=\int_{-1}^1 |P(t',t_{n+1})|^2\, dt_{n+1}, \qquad t' \in \R^n,
$$
a polynomial of degree at most $2N$ on $\R^n$.  If $|P(t)|< \eps\|P\|_{C^0(\mathbb I^{n+1})}$, then $|Q(t')|<\eps\|P\|_{C^0(\mathbb I^{n+1})}^2$, or $|P(t',t_n)|^2 < \eps|Q(t')|$.  However, by equivalence of norms, $|Q(t')| \sim \|P(t',\cdot)\|_{C^0(\mathbb I)}^2$, so $\|Q\|_{C^0(\mathbb I^n)} \sim \|P\|_{C^0(\mathbb I^{n+1})}^2$, and the conclusion follows from the $1$ and $n$-dimensional cases.  
\end{proof}

\begin{proof}[Proof of Lemma \ref{L:minimal I}]
Fix an $n$-tuple $I = (w_1,\ldots,w_n)  \in \scriptW^n$ such that $|\tilde \lambda_I(0)| \gtrsim |\tilde \Lambda(0)|$, where $\tilde\lambda_I$ and $\tilde \Lambda$ are defined using vector fields $\tilde X_w$ generated by the $\tilde X_i := KX_i$, $i=1,2$.  (The constant $K$ allows us to apply the technical lemmas above on large balls.) By Lemma~\ref{L:Phi^I one to one}  (see also the proof of Lemma~\ref{L:doubling}), for $K$ sufficiently large,
\begin{equation} \label{E:Psi_0 in Phi_0^I}
 \Psi_0(\{|t|<1\}) \subseteq \Phi_0^I(\{|t|<C\}).
\end{equation}
With the vector fields $Y_w$ defined as in Lemma~\ref{L:Phi^I one to one} (using the $X_i$, not the $KX_i$), Lemmas~\ref{L:lambda t sim lambda 0} and~\ref{L:det D Phi I} imply 
\begin{equation} \label{E:det YI sim 1}
|\det(Y_{w_1},\ldots,Y_{w_n})| \sim 1, \qtq{throughout} \{|t|<C\},
\end{equation}
provided $K$ is sufficiently large.

We will prove that there exists a minimal $n$-tuple $I^0 = (w_1^0,\ldots,w_n^0)$ such that 
\begin{equation} \label{E:det Ys sim 1}
\|\det(Y_{w_1^0},\ldots,Y_{w_n^0})\|_{C^0(\{|t|<C\})} \sim 1.
\end{equation}
Before proving \eqref{E:det Ys sim 1}, we show that it implies inequality \eqref{E:minimal I}.  

By Lemma~\ref{L:lambda t sim lambda 0} $|\Lambda(x)| \sim |\Lambda(0)|$ for all $x \in \Phi_0^I(\{|t|<C\})$.  Thus, unwinding the definition (from Lemma~\ref{L:Phi^I one to one}) of the $Y_w$, \eqref{E:det Ys sim 1} implies that
\begin{equation} \label{E:unwind det Y}
%\|\lambda_{I^0} \circ \Phi_0^I\|_{C^0\{|t|<1\}} \sim
\|\lambda_{I^0} \circ \Phi_0^I\|_{C^0\{|t|<C\}} \sim  |\Lambda(0)| \sim \|\Lambda \circ \Psi_0\|_{C^0\{|t|<1\}}. 
\end{equation}
By \eqref{E:unwind det Y} and \eqref{E:Psi_0 in Phi_0^I},
\begin{align*}
&\{x \in \Psi_0(\{|t|<1\}) : |\lambda_{I^0}(x)| < \delta |\Lambda(x)|\}
%&\qquad  
%\subseteq \{x \in \Psi_0(\{|t|<1\}) : |\lambda_{I^0}(x)| \lesssim \delta|\Lambda(0)\}\\
%&\qquad 
\subseteq \{x \in \Phi_0^I(\{|t|<C\}) : |\lambda_{I^0}(x)| \lesssim \delta |\Lambda(0)|\},
\end{align*}
for $\delta > 0$.  
By the change of variables formula and \eqref{E:unwind det Y}, Lemmas~\ref{L:det D Phi I} and~\ref{lemmette}, our hypothesis and the equivalence of norms, and finally the change of variables formula and Lemma~\ref{L:finite to one}, 
\begin{align*}
&|\{x \in \Phi_0^I(\{|t|<C\}) : |\lambda_{I^0}(x)| \lesssim \delta |\Lambda(0)|\}|\\
&\qquad \leq \|\det D\Phi_0^I\|_{C^0(\{|t|<C\})}|\{|t|<C : |\lambda_{I^0} \circ \Phi_0^I(t)| \lesssim \delta \|\lambda_{I^0} \circ \Phi_0^I\|_{C^0(\{|t|<C\})}\}|\\
&\qquad \lesssim \delta^{1/C} |\Lambda(0)| \lessapprox \delta^{1/C} \|\det D\Psi_0\|_{L^1(\{|t|<1\})} \lesssim \delta^{1/C}|\Psi_0(\{|t|<1\})|.
\end{align*}
Setting $\delta = c'\eps^{C'}$ with $c'$ and $C'$ sufficiently large depending on $C$ and the implicit constant in the hypothesis of the lemma yields \eqref{E:minimal I}.  

It remains to prove \eqref{E:det Ys sim 1}.  We will prove inductively that for each $1 \leq k \leq n$, there exists a minimal $k$-tuple $(w_1^0,\ldots,w_k^0)$ such that $\|Y_{w_1^0} \wedge \cdots \wedge Y_{w_k^0}\|_{C^0(\{|t|<1\})} \sim 1$.  Boundedness of the $Y$'s and our hypothesis imply that $|Y_1(0)| \sim 1$.  For the induction step, it will be useful to have two constants, $c,\delta_N>0$, depending only on $N$.  We will choose $c$ sufficiently small that the deductions below are valid, and then choose $\delta_N$ sufficiently small (depending on $c$ and various implicit constants) to derive a contradiction if the induction step fails.  

Suppose that for some $k < n$, we have found a minimal $k$-tuple $(w_1^0,\ldots,w_k^0)$, with $w_1^0 = (1)$, and some $|t^0| <1$ such that 
\begin{equation} \label{E:Yw1 through Ywk}
|Y_{w_1^0}(t^0) \wedge \cdots \wedge Y_{w_k^0}(t^0)| \sim 1.
\end{equation}
Set $\scriptW_k^0:=\{w_1^0,\ldots,w_k^0\}$.    

By \eqref{E:det YI sim 1}, we may extend $Y_{w_1^0},\ldots,Y_{w_k^0}$ to a frame on $\{|t-t^0|<c\}$ by adding vector fields $Y_{w_i}$.  Thus (after possibly reordering the $w_i$) failure of the inductive step implies that for each 
$$
w \in \scriptW_k^1 := 
\begin{cases}
\{(1),(2)\}, \qtq{if} k=1,\\
\scriptW_k^0 \cup \{(i,w):i \in \{1,2\}, \, w \in\scriptW_k^0\}, \: k > 1,
\end{cases}
$$
$|Y_{w_1^0} \wedge \cdots \wedge Y_{w_k^0} \wedge Y_w(t)| < \delta_N$ for all $t$ such that $|t-t^0|<c$.  Therefore we can write
\begin{equation} \label{E:Yw flat}
Y_w(t) = \sum_{i=1}^k a_w^i(t)Y_{w_i^0}(t) + \sum_{j=k+1}^n a_w^j(t)Y_{w_j}(t), \quad w \in \scriptW_k^1,
\end{equation}
where
\begin{equation} \label{E:a_w^i is really small}
\|a_w^i\|_{C^N(\{|t-t^0| < c\})} \lesssim 
\begin{cases}
1,\qquad &1 \leq i \leq k\\
\delta_N,\qquad & k+1 \leq j \leq n.
\end{cases}
\end{equation}
 Taking the Lie bracket of $Y_i$, $i=1,2$ (or just $Y_1$, when $k=1$), with some $Y_w$, $w \in \scriptW_k^1 \setminus \scriptW_k^0$, 
$$
[Y_i,Y_w] = \sum_{i=1}^k Y_i(a_i)Y_{w_i^0} + \sum_{i=1}^k a_i^0[Y_i,Y_{w_i^0}] + \sum_{j=k+1}^n Y_i(a_j)Y_{w_j} + \sum_{j=k+1}^n a_j[Y_i,Y_{w_j}],  
$$
and we see that (\ref{E:Yw flat}-\ref{E:a_w^i is really small}) hold for 
$$
w \in \scriptW_k^2 := \scriptW_k^1 \cup \{(i,w):i \in \{1,2\}, \, w \in\scriptW_k^1\}.
$$
By induction, (\ref{E:Yw flat}-\ref{E:a_w^i is really small}) are valid for each $Y_w$, $w \in \scriptW$, so 
$$
|\det(Y_{w_1}(t^0),\ldots,Y_{w_n}(t^0))| < \delta_N
$$ 
(because the $Y_{w_i}$ must all lie near the span of $Y_{w_1^0},\ldots,Y_{w_k^0}$), a contradiction to \eqref{E:det YI sim 1}.  
\end{proof}

For $I = (w_1,\ldots,w_n) \in \scriptW^n$ and $\sigma \in S_n$ a permutation, we set $I_\sigma := (w_{\sigma(1)},\ldots,w_{\sigma(n)})$.  

\begin{lemma} \label{E:cover by minimal Phi}
Assume that $|\Lambda(0)| \lessapprox \sum_\beta| J^\beta(0)|$. There exist $c',C'$ such that for all sufficiently small $c$ and large $C$, the following holds.   There exists a minimal $n$-tuple $I \in \scriptW^n$, which is allowed to depend on the $X_j$, such that for all $\eps > 0$, there exists a collection $\scriptA \subseteq \Psi_0(\{|t| < 1\})$, of cardinality $\#\scriptA \lessapprox_{c,C} 1$, such that\\
(i)  $|\Psi_0(\{|t| < 1\}) \cap \bigcup_{x \in \scriptA} \bigcap_{\sigma \in S_n} \Phi_x^{I_\sigma}(\{|t| < c\eps^C\})| \geq (1-\eps)|\Psi_0(\{|t| < 1\})|$,\\
and, moreover, for all $x \in \scriptA$, $\sigma,\sigma' \in S_n$, and $y \in \Phi_x^{I_\sigma}(\{|t| < c\eps^C\})$,\\
(ii) $\Phi_y^{I_{\sigma'}}$ is one-to-one on $\{|t| < c' \eps^{C'}\}$, with Jacobian determinant
$$
|\det D\Phi_y^{I_{\sigma'}}(t)| \sim |\lambda_I(y)| \sim |\lambda_I(x)| \gtrapprox |\Lambda(x)| \sim |\Lambda(\Phi_y^{I_{\sigma'}}(t))|,
$$
(iii) $\Phi_x^{I_\sigma}(\{|t| < c\eps^C\}) \subseteq \Phi_y^{I_{\sigma'}}(\{|t| < c' \eps^{C'}\})$.
\end{lemma}

\begin{proof}
By Lemma~\ref{L:minimal I}, there exists $\delta \gtrapprox 1$ and a minimal $I \in \scriptW^n$ such that if
$$
G:= \{x \in \Psi_0(\{|t| < 1\}) : |\lambda_I(x)| \geq \delta |\Lambda(x)|\},
$$
then $|G| \geq (1-\eps)|\Psi_0(\{|t| < 1\})|$.  We may assume: that $c'$ is sufficiently small, that $\eps^{C'} < c'\delta$, and that $c \eps^C < c'^3\delta^2$.  Conclusions (ii) and (iii) of the lemma for any choice of such balls are direct applications of Lemmas~\ref{L:det D Phi I}, \ref{L:Phi^I one to one}, and \ref{L:doubling}.  

It remains to cover $G$ by a controllable number of balls of the form
$$
B_x(\rho) := \bigcap_{\sigma \in S_n} \Phi_x^{I_{\sigma}}(\{|t| < \rho\}), \qquad x \in G,
$$
in the special case $\rho = c \eps^C$.  
We will use the generalized version of the Vitali Covering Lemma in \cite{BigStein}, for which we need to verify the doubling and engulfing properties.  By Lemma~\ref{L:doubling}, for all $0 < \rho < c'\delta$, $\sigma \in S_n$, and $x \in G$,
\begin{equation} \label{E:size B}
\Phi_x^{I_\sigma}(\{|t| < \rho\}) \supseteq B_x(\rho) \supseteq \Phi_x^{I_\sigma}(\{|t| < c'\delta\rho\}).
\end{equation}
Hence by Lemma~\ref{L:det D Phi I}, $|B_x(\rho)| \approx |\lambda_I(x)|\rho^n \approx |\Lambda(0)|\rho^n$.  Therefore the balls are indeed doubling.  The engulfing property also follows from Lemma~\ref{L:doubling}, since $B_{x_1}(c'\delta\rho) \cap B_{x_2}(c'\delta\rho) \neq \emptyset$ implies that $B_{x_1}(c'\delta\rho) \subseteq B_{x_2}(\rho)$.  

If we choose $\scriptA \subseteq G$ so that $\{B_x(c^2\eps^C)\}_{x \in \scriptA}$ is a maximal disjoint set, then $\bigcup_{x \in \scriptA}B_x(c^2\eps^C) \subseteq \Psi_0(\{|t| < 2\})$ and $G \subseteq \bigcup_{x \in \scriptA}B_x(c\eps^C)$.  Applying \eqref{E:size B} and Lemma~\ref{L:Lambda sim J},
$$
\#\scriptA |\Lambda(0)|(c^2\eps^C)^n \lessapprox |\Psi_0(\{|t| < 2\})| \lesssim |\Lambda(0)|.
$$
\end{proof}

\section{Connection with the work of Tao--Wright} \label{S:TW}

In this section, we translate Theorem~\ref{T:main} into results more closely connected with the main theorem of \cite{TW}.  We are also able to prove variants of Theorem~\ref{T:main} with weights that are, in principle, easier to compute.  

The results of \cite{TW} are stated in terms of the Newton polytope associated to the vector fields $X_1,X_2$.  To define it (and two other, closely related, polytopes), we need some additional notation.  The degree of a word $w \in \scriptW$ is defined to be the element $\deg w \in \Z_{\geq 0}^2$ whose $i$-th entry is the number of $i$'s in $w$.  The degree of a $k$-tuple $I \in \scriptW^k$ is the sum of the degrees of the entries of $I$.  We denote by $\rm{ch}$ the operation of taking  the convex hull of a set.  For $E \subseteq \R^n$, we define
\begin{equation} \label{E:big P}
\scriptP_E^\cup := \rm{ch} \bigcup_{I \in \scriptW^n:\lambda_I \not\equiv 0 \, \text{on}\, E} \deg I + [0,\infty)^2.
\end{equation}
Given $x_0$, we may define
\begin{equation} \label{E:med P}
\scriptP_{x_0} := \scriptP_{\{x_0\}}^\cup.
\end{equation}
Finally, given a set $E \subseteq \R^n$, we may define
\begin{equation} \label{E:small P}
\scriptP_E^\cap := \bigcap_{x_0 \in E} \scriptP_{x_0}.
\end{equation}

In \cite{TW}, Tao--Wright considered bounds of the form 
\begin{equation} \label{E:TW cutoff}
|\int f_1 \circ \pi_1(x) f_2 \circ \pi_2(x)\, a(x) \, dx| \leq C_{a,\pi_1,\pi_2} \|f_1\|_{p_1}\|f_2\|_{p_2},
\end{equation}
with $a$ a continuous function with compact support, $\pi_1,\pi_2$ smooth submersions (no polynomial nor nilpotency hypothesis), and $p_1,p_2 \in [1,\infty]$.  Such bounds are easily seen to be true if $p_1^{-1}+p_2^{-1} \leq 1$.  In the case $p_1^{-1}+p_2^{-1} > 1$, we define 
\begin{equation} \label{E:b(p)}
b(p) = (b_1,b_2):=\bigl(\tfrac{p_1^{-1}}{p_1^{-1}+p_2^{-1} - 1}, \tfrac{p_1^{-1}}{p_1^{-1}+p_2^{-1} - 1} \bigr).
\end{equation}
Tao--Wright proved that \eqref{E:TW cutoff} fails if $(b_1,b_2) \not\in \scriptP_{\{a \neq 0\}}^\cap$ and holds if $(b_1,b_2) \in \rm{int} \, \scriptP_{\rm{supp}\,a}^\cap$.  

These results leave open two natural questions:  what is the role played by the behavior of $a$ near its zero set, and what happens on the boundaries of these polytopes. This article answers these questions in some special cases.  To understand how, we first recall the connection between the polytopes defined above and the weights $\rho_\beta$.  

Given a multiindex $\beta \in \Z_{\geq 0}^n$, we define 
$$
\tilde b(\beta):= \bigl(\sum_{j \,\text{even}} 1+\beta_j, \sum_{j \,\text{odd}} 1+\beta_j\bigr),
$$
and recall the definition \eqref{E:def b} of $b(\beta)$ and \eqref{E:def J} of $J_\beta$ and $\tilde J_\beta$.  Proposition~2.3 of \cite{BSapde} implies that
\begin{equation} \label{E:scriptP cap two ways}
\scriptP_{x_0} = \rm{ch}\left[ \bigl(\bigcup_{\beta:J_\beta(x_0) \neq 0} b(\beta)+[0,\infty)^2\bigr) \cup \bigl(\bigcup_{\beta:\tilde J_\beta(x_0) \neq 0} \tilde b(\beta) + [0,\infty)^2\bigr)\right],  
\end{equation}
and further that for $b$ an extreme point of $\scriptP_{x_0}$,
\begin{equation} \label{E:lambda sim J extreme}
\sum_{I:\deg I = b} |\lambda_I(x_0)|  \sim_b \sum_{\beta:b(\beta) = b} |J_\beta(x_0)| + \sum_{\beta:\tilde b(\beta) = \tilde b} |\tilde J_\beta(x_0)|.  
\end{equation}
The comparison \eqref{E:lambda sim J extreme} is thus valid everywhere on $E$ for $b$ an extreme point of $\scriptP_E^\cup$, since both sides of \eqref{E:lambda sim J extreme} are zero when $b$ is not an extreme point of $\scriptP_{x_0}$.  Combining these results with Theorem~\ref{T:main} and Proposition~2.2 of \cite{BSapde}, we obtain the following sharp result.  

\begin{theorem}  \label{T:sharp extreme}
Assume that hypotheses (i) and (ii) of Theorem~\ref{T:main} are in effect, and that $b:=b(p)$ is an extreme point of $\scriptP_{\R^n}^\cup$.  Then
\begin{equation} \label{E:sharp extreme}
\sup_{f_1,f_2: \|f_1\|_{p_1} = \|f_2\|_{p_2} = 1} \bigl| \int_{\R^n} \prod_{j=1}^2 f_j \circ \pi_j(x)\, a(x)\, dx| \sim \bigl\|\tfrac a{w_b}\bigr\|_{C^0(\{x:w_b(x) \neq 0\})}, 
\end{equation}
where $w_b$ is the weight defined by 
\begin{equation} \label{E:def wb}
w_b := \sum_{I:\deg I = b} |\lambda_I|^{\frac1{b_1+b_2-1}}.
\end{equation}
\end{theorem}

\begin{proof}
The `$\lesssim$' direction directly follows from Theorem~\ref{T:main}, \eqref{E:lambda sim J extreme}, and the comments after \eqref{E:lambda sim J extreme}.  The `$\gtrsim$' direction is a direct application of Proposition~2.2 of \cite{BSapde} (a different notation for polytopes was used in that article).  
\end{proof}

Uniform upper bounds are also possible under slightly weaker hypotheses on $b$.  

\begin{theorem}\label{T:TW version}
Under the hypotheses (i) and (ii) of Theorem~\ref{T:main}, if $b:=b(p)$ is a minimal element of $\scriptP^\cup_{\R^n}$ under the coordinate-wise partial order on $\R^2$, then
$$
|\int_{\R^n} f_1 \circ \pi_1(x)\,  f_2 \circ \pi_2(x)\, w_b(x)\, dx| \lesssim \|f_1\|_{p_1}\|f_2\|_{p_2}.
$$
\end{theorem}

\begin{proof}
Set 
$$
J_b^\Sigma := \sum_{\beta:b(\beta) = b} |J_\beta| + \sum_{\beta:\tilde b(\beta) = b} |\tilde J_\beta|.
$$
By Lemma~\ref{L:Lambda sim J}, for all $\alpha \in (0,\infty)^2$ and $x_0 \in \R^n$, and $\deg I = b$,
\begin{equation} \label{E:lambda lesssim J}
\alpha^b |\lambda_I(x_0)| \lesssim \sum_{b'} \alpha^{b'}J_{b'}^\Sigma(x_0) = \sum_{b' \in \scriptP_{\R^n}^\cup} \alpha^{b'}J_{b'}^\Sigma(x_0).
\end{equation}
By our assumption on $b$ and the definition of $\scriptP_{\R^n}^\cup$, there exists $\nu \in (0,\infty)^2$ such that $b \cdot \nu \leq b' \cdot \nu$ for all $b' \in \scriptP_{\R^n}^\cup$.  Replacing $\alpha = (\alpha_1,\alpha_2)$ with $(\delta^{\nu_1}\alpha_1,\delta^{\nu_2}\alpha_2)$ in \eqref{E:lambda lesssim J} and sending $\delta \searrow 0$, we see that
\begin{equation} \label{E:lambda face}
\alpha^b |\lambda_I(x_0)| \lesssim \sum_{b' \in \mathbb F} \alpha^{b'}J_{b'}^\Sigma(x_0), 
\end{equation}
where $\mathbb F := \{b' \in \scriptP_{\R^n}^\cup : b' \cdot \nu = b \cdot \nu\}$.  The face $\mathbb F$ is a line segment (possibly a singleton),
$$
\mathbb F = \{b^0 + t \omega : 0 \leq t \leq 1\},
$$
for some vector $\omega$ perpendicular to $\nu$.  Setting $\alpha^\omega:=\delta$, \eqref{E:lambda face} is equivalent to 
$$
\delta^{\theta_0} |\lambda_I(x_0)| \lesssim \sum_i \delta^{\theta_i} J_{b^{\theta_i}}^\Sigma(x_0), \qquad \delta > 0,
$$
where $b := b^0+\theta_0\omega$ and $\mathbb F \cap \N^2 = \{b^0 + \theta^i \omega : 1 \leq i \leq m_n\}$.  By Lemma~\ref{L:extract 2 terms}, 
$$
|\lambda_I(x_0)| \lesssim J_b^\Sigma(x_0) + \sum_{\theta_i < \theta_0 < \theta_j} (J_{b^{\theta_i}}^\Sigma(x_0))^{\frac{\theta_j-\theta_0}{\theta_j-\theta_i}}(J_{b^{\theta_j}}^\Sigma(x_0))^{\frac{\theta_0-\theta_i}{\theta_j-\theta_i}}=:(J^{\theta_0})(x_0),
$$
for all $x_0 \in \R^n$.  Finally, by Theorem~\ref{T:main}, complex interpolation, and the triangle inequality,
$$
|\int_{\R^n} f_1 \circ \pi_1(x)\, f_2 \circ \pi_2(x)\, |J^\theta_0(x)|^{\frac{1}{b_1+b_2-1}} \, dx| \lesssim \|f_1\|_{p_1}\|f_2\|_{p_2}.
$$
\end{proof}

Finally, we give the endpoint version of the main result of \cite{TW}.

\begin{theorem} \label{T:TW v2}
Let $a$ be a continuous function with compact support, and assume that $\pi_1,\pi_2$ obey the hypotheses of Theorem~\ref{T:main} and, in addition, that the $\pi_j$ are submersions throughout $\rm{supp}\,a$.  If $b(p) \in \scriptP_{\rm{supp}\,a}^\cap$, then
\begin{equation}\label{E:TW v2}
|\int_{\R^n} f_1 \circ \pi_1(x)\, f_2 \circ \pi_2(x)\, a(x) \, dx| \lesssim_{a,\pi_1,\pi_2} \|f_1\|_{p_1}\|f_2\|_{p_2}.
\end{equation}
\end{theorem}

\begin{proof}
Let $x_0 \in \supp a$.  By \eqref{E:scriptP cap two ways} and our hypothesis, there exist $b^i$, $i=0,1$, and $0 \leq \theta \leq 1$ such that $b(p) \succeq b^\theta:= (1-\theta)b^0+\theta b^1$ and $J_{b^i}^\Sigma(x_0) \sim_{a,\pi_1,\pi_2} 1$.  

By continuity, $J_{b^i}^\Sigma(x) \sim_{a,\pi_1,\pi_2} 1$ for $x$ in some neighborhood $U$ of $x_0$.  By Theorem~\ref{T:main},
\begin{equation} \label{E:TW v2 p01}
|\int_{U} f_1 \circ \pi_1(x)\, f_2 \circ \pi_2(x)\, dx| \lesssim_{a,\pi_1,\pi_2} \prod_{j=1}^2 \|f_j\|_{L^{q_j}(\pi_j(U))}, \qquad q \in [1,\infty]^2,
\end{equation}
holds with $q = p^i$ computed from the $b^i$ using \eqref{E:def p}.  By interpolation, \eqref{E:TW v2 p01} also holds with $q=p^\theta$ computed from $b^\theta$ using \eqref{E:def p}.  An elementary computation shows that 
$$
(p_1^{-1},p_2^{-1}) = (1-\nu)(q^{-1},1-q^{-1}) + \nu((p_1^\theta)^{-1},(p_2^\theta)^{-1}).
$$
Our hypothesis that the $\pi_j$ are submersions on $\rm{supp}\, a$ and H\"older's inequality imply that \eqref{E:TW v2 p01} holds whenever $q_1^{-1}+q_2^{-1} \leq 1$, and hence by interpolation, \eqref{E:TW v2 p01} holds at $p$.  Inequality \eqref{E:TW v2} follows by using a partition of unity.  
\end{proof}

%%%%%%%%%%%%%%%%%%%%%%%%%%%%%%%%%%%%%%%%%%%%%%%%%
%%%%%%%%%%%%%%%%%%%%%%%%%%%%%%%%%%%%%%%%%%%%%%%%%
%%%%%%%%%%%%%%%%%%%%%%%%%%%%%%%%%%%%%%%%%%%%%%%%%

\section{Quasiextremal pairs for the restricted weak type inequality} \label{S:Quasiex}

The purpose of this section is to prove that pairs $E_1,E_2$ that nearly saturate inequality \eqref{E:RWT} are well approximated as a bounded union of ``balls'' parametrized by maps of the form $\Phi^I_{x}$, with $I$ a (reordering of a) minimal $n$-tuple of words.  Results of this type had been previously obtained in \cite{ChQex, BSijm} for other operators and in \cite{BiswasThesis} for a particular instance of the class considered here.  

We begin with some further notation.

\subsection*{Notation}  We recall the maps
$$
\Phi_{x_0}^I(t) := e^{t_n X_{w_n}} \circ \cdots \circ e^{t_1 X_{w_1}}(x_0), \qquad I = (w_1,\ldots,w_n) \in \scriptW^n,
$$
$\Psi_{x_0}:= \Phi_{x_0}^{(1,2,1,2,\ldots)}$, $\tilde\Psi_{x_0}(t):= \Phi_{x_0}^{(2,1,2,1,\ldots)}$ from the previous section.  For $\alpha \in (0,\infty)^2$, we define parallelepipeds
$$
Q_\alpha^I := \{t \in \R^n : |t_i| < \alpha^{\deg w_i}, \: 1 \leq i \leq n\}, \qquad I \in \scriptW^n,
$$
$Q_\alpha:= Q_\alpha^{(1,2,1,2,\ldots)}$, $\tilde Q_\alpha := Q_\alpha^{(2,1,2,1,\ldots)}$.  These give rise to families of balls,
$$
B^I(x_0;\alpha):= \Phi_{x_0}^I(Q_\alpha^I), \qquad B^n(x_0;\alpha):= \Psi_{x_0}(Q_\alpha) \cup \tilde\Psi_{x_0}(\tilde Q_\alpha).
$$
For $I=(w_1,\ldots,w_n)$ an $n$-tuple of words and $\sigma \in S_n$ a permutation, we recall that $I_\sigma:= (w_{\sigma(1)},\ldots,w_{\sigma(n)})$.  

\begin{proposition} \label{P:quasiex}
Let $c,c',C,C'$ be as described in Lemma~\ref{E:cover by minimal Phi}.  Let $E_1,E_2$ be open sets, and let $\eps > 0$.  Define
$$
\Omega:= \{\rho_\beta \sim 1\} \cap \pi_1^{-1}(E_1) \cap \pi_2^{-1}(E_2), \qquad \alpha_j:= \tfrac{|\Omega|}{|E_j|}, \: j=1,2.
$$
If 
\begin{equation}\label{E:quasiex} 
|\Omega| \geq \eps |E_1|^{\frac1{p_1}}|E_2|^{\frac1{p_2}},
\end{equation}
there exist a set $\scriptA \subseteq \Omega$ of cardinality $\#\scriptA \lessapprox_{c,C} 1$ and a minimal $n$-tuple $I \in \scriptW^n$ such that\\
(i)
$$
|\Omega \cap \bigcup_{x \in \scriptA} \bigcap_{\sigma \in S_n} B^{I_\sigma}(x;c\eps^C\alpha)| \gtrsim |\Omega|,
$$
(ii)  For every $x \in \scriptA$, $\sigma,\sigma' \in S_n$, and $y \in B^{I_\sigma}(x;c\eps^C\alpha)$, $\Phi_y^{I_{\sigma'}}$ is one-to-one with Jacobian determinant
$$
|\alpha^{\deg I} \det D\Phi_y^{I_{\sigma'}}| \sim \alpha^{\deg I}|\lambda_I(x)| \approx \alpha^b \approx |\Omega|,
$$
on $Q^{I_{\sigma'}}_{c'\eps^{C'}\alpha}$, and, moreover, $B^{I_\sigma}(x,c\eps^C\alpha) \subseteq B^{I_{\sigma'}}(y;c'\eps^{C'}\alpha)$.
\end{proposition}

By applying Lemma~\ref{E:cover by minimal Phi} with $C'\eps^{-C'}\alpha_1X_1,C'\eps^{-C'}\alpha_2X_2$ in place of $X_1,X_2$, to prove Proposition~\ref{P:quasiex}, it suffices to prove the following.

\begin{lemma} \label{L:Psi balls}
Under the hypotheses of Proposition~\ref{P:quasiex}, there exist a set $\scriptA$ of cardinality $\#\scriptA \lessapprox 1$ such that\\
(i) $|\Omega \cap \bigcup_{x \in \scriptA} B^n(x;C'\eps^{-C'}\alpha)| \gtrsim |\Omega|$\\
(ii) For every $x \in \scriptA$ and $y \in B^n(x;C'\eps^{-C'}\alpha)$, $\sum_I \alpha^{\deg I} |\lambda_I(y)| \approx \alpha^b$.
\end{lemma}

\begin{proof}[Proof of Lemma~\ref{L:Psi balls}]
Inequality \eqref{E:quasiex} implies, after some arithmetic, that
\begin{equation} \label{E:Omega alpha b}
|\Omega| \lessapprox \alpha^b.
\end{equation}
Conversely, the conclusion of Proposition~\ref{P:RWT} is equivalent to $|\Omega| \gtrsim \alpha^b$.  We will prove this lemma by essentially repeating the proof of Proposition~\ref{P:RWT}, while keeping in mind the constraint \eqref{E:Omega alpha b}.  In the proof, we will extensively use the notations from the proof of Proposition~\ref{P:RWT}.  

In the proof of Proposition~\ref{P:RWT}, we only needed to refine the set $\Omega$ $n$ times, but here it will be useful to refine further.  Letting $x_0 \in \Omega_{-1} \subseteq \Omega_0$,
\begin{align*}
\tilde\Psi_{x_0}(t) &\in \Omega_{n-1}, \:\:\text{if}\:\: t_j \in S_{j-1}(\tilde\Psi_{x_0}(t_1,\ldots,t_{j-1},0)), \:\: j=1,\ldots,n\\
\Psi_{x_0}(t) &\in \Omega_{n}, \:\:\text{if}\:\: t_j \in S_{j}(\Psi_{x_0}(t_1,\ldots,t_{j-1},0)), \:\: j=1,\ldots,n.
\end{align*}
Thus exactly the arguments leading up to \eqref{E:refined lb} imply that
$$
|\Omega| \gtrsim \sum_{\beta'} \alpha^{b(\beta')}|\partial^{\beta'}\det D\Psi_{x_0}(0)| + \alpha^{\tilde b(\beta')}|\partial^{\beta'}\det D\tilde\Psi_{x_0}(0)|.
$$
As was observed in \eqref{E:refined lb}, the right side above is at least $\alpha^b$, and by \eqref{E:Omega alpha b}, it is at most $C'\eps^{-C'}\alpha^b$.  Let $\tilde\Omega:=\Omega_{-1}$.  We have just seen that 
$$
\sum_{\beta'} \alpha^{b(\beta')}|\partial^{\beta'}\det D\Psi_{x_0}(0)| + \alpha^{\tilde b(\beta')}|\partial^{\beta'}\det D\tilde\Psi_{x_0}(0)| \approx \alpha^b, \qquad x_0 \in \tilde\Omega,
$$
so by Lemma~\ref{L:Lambda sim J},
\begin{equation} \label{E:Lambda alpha b}
\sum_I \alpha^{\deg I} |\lambda_I(x_0)| \approx \alpha^b, \qquad x_0 \in \tilde \Omega.
\end{equation}
Moreover, by the proof of Proposition~\ref{P:RWT}, $|\tilde\Omega| \sim |\Omega| \approx \alpha^b$.  Thus the proof of our lemma will be complete if we can cover a large portion of $\tilde\Omega$ using a set $\scriptA \subseteq \tilde\Omega$.  

To simplify the notation, we will give the remainder of the argument under the assumption that \eqref{E:Lambda alpha b} holds on $\Omega$; the general case follows from the same proof, since \eqref{E:quasiex} holds with $\Omega$ replaced by $\tilde \Omega$.  Our next task is to obtain better control over the sets $\scriptF_j$, $S_j(\cdot)$ arising in the proof of Proposition~\ref{P:RWT}.  We begin by bounding the measure of these sets.

If $|S_n(x)| \geq C' \eps^{-C'} \alpha_n$ for all $x$ in some subset $\Omega' \subseteq \Omega$ with $|\Omega'| \gtrsim |\Omega|$, we could have refined so that $\Omega_{n-1} \subseteq  \Omega'$, yielding
\begin{align*}
|\Omega| &\gtrsim \alpha_n^{\beta_n} (C'\eps^{-C'} \alpha_n) \int_{\scriptF_{n-1}} |\partial_n^{\beta_n} \det D_t \Psi_{x_0}(t',0)|\, dt'\\
&\geq C' \eps^{-C'} \alpha_1^{b_1} \alpha_2^{b_2},
\end{align*}
a contradiction to \eqref{E:Omega alpha b} for $C'$ sufficiently large.  Thus we may assume that $|S_n(x)| \lessapprox \alpha_n$ on at least half of $\Omega$, and we may refine so that $|S_n(x)| \lessapprox \alpha_n$ throughout $\Omega_{n-1}$.  Similarly, we may refine so that $|S_{n-1}(x)| \lessapprox \alpha_{n-1}$ for each $x \in \Omega_{n-2}$.  Thus, by adjusting the refinement procedure at each step, we may assume that for each $1 \leq j \leq n-1$ and each $t \in \scriptF_{j-1}$,
\begin{equation} \label{E:bound Sj}
|S_j(\Psi_{x_0}(t,0))| = |\{t_j \in \R : (t,t_j) \in \scriptF_j\}| \lessapprox \alpha_j.
\end{equation}

We have not yet used the gain coming from Proposition~\ref{P:poly refinement}.  We will do so now to control the diameter of our parameter set.  The key observation is that we may assume that $\sum_{j \, \rm{odd}} \beta_j$ and $\sum_{j \, \rm{even}} \beta_j$ are both positive.  Indeed, this positivity is trivial for $n \geq 4$, because if $t_j = 0$ for any $1 < j < n$, then $\det D \Psi_{x_0}(t) = 0$.  Thus the only way our claim can fail is if $n=3$ and $\beta = (0,k,0)$, but in this case,
$$
\partial^\beta \det D \Psi_{x_0}(0) = \partial_2 \partial_1^{k-1} \det D\tilde\Psi_{x_0}(0),
$$
and we can simply interchange the roles of the indices 1 and 2 throughout the argument.

Let $j$ be the maximal odd index with $\beta_j > 0$.  Suppose that on at least half of $\Omega_j$, $|J_j(x)| \geq C'\eps^{-C'} |S_j(x)|$.  Then by adjusting our refinement procedure, we may assume that $x \in \Omega_{j-1}$ implies that $|J_j(x)| \geq C' \eps^{-C'} |S_j(x)|$; we note that this implies $|J_j(x)| \geq C'\eps^{-C'} \alpha_j$.  In view of \eqref{E:bound Sj},
\begin{align*}
|\Omega| &\gtrsim \alpha_n^{\beta_n+1} \cdots \alpha_{j+1}^{\beta_{j+1}+1} \int_{\scriptF_{j-1}} |\partial_n^{\beta_n} \cdots \partial_j^{\beta_j} \det D\Psi_{x_0}(t',0)| \\
&\qquad \times \bigl(\tfrac{|J_j(\Psi_{x_0}(t',0))|}{|S_j(\Psi_{x_0}(t',0))|})^{(1-\delta)j} |S_j(\Psi_{x_0}(t',0))|^{j+1}\, dt'\\
& \geq C' \eps^{-C'j}\alpha_1^{b_1}\alpha_2^{b_2}.
\end{align*}
For $C'$ sufficiently large, this gives a contradiction.  Thus on at least half of $\Omega_j$, $|J_j(x)| \lessapprox \alpha_j = \alpha_1$, so we may refine so that for each $x \in \Omega_{j-1}$, $|J_j(x)| \lessapprox \alpha_1$.  Repeating this argument for the maximal even index $j'$ with $\beta_{j'} > 0$, we may ensure that for each $x \in \Omega_{j'-1}$, $|J_{j'}(x)|\lessapprox \alpha_2$.  Finally, replacing $\Omega_n$ with $\Omega_{\min\{j,j'\}-1}$ and then refining, we can ensure that for $x_0 \in \Omega_0$, $1 \leq j \leq n$, and $t \in \scriptF_{j-1}$,
\begin{equation} \label{E:bound Jj}
|J_j(\Psi_{x_0}(t,0))| = |J(N,\{t_j \in \R : (t,t_j) \in \scriptF_j\})| \lessapprox \alpha_j.
\end{equation}
Refining further, we obtain a set $\Omega_{-n} \subseteq \Omega_0$, with $|\Omega_{-n}| \gtrsim |\Omega|$, such that for each $x_0 \in \Omega_{-n}$, there exists a parameter set 
$$
\scriptF_{x_0} \subseteq [-C'\eps^{-C'}\alpha_1,C'\eps^{-C'}\alpha_1] \times [-C'\eps^{-C'}\alpha_2,C'\eps^{-C'}\alpha_2] \times \cdots
$$
such that
\begin{equation} \label{E:big bite}
\begin{gathered}
\Psi_{x_0}(\scriptF_n) \subseteq \Omega_0 \cap B(x_0;C'\eps^{-C'}\alpha),\\
|\Psi_{x_0}(\scriptF_n)| \gtrapprox |B^n(x_0;C'\eps^{-C'}\alpha)|.
\end{gathered}
\end{equation}

We fix a point $x_0 \in \Omega_{-n}$ and a parameter set $\scriptF_{x_0}$ as above.  We add $x_0$ to $\scriptA$.  If (i) holds, we are done.  Otherwise, we apply the preceding to
$$
\Omega \setminus \bigcup_{x \in \scriptA} B^n(x;C'\eps^{-C'}\alpha),
$$
and find another point to add to $\scriptA$.  By \eqref{E:big bite} and $|\Omega| \lessapprox \alpha^b$, this process stops while $\#\scriptA \lessapprox 1$.  

This completes the proof of Lemma~\ref{L:Psi balls}, and thus of Proposition~\ref{P:quasiex} as well.
\end{proof}

%%%%%%%%%%%%%%%%%%%%%%%%%%%%%%%%%%%%%%%%
%%%%%%%%%%%%%%%%%%%%%%%%%%%%%%%%%%%%%%%%
%%%%%%%%%%%%%%%%%%%%%%%%%%%%%%%%%%%%%%%%

\section{Strong-type bounds on a single scale} \label{S:ST scale 1}

%%%%%%%%%%%%%%%%%%%%%%%%%%%%%%%%%%%%%%%%
%%%%%%%%%%%%%%%%%%%%%%%%%%%%%%%%%%%%%%%%
%%%%%%%%%%%%%%%%%%%%%%%%%%%%%%%%%%%%%%%%

This section is devoted to a proof of the following.

\begin{proposition} \label{P:local strong type}
$$
|\int_{\{\rho_\beta \sim 1\}} f_1 \circ \pi_1 \, f_2 \circ \pi_2 \, dx| \lesssim \|f_1\|_{p_1} \|f_2\|_{p_2}.
$$
\end{proposition}

\begin{proof}[Proof of Proposition~\ref{P:local strong type}]
It suffices to prove the proposition in the special case 
$$
f_i = \sum_k 2^k \chi_{E_i^k}, \qquad \|f_i\|_{p_i} \sim 1, \qquad i=1,2,
$$
with the $E_1^k$ pairwise disjoint, and likewise, the $E_2^k$.  Thus we want to bound
$$
\sum_{j,k} 2^{j+k} |\Omega^{j,k}|, \qquad \Omega^{j,k} := \{\rho_\beta \sim 1\} \cap \pi_1^{-1}(E_1^j) \cap \pi_2^{-1}(E_2^k).
$$
We know from Proposition~\ref{P:RWT} that 
$$
|\Omega^{j,k}| \lesssim |E_1^j|^{1/p_1} |E_2^k|^{1/p_2}.
$$
For $0 < \eps \lesssim 1$, we define
$$
\scriptL(\eps) := \{(j,k) : \tfrac12 \eps |E_1^j|^{1/p_1} |E_2^k|^{1/p_2} \leq |\Omega^{j,k}| \leq 2\eps |E_1^j|^{1/p_1} |E_2^k|^{1/p_2}\}.
$$
We additionally define for $0 < \eta_1,\eta_2 \leq 1$,
$$
\scriptL(\eps,\eta_1,\eta_2) := \{(j,k) \in \scriptL(\eps) : 2^{jp_1}|E_1^j| \sim \eta_1, \: 2^{kp_2}|E_2^k| \sim \eta_2\}.
$$

Let $\eps,\eta_1,\eta_2 \lesssim 1$ and let $(j,k) \in \scriptL(\eps,\eta_1,\eta_2)$.  Set 
$$
\alpha^{j,k} = (\alpha^{j,k}_1,\alpha^{j,k}_2) := \bigl(\tfrac{|\Omega^{j,k}|}{|E^j_1|}, \tfrac{|\Omega^{j,k}|}{|E^k_2|}\bigr).  
$$
Proposition~\ref{P:quasiex} guarantees the existence of a minimal $I \in \scriptW^n$ and a finite set $\scriptA^{j,k} \subseteq \Omega^{j,k}$ such that (i) and (ii) of that proposition (appropriately superscripted) hold.  (Since there are a bounded number of minimal $n$-tuples, we may assume in proving the proposition that all of these minimal $n$-tuples are the same.)  Set
\begin{equation} \label{E:def tilde Omega}
\tilde\Omega^{j,k} := \Omega^{j,k} \cap \bigcup_{x \in \scriptA^{j,k}} \bigcap_{\sigma \in S_n} B^{I_\sigma}(x,c\eps^C\alpha^{j,k}).
\end{equation}
Our main task in this section is to prove the following lemma.  

\begin{lemma} \label{L:exist disjoint}
Fix $\eps,\eta_1,\eta_2 \lesssim 1$ and set $\scriptL:= \scriptL(\eps,\eta_1,\eta_2)$.  Then
\begin{align}
\label{E:disjoint E1s}
\sum_{k : (j,k) \in \scriptL} |\pi_1(\tilde\Omega^{j,k})| \lesssim (\log \eps^{-1}) |E_1^j|, \qquad j \in \Z\\
\label{E:disjoint E2s}
\sum_{j : (j,k) \in \scriptL} |\pi_2(\tilde\Omega^{j,k})| \lesssim (\log \eps^{-1}) |E_2^k|, \qquad k \in \Z.
\end{align}
\end{lemma}

We assume Lemma~\ref{L:exist disjoint} for now and complete the proof of Proposition~\ref{P:local strong type}.  It suffices to show that for each $\eps,\eta_1,\eta_2$, if $\scriptL := \scriptL(\eps,\eta_1,\eta_2)$, then
\begin{equation} \label{E:sum on L}
\sum_{(j,k) \in \scriptL} 2^{j+k} |\Omega^{j,k}| \lesssim \eps^{a_0} \eta_1^{a_1}\eta_2^{a_2},
\end{equation}
with each $a_i$ positive.  Indeed, once we have proved the preceding inequality, we can just sum on dyadic values of $\eps,\eta_1,\eta_2$. 

We turn to the proof of \eqref{E:sum on L}.  It is a triviality that $\#\scriptL(\eps,\eta_1,\eta_2) \lesssim \eta_1^{-1}\eta_2^{-1}$, so
\begin{equation} \label{E:pos eps neg eta}
\begin{aligned}
&\sum_{(j,k) \in \scriptL} 2^{j+k} |\Omega^{j,k}| \sim \eps \sum_{(j,k) \in \scriptL} 2^{j+k} |E_1|^{1/p_1} |E_2|^{1/p_2} \\
&\qquad \sim \eps (\#\scriptL) \eta_1^{1/p_1}\eta_2^{1/p_2} \lesssim \eps \eta_1^{-1/p_1'}\eta_2^{-1/p_2'}.
\end{aligned}
\end{equation}
Define
$$
q_i := (p_1^{-1}+p_2^{-1})p_i, \qquad i=1,2,
$$
then since
$$
p_1^{-1}+p_2^{-1} = \frac{b_1+b_2}{b_1+b_2-1} > 1,
$$
we have $q_i > p_i$, $i=1,2$, and $q_1 = q_2'$.  Applying Lemma~\ref{L:exist disjoint},
\begin{align*}
\sum_{(j,k) \in \scriptL} 2^{j+k} |\Omega^{j,k}| &\sim \sum_{(j,k) \in \scriptL} 2^{j+k} |\tilde\Omega^{j,k}| \lesssim \sum_{(j,k) \in \scriptL} 2^{j+k} |\pi_1(\tilde\Omega^{j,k})|^{1/p_1}|\pi_2(\tilde\Omega^{j,k})|^{1/p_2} \\
&\lesssim \bigl(\sum_{(j,k) \in \scriptL} 2^{jq_1} |\pi_1(\tilde\Omega^{j,k})|^{q_1/p_1}\bigr)^{1/q_1} \bigl(\sum_{(j,k) \in \scriptL} 2^{kq_2} |\pi_2(\tilde\Omega^{j,k})|^{q_2/p_2}\bigr)^{1/q_2}\\
&\lesssim \eta_1^{1/p_1-1/q_1}\eta_2^{1/p_2-1/q_2} \\
&\qquad\qquad \times \bigl(\sum_{(j,k) \in \scriptL} 2^{jp_1} |\pi_1(\tilde\Omega^{j,k})|\bigr)^{1/q_1} \bigl(\sum_{(j,k) \in \scriptL} 2^{kp_2} |\pi_2(\tilde\Omega^{j,k})|\bigr)^{1/q_2}\\
&\lesssim \log\eps^{-1} \eta_1^{1/p_1-1/q_1}\eta_2^{1/p_2-1/q_2} \bigl(\sum_j 2^{jp_1} |E_1^j|\bigr)^{1/q_1} \bigl(\sum_k 2^{kp_2} |E_2^k|\bigr)^{1/q_2}\\
&\lesssim \log\eps^{-1} \eta_1^{1/p_1-1/q_1}\eta_2^{1/p_2-1/q_2}.
\end{align*}
Combining this estimate with \eqref{E:pos eps neg eta} gives \eqref{E:sum on L}, completing the proof of Proposition~\ref{P:local strong type}, conditional on Lemma~\ref{L:exist disjoint}.
\end{proof}

We turn to the proof of Lemma~\ref{L:exist disjoint}.  We will only prove \eqref{E:disjoint E1s}, and we will take care that our argument can be adapted to prove \eqref{E:disjoint E2s} by interchanging the indices.  (The roles of $\pi_1$ and $\pi_2$ are not \textit{a priori} symmetric, because their roles in defining the weight $\rho$ are not symmetric.)  The argument is somewhat long and technical, so we start with a broad overview.  

Assume that \eqref{E:disjoint E1s} fails.  By Proposition~\ref{P:quasiex}, the $\tilde\Omega^{j,k}$ can be well approximated as the images of ellipsoids (the $Q^I_{\alpha^{j,k}}$) under polynomials of bounded degree (the $\Phi^I_{x^{j,k}}$).  The definition of $\scriptL$ ensures that the $\alpha^{j,k}$, and hence the radii of these ellipsoids, live at many different dyadic scales (this is where the minimality condition in Proposition~\ref{P:quasiex} will be used).  On the other hand, the projections $\pi_1(\tilde\Omega^{j,k})$ must have a large degree of overlap (otherwise, the volume of the union would bound the sum of the volumes).  In particular, we can find a large number of $\tilde\Omega^{j,k}$ that all have essentially the same projection.  These $\tilde\Omega^{j,k}$ all lie along a single integral curve of $X_1$.  The shapes of the $\tilde\Omega^{j,k}$ are determined by widely disparate parameters, the $\alpha^{j,k}$, and polynomials, the $\Phi^I_{x^{j,k}}$.  We can take $x^{j,k} = e^{t^{j,k}X_1}(x_0)$, for a fixed $x_0$, and we use the condition that the projections are all essentially the same to prove that there exists an associated polynomial $\gamma:\R \to \R^n$ that is transverse to its derivative $\gamma'$ at more scales than Lemma~\ref{L:gamma || gamma'} allows.  

We begin by making precise the assertion that many $\tilde\Omega^{j,k}$ must have essentially the same projection.  The main step is an elementary lemma.

\begin{lemma} \label{L:intersect M}
Let $\{E^k\}$ be a collection of measurable sets, and define $E:=\bigcup_k E^k$.  Then for each integer $M \geq 1$,
\begin{equation} \label{E:intersect M}
\sum_k |E^k| \lesssim_M |E| + |E|^{\frac{M-1}M}\bigl(\sum_{k_1 < \cdots < k_M} |E^{k_1} \cap \cdots \cap E^{k_M}|\bigr)^{\frac1M}.
\end{equation}
\end{lemma}

\begin{proof}[Proof of Lemma~\ref{L:intersect M}]
We review the argument in the case $M=2$, which amounts to a rephrasing of an argument from \cite{ChQex}.  By Cauchy--Schwarz,
\begin{align*}
\sum_k |E^k| &= \int_E \sum_k \chi_{E_k} 
\leq |E|^{\frac12}\bigl(\int_E |\sum_k \chi_{E_k}|^2\bigr)^{\frac12} \\
&= |E|^{\frac12}\bigl(\sum_k |E_k| + 2\sum_{k_1 < k_2} |E_{k_1} \cap E_{k_2}|\bigr)^{\frac12}\\
&\leq \tfrac12 \sum_k |E_k| +\tfrac12 |E| + 2|E|^{\frac12}\bigl(\sum_{k_1 < k_2} |E^{k_1} \cap E^{k_2}|\bigr)^{\frac12},
\end{align*}
and inequality \eqref{E:intersect M} follows by subtracting $\tfrac12 \sum_k |E^k|$ from both sides.  

Now to the case of larger $M$.  Arguing analogously to the $k=2$ case implies that
\begin{equation} \label{E:intersect less M}
\sum_k |E^k| \lesssim_M |E|^{\frac{M-1}M}\bigl(\sum_{i=1}^M \sum_{k_1 < \ldots < k_i} |\bigcap_{l=1}^i E^{k_l}|\bigr)^{\frac1M}.
\end{equation}
Suppose that \eqref{E:intersect M} is proved for $2,\ldots,M-1$.  Let $1 < i < M$.  For fixed $k_1 < \cdots < k_{i-1}$, 
\begin{align*}
&\sum_{k_i} |E^{k_1} \cap \cdots \cap E^{k_i}|
 \lesssim_M |E^{k_1} \cap \cdots \cap E^{k_{i-1}}| \\
 &\qquad\qquad + |E^{k_1} \cap \cdots \cap E^{k_{i-1}}|^{\frac{M-i-1}{M-i}}\bigl( \sum_{k_i < \cdots < k_M}|E^{k_1} \cap \cdots \cap E^{k_M}|\bigr)^{\frac1{M-i-1}}\\
& \lesssim_M |E^{k_1} \cap \cdots \cap E^{k_{i-1}}| + \sum_{k_i < \cdots < k_M}|E^{k_1} \cap \cdots \cap E^{k_M}|.
\end{align*}
By induction and \eqref{E:intersect less M},
$$
\sum_k |E^k| \lesssim_M |E|^{\frac{M-1}M} \bigl(\sum_k |E^k| + \sum_{k_1 < \cdots < k_M} |E^{k_1} \cap \cdots \cap E^{k_M}|\bigr)^{\frac1M},
$$
which implies \eqref{E:intersect M}.
\end{proof}

Our next goal is to reduce the proof of Lemma~\ref{L:exist disjoint}, specifically, the proof of \eqref{E:disjoint E1s} to the following.

\begin{lemma} \label{L:M projections are disjoint}
For $M>M(N)$ sufficiently large and each $A > 0$, there exists $B > 0$ such that for all $0 < \delta \leq \eps$, if $j_0 \in \Z$ and $\scriptK \subseteq \Z$ is a $(B \log \delta^{-1})$-separated set with cardinality $\#\scriptK \geq M$ and $\{j_0\} \times \scriptK \subseteq \scriptL$, then
\begin{equation} \label{E:M projections are disj}
|\bigcap_{k \in \scriptK} \pi_1(\tilde\Omega^{j_0,k})| < A^{-1}\delta^A 2^{-j_0 p_1} \eta_1.
\end{equation}
\end{lemma}

\begin{proof}[Proof of Lemma~\ref{L:exist disjoint}, conditional on Lemma~\ref{L:M projections are disjoint}]
We will only prove inequality \eqref{E:disjoint E1s}.  The obvious analogue of Lemma~\ref{L:M projections are disjoint}, which has the same proof as Lemma~\ref{L:M projections are disjoint}, implies inequality \eqref{E:disjoint E2s}.  

Fix $M=M(N)$ sufficiently large to satisfy the hypotheses of Lemma~\ref{L:M projections are disjoint} and fix $A > Mp_1$.  Now fix $B=B(M,N,A)$ as in the conclusion of Lemma~\ref{L:M projections are disjoint}.  Let $\delta := \min\{\delta_0,\eps\}$, with $\delta_0$ to be determined, and let $\scriptK_0 \subseteq \Z$ be a finite $(B \log\delta^{-1})$-separated set with $\{j_0\} \times \scriptK_0 \subseteq \scriptL$.  By Lemma~\ref{L:intersect M}, Lemma~\ref{L:M projections are disjoint}, then the approximation $\binom{\#\scriptK_0}{M} \sim_M (\#\scriptK_0)^M$ and the definition of $\scriptL$,
\begin{equation} \label{E:bound sum pi1(Omegak) intersection}
\begin{aligned}
\sum_{k \in \scriptK_0}|\pi_1(\tilde\Omega^{j_0,k})| &\lesssim_M |E_1^{j_0}| + |E^{j_0}_1|^{\frac{M-1}M}(\sum_{\scriptK \subseteq \scriptK_0; \#\scriptK=M} |\bigcap_{k \in \scriptK} \pi_1(\tilde\Omega^{j_0k})|)^{\frac1M}\\
&\lesssim_M |E^{j_0}_1| + |E^{j_0}_1|^{\frac{M-1}M} \#\scriptK_0 (A^{-1}\delta^A |E^{j_0}_1|)^{\frac1M}.
\end{aligned}
\end{equation}
Quasiextremality and the restricted weak type inequality give
$$
\delta |E^{j_0}_1|^{\frac1{p_1}}|E^k_2|^{\frac1{p_2}} \lesssim |\Omega^{j_0k}| \sim |\tilde\Omega^{j_0k}| \lesssim |\pi_1(\tilde\Omega^{j_0k})|^{\frac1{p_1}}|E^k_2|^{\frac1{p_2}}, \quad k \in \scriptK_0
$$
whence 
\begin{equation} \label{E:sum pi1(Omegak) big}
\sum_{k \in \scriptK_0} |\pi_1(\tilde\Omega^{j_0k})| \gtrsim \#\scriptK_0 \delta^{p_1}|E^{j_0}_1|.
\end{equation}
For $\delta_0 = \delta_0(p_1,A,M)$ sufficiently small, $\delta^{p_1} > C_M (A^{-1}\delta^A)^{\frac1M}$, with $C_M$ as large as we like, so inserting \eqref{E:sum pi1(Omegak) big} into \eqref{E:bound sum pi1(Omegak) intersection} implies
\begin{equation} \label{E:sum pi1(Omegak) small}
\sum_{k \in \scriptK_0}|\pi_1(\tilde\Omega^{j_0,k})| \lesssim_M |E^{j_0}_1|.
\end{equation}
Since $\scriptK_0$ was arbitrary and $p_1,M,A,B$ all ultimately depend only on $N$ alone, \eqref{E:sum pi1(Omegak) small} implies \eqref{E:disjoint E1s}.
\end{proof}

It remains to prove Lemma~\ref{L:M projections are disjoint}.  

\begin{lemma} \label{L:M balls}
For $M = M(N)$ sufficiently large and each $A > 0$, there exists $B > 0$ such that the following holds for all $0 < \delta \leq \eps$.  Fix $j_0 \in \Z$ and let $\scriptK \subseteq \Z$ be a $(B\log \delta^{-1})$-separated set with cardinality $\#\scriptK = M$ and $\{j_0\} \times \scriptK \subseteq \scriptL$.  Let $x^{j_0k} \in \tilde\Omega^{j_0k}$, $k \in \scriptK$.  Then
\begin{equation} \label{E:M balls}
|\bigcap_{k \in \scriptK} \pi_1(\bigcap_{\sigma \in S_n} B^{I_\sigma}(x^{j_0,k},c\delta^C \alpha^{j_0,k}))| < A^{-1}\delta^A 2^{-j_0p_1}\eta_1.
\end{equation}
\end{lemma}

We note that once the lemma holds for $M=M(N)$, it immediately holds for all $M>M(N)$ as well.  

\begin{proof}[Proof of Lemma~\ref{L:M projections are disjoint}, conditional on Lemma~\ref{L:M balls}]
By definition \eqref{E:def tilde Omega}, each $\tilde\Omega^{j_0k}$ is covered by $C\eps^{-C}$ balls of the form $\bigcap_{\sigma \in S_n} B^{I_\sigma}(x,c\eps^C \alpha^{j_0,k})$; in fact, by the proof of Proposition~\ref{P:quasiex}, it is also covered by $C\delta^{-C}$ balls $\bigcap_{\sigma \in S_n} B^{I_\sigma}(x,c\delta^C \alpha^{j_0,k})$, for each $0 < \delta \leq \eps$.  Thus  $\bigcap_{k \in \scriptK} \pi_1(\tilde\Omega^{j_0,k})$ is covered by $(C\delta^{-C})^M$ $M$-fold intersections of projections of balls, so \eqref{E:M balls} (with a larger value of $A$) implies \eqref{E:M projections are disj}.
\end{proof}

The remainder of the section will be devoted to the proof of Lemma~\ref{L:M balls}.  We will give the proof when $\delta = \eps$; since an $\eps$-quasiextremal $\Omega^{j_0k}$ is also $\delta$-quasiextremal for every $0 < \delta < \eps$, all of our arguments below apply equally well in the case $\delta < \eps$.  (We recall that allowing the more general parameter $\delta$ instead of $\eps$ gave us slightly more technical flexibility in the proof of Lemma~\ref{L:exist disjoint} from Lemma~\ref{L:M projections are disjoint}.)

The potential failure of $\pi_1$ to be a polynomial presents a technical complication.  (Coordinate changes are not an option in the non-minimal case.)  By reordering the words in $I=(w_1,\ldots,w_n)$, we may assume that $w_n=(1)$.  Fix $k_0 \in \scriptK$, and set $x_0 = x^{j_0k_0}$.  We define a ``cylinder''
$$
\scriptC:= \Phi_{x_0}^I(U), \qquad U:=\{(t',t_n) : (t',0) \in Q^I_{c\eps^C\alpha^{j_0k_0}}\}.
$$
Set $U_0 := \{(t',0) \in U\}$ and define
\begin{align*}
U_+ &:= \{t \in U : t_n > 0 \ctc{and, for all} 0 < s \leq t_n,\: \Phi_{x_0}^I(t',s) \notin \Phi_{x_0}^I(\overline U_0)\},\\
U_- &:= \{t \in U : t_n < 0 \ctc{and, for all} 0 > s \geq t_n,\: \Phi_{x_0}^I(t',s) \notin \Phi_{x_0}^I(\overline U_0)\},
\end{align*}
and $\scriptC_0 := \Phi_{x_0}^I(\overline U_0)$, $\scriptC_{\pm}:=\Phi_{x_0}^I(U_{\pm})$.  

\begin{lemma} \label{L:scriptC}
The map $\Phi_{x_0}^I$ is nonsingular, with 
$$
|\det D\Phi_{x_0}^I| \approx \lambda_I(x_0) \approx (\alpha^{j_0k_0})^{b-\deg I},
$$
on $U$.  The sets $U_\pm$ are open, and $\Phi_{x_0}^I$ is one-to-one on each of them.  Finally, 
$$
\scriptC \subseteq \scriptC_+ \cup \scriptC_- \cup \scriptC_0 \cup \scriptC_\partial,
$$
where $\scriptC_\partial:=\Phi^I_{x_0}(\partial U)$.   
\end{lemma}

\begin{proof}[Proof of Lemma~\ref{L:scriptC}]
Since $X_1$ is divergence-free, 
$$
\det D\Phi_{x_0}^I(t) = \det D\Phi_{x_0}^I(t',0),
$$
and thus the conclusions about the size of this Jacobian determinant follow from Proposition~\ref{P:quasiex}.

By conclusion (ii) of Proposition~\ref{P:quasiex} and continuity of $\Phi_{x_0}^I$, we see that $U_\pm$ is an open set containing
$$
\{(t',t_n) : (t',0) \in Q_{c\eps^C\alpha^{j_0k_0}}, 0 < \pm t_n < c\eps^C\alpha^{j_0k_0}_1\}.
$$

Suppose $t,u \in U_+$, $\Phi_{x_0}^I(t) = \Phi_{x_0}^I(u)$, and $u_n \leq t_n$.  Then $\Phi_{x_0}^I(t',t_n-u_n) = \Phi_{x_0}^I(u',0)$.  If $t_n=u_n$, then $t=u$, because $\Phi_{x_0}^I$ is one-to-one on $Q_{c\eps^C\alpha^{j_0k_0}} \ni (u',0),(t',0)$.  Otherwise, $(t',t_n-u_n) \in U_+$, so $\Phi_{x_0}^I(t',t_n-u_n) = \Phi_{x_0}^I(u',0)$ is impossible.  Thus $\Phi_{x_0}^I$ is indeed one-to-one on $U_\pm$.  

Finally, let $t \in U$, with $t_n > 0$ and $t \notin U_+$.  We need to show that $\Phi_{x_0}^I(t) \in \scriptC_+ \cup \scriptC_0 \cup \scriptC_\partial$.  The curve $\{\Phi_{x_0}^I(t',s):s \in \R\}$ intersects $\scriptC_0$ a bounded number of times, so, by the definition of $U_+$, there exists some maximal $0 < s \leq t_n$ such that $\Phi_{x_0}^I(t',s) = \Phi_{x_0}^I(u',0)$, for some $(u',0) \in \overline U_0$.  Thus $\Phi_{x_0}^I(t',t_n) = \Phi_{x_0}^I(u',t_n-s)$.  If $s=t_n$, $(u',0) \in \partial U$, or $(u',t_n-s) \in U_+$, we are done.  Otherwise, there exists some $0 < r < t_n-s$ and $(v',0) \in \overline U_0$ such that $\Phi_{x_0}^I(u',r) = \Phi_{x_0}^I(v',0)$, whence $\Phi_{x_0}(t',s+r) = \Phi_{x_0}^I(v',0)$, contradicting maximality of $s$.  
\end{proof}

On $U_{\pm}$, $\Phi_{x_0}^I$ has a smooth inverse, and we define a map $\tilde\pi_1$ on $\scriptC_{\pm}$ by
$$
\tilde\pi_1:=((\alpha^{j_0k_0})^{-\deg w_1} (\Phi_{x_0}^I)^{-1}_1,\ldots,(\alpha^{j_0k_0})^{-\deg w_{n-1}}(\Phi_{x_0}^I)^{-1}_{n-1}).
$$

\begin{lemma}\label{L:tilde pi}
Define 
$$
F(t'):=\pi_1 \circ \Phi_{x_0}^I \bigl((c\eps^C \alpha^{j_0k_0})^{\deg w_1}t_1,\cdots,(c\eps^C \alpha^{j_0k_0})^{\deg w_{n-1}}t_{n-1},0\bigr), \qquad |t'|<1.
$$
Then $F$ is a bounded-to-one local diffeomorphism satisfying $|\det DF| \approx |E^{j_0}_1|$ and $\pi_1|_{\scriptC_\pm} = F \circ \tilde\pi_1|_{\scriptC_\pm}$.  
\end{lemma}

\begin{proof}[Proof of Lemma~\ref{L:tilde pi}]
Let 
$$
Q_{c\eps^C\alpha^{j_0k_0}}^\flat:=\{t':(t',0) \in Q_{c\eps^C\alpha^{j_0k_0}}\}.
$$
We begin by proving that $t' \mapsto \pi_1 \circ \Phi_{x_0}^I(t',0)$ is bounded-to-one on $Q^\flat_{c\eps^C\alpha^{j_0k_0}}$.  
By the implicit function theorem, the definition of $X_1$, and hypothesis (ii) of Theorem~\ref{T:main}, for every $y \in \R^{n-1}$, $\pi_1^{-1}(y)$ intersects at most one nonconstant integral curve of $X_1$.  Therefore, since $X_1$ is nonvanishing on $\scriptC$, if $\pi_1 \circ \Phi_{x_0}^I(u',0) = \pi_1\circ\Phi_{x_0}^I(t',0)$ for some $u' \neq t'$, we may assume that $\Phi_{x_0}^I(t',0) = \Phi_{x_0}^I(u',u_n)$ for some $u_n>0$.  By Lemma~\ref{L:finite to one} and $\det D\Phi_{x_0}^I \neq 0$ on $U$, given $t' \in Q_{c\eps^C\alpha^{j_0k_0}}^\flat$, there are only a bounded number of such $u'$.  

Our definition of $\tilde\pi_1$ implies that $\pi_1 = F \circ \tilde\pi_1$ on $\scriptC_0$, and hence on all of $\scriptC$ (since both sides are constant on $X_1$'s integral curves).  

Let $A \subseteq \{|t'| < 1\}$.  Then 
$$
B:= \tilde\pi_1^{-1}(A) \cap \{\Phi^I_{x_0}(t):t \in U, \: |t_n|<c\eps^C\alpha^{j_0k_0}_1\}
$$
equals the image
$$
\{\Phi_{x^0}^I(t) : ((c\eps^C \alpha^{j_0k_0})^{-\deg w_1}t_1,\ldots,(c\eps^C\alpha^{j_0k_0})^{-\deg w_{n-1}}t_n) \in A, \: |t_n| < c\eps^C\alpha^{j_0k_0}_1\},
$$
and hence, by Proposition~\ref{P:quasiex}, has volume 
\begin{equation} \label{E:vol B star}
|B| \approx (\alpha^{j_0k_0})^{\deg I} |\lambda_I(x^0)||A| \approx |\Omega^{j_0k_0}||A|.
\end{equation}
By the definition of $B$, the coarea formula, \eqref{E:vol B star}, and the definition of $\alpha^{j_0k_0}$,
$$
|F(A)| = |\pi_1(B)| \approx (\alpha^{j_0k_0}_1)^{-1}|B| \approx (\alpha^{j_0k_0}_1)^{-1}|\Omega^{j_0k_0}||A| = |E^{j_0}_1||A|.
$$
The estimate on the Jacobian determinant of $F$ follows from the change of variables formula. 
\end{proof}

The next lemma allows us to replace $\scriptC$ with the domain of $\tilde \pi_1$.

\begin{lemma}\label{L:U plus}
If \eqref{E:M balls} fails for some $M,A,B,\delta=\eps>0,j_0,\scriptK,\{x^{j_0k}\}_{k \in \scriptK}$ satisfying the hypothesis of Lemma~\ref{L:M balls}, then there exists $\scriptK' \subseteq \scriptK$, of cardinality $\#\scriptK' \sim M$, such that
\begin{equation} \label{E:U plus}
|\bigcap_{k \in \scriptK'} \tilde\pi_1(\scriptC_+ \cap \bigcap_{\sigma \in S_n} B^{I_\sigma}(x^{j_0,k},c\eps^C \alpha^{j_0k}))| \geq A^{-1}\eps^A,
\end{equation}
or such that \eqref{E:U plus} holds with `$-$' in place of `$+$.'  Here the quantity $A$ depends on the corresponding quantity in Lemma~\ref{L:M balls} and $N$.  
\end{lemma}

\begin{proof}[Proof of Lemma~\ref{L:U plus}]
For $k \in \scriptK$, set 
$$
B^k := \bigcap_{\sigma \in S_n} B^{I_\sigma}(x^{j_0k};c\eps^C\alpha^{j_0k}).
$$

Since $\pi_1(B^{k_0}) \subseteq \pi_1(B^I(x,c\eps^C\alpha^{j_0k_0}))$, our hypothesis that $\pi_1$ fibers lie on a single integral curve of $X_1$ implies that $\bigcap_{k\in\scriptK} \pi_1(B^k) = \bigcap_{k \in \scriptK} \pi_1(\scriptC \cap B^k)$.  
The projection $\pi_1(\scriptC_\partial)$ has measure zero.  For a.e. $y \in \bigcap_{k\in\scriptK} \pi_1(B^k)$, $\pi_1^{-1}(y) \cap B^k = \Phi_{x^0}^I(t_0',J)$ for some set $J \subseteq \R$ having positive measure; thus, $|\bigcap_{k\in\scriptK} \pi_1(B^k)| = |\bigcap_{k \in \scriptK} \pi_1(B^k \setminus \scriptC_0)|$.  Putting these two observations together with Lemma~\ref{L:scriptC} and using standard set manipulations,
$$
|\bigcap_{k \in \scriptK} \pi_1(\scriptC \cap B^k)| = |\bigcup_{\bullet \in \{+,-\}^\scriptK} \bigcap_{k \in \scriptK} \pi_1(\scriptC_{\bullet_k} \cap B^k)|.
$$
Thus if \eqref{E:M balls} fails, there exists a decomposition $\scriptK = \scriptK_+ \cup \scriptK_-$ such that
$$
\min\{|\bigcap_{k \in \scriptK_+}\pi_1(\scriptC_+ \cap B^k)|,|\bigcap_{k \in \scriptK_-}\pi_1(\scriptC_- \cap B^k)|\} > A^{-1}\eps^A 2^{-j_0 p_1}\eta_1.
$$
One of $\scriptK_+,\scriptK_-$ must have cardinality $\#\scriptK_\bullet \gtrsim M$; we may assume that the larger is $\scriptK_+=:\scriptK'$.  Inequality \eqref{E:U plus} then follows from Lemma~\ref{L:tilde pi} and the definition of $\scriptL$.  
\end{proof}

The next lemma verifies that a slightly enlarged version of each $B^k$ has large intersection with $\scriptC_+$.    

\begin{lemma}\label{L:Gk}
Assume that \eqref{E:U plus} holds, and let $k \in \scriptK'$, $y^{j_0k} \in B^k$, with $B^k$ defined as in the proof of Lemma~\ref{L:U plus}.  Set 
$$
G^k:=\scriptC_+ \cap \bigcap_{\sigma \in S_n} B^{I_\sigma}(y^{j_0k};c'\eps^{C'}\alpha^{j_0k}).
$$
Then $|G^k| \gtrapprox A^{-1}\eps^A|B^k|$ and $B^k \cap \scriptC_+ \subseteq G^k$.  
\end{lemma}

\begin{proof}[Proof of Lemma~\ref{L:Gk}]
By conclusion (ii) of Proposition~\ref{P:quasiex}, $B^k\cap \scriptC_+ \subseteq G^k$.  Let $x \in B^k \cap \scriptC_+$.  So long as $x \notin \scriptC_\partial$ (which has measure zero), $e^{tX_1}(x) \in \scriptC_+$ for all except finitely many positive values of $t$ (i.e.\ except for those $t$ for which $e^{tX_1}(x) \in \scriptC_0$).  Additionally,
$$
e^{tX_1}(x) \in \bigcap_{\sigma \in S_n} B^{I_\sigma}(y^{j_0k},c'\eps^{C'}\alpha^{j_0k}),
$$
for all $|t|<c\eps^C\alpha^{j_0k}$, so $e^{tX_1}(x) \in G^k$ for $t$ in a set of measure $\gtrapprox \alpha^{j_0k}$.  By the coarea formula, then Lemma~\ref{L:tilde pi} and \eqref{E:U plus}, the definition of $\alpha^{j_0k}$, and finally Proposition~\ref{P:quasiex},
$$
|G^k| \gtrapprox \alpha^{j_0k}_1|\pi_1(B^k \cap \scriptC_+)| \approx A^{-1}\eps^A \alpha_1^{j_0k}|E^{j_0}_1| \approx A^{-1}\eps^A|\Omega^{j_0k}| \approx A^{-1}\eps^A |B^k|.
$$
\end{proof}

To motivate the next lemma, we recall that our goal is to show that a certain inequality holds at many points of the form $e^{t^{jk}X_1}(x_0) \in B^k$.  This will be possible because the set of $y^{jk} \in B^k$ at which the inequality fails must be very small, and hence have small projection.  

\begin{lemma}\label{L:tilde Gk}
Under the hypotheses and notation of Lemma~\ref{L:Gk}, there exists a subset $\tilde G^k \subseteq G^k$ such that $|G^k \setminus \tilde G^k| < D^{-1}\eps^D|G^k|$, with $D=D(N,A)$ sufficiently large for later purposes, such that for all $x \in \tilde G^k$, 
\begin{equation} \label{E:Xw lesssim 1}
|D\tilde\pi_1(x) (\alpha^{j_0k})^{\deg w_i}X_{w_i}(x)| \lessapprox_A 1, \qquad 1 \leq i \leq n-1.
\end{equation}
\end{lemma}

\begin{proof}[Proof of Lemma~\ref{L:tilde Gk}]
To simplify our notation somewhat, we will say that a subset $\tilde G^k \subseteq G^k$ constitutes the \textit{vast majority} of $G^k$ if $|G^k \setminus \tilde G^k| < D^{-1}\eps^D|G^k|$, with $D=D(N,A)$ as small as we like.  

Taking intersections, it suffices to establish the lemma for a single index $1 \leq i \leq n-1$.  We recall that $w_i \neq (1)$.  Fix a permutation $\sigma \in S_n$ such that $\sigma(n) = i$.  By construction, $G^k \subseteq B^{I_\sigma}(x^{j_0k},c'\eps^{C'}\alpha^{j_0k})$.  By Lemma~\ref{L:Gk}, 
$$
|B^{I_\sigma}(x^{j_0k},c'\eps^{C'}\alpha^{j_0k}) \cap \scriptC_+| \gtrapprox_A |B^{I_\sigma}(x^{j_0k};c'\eps^{C'}\alpha^{j_0k})|,
$$
so our Jacobian bound, $|\det D\Phi_{x^{j_0k}}^{I_\sigma}| \approx |\lambda_I(x^{j_0k})|$ on $Q^{I_\sigma}_{c'\eps^{C'}\alpha^{j_0k}}$, implies that for the vast majority of points $x \in G^k$, $e^{tX_w}(x) \in \scriptC_+$ for all $t \in E_x$, $E_x$ some set of measure $|E_x| \gtrapprox_A (\alpha^{j_0k})^{\deg w_i}$.  

By Lemmas~\ref{L:intersect cylinder} and~\ref{L:proj gamma}, $E_x$ can be written as a union of a bounded number of intervals on which each component of $\tfrac{d}{dt} \tilde\pi_1(e^{tX_{w_i}}(x))$ is single signed.  Thus, using the semigroup property of exponentiation, we see that for the vast majority of $x \in G^k$, there exists an interval $J_x \ni 0$, of length $|J_x| \gtrapprox_A (\alpha^{j_0k})^{\deg w_i}$, such that $e^{tX_{w_i}}(x) \in \scriptC_+$ and the components of $\tfrac{d}{dt} \tilde\pi_1(e^{tX_{w_i}}(x))$ do not change sign on $J_x$.  

Let $x \in G^k$ be one of these majority points.  By the Fundamental Theorem of Calculus and
$$
\tilde\pi_1(e^{tX_{w_i}}(x)) \subseteq \tilde\pi_1(\scriptC_+) \subseteq \{u \in \R^{n-1}:|u| < 1\}, \qquad t \in J_x,
$$
combined with the above non-sign-changing condition,
$$
\int_{J_x}|\tfrac{d}{dt} \tilde\pi_1(e^{tX_{w_i}}(x))|\, dt \sim |\int_{J_x}\tfrac{d}{dt} \tilde\pi_1(e^{tX_{w_i}}(x))\, dt| < 1.
$$
Thus on the vast majority of $J_x$, 
$$
|\tfrac{d}{dt} \tilde\pi_1(e^{tX_{w_i}}(x))| \lessapprox_A |J_x|^{-1} \lessapprox_A (\alpha^{j_0k})^{-\deg w_i}.
$$
The conclusion of the lemma follows from the Chain Rule and our Jacobian estimate on $\Phi_{x^{j_0k}}^{I_\sigma}$.  
\end{proof}

Finally, we come to the main step in deriving the promised contradiction.  

\begin{lemma}  \label{L:wedges approx 1}
Under the hypotheses and notation of Lemma~\ref{L:tilde Gk}, there exist a point $y^0 \in B^{k_0}$ and times $t^{j_0k} \in \R$, $k \in \scriptK'$ such that for any $1 \leq j \leq n$ and any choice of $1 \leq i_1 < \cdots < i_j \leq n-1$ and $k \in \scriptK'$,
\begin{equation} \label{E:wedges approx 1}
|\bigwedge_{l=1}^j D\tilde\pi_1(e^{t^{j_0k}X_1}(y^0)) (\alpha^{j_0k})^{\deg w_{i_l}}X_{w_{i_l}}(e^{t^{j_0k}X_1}(y^0))| \approx_A 1.
\end{equation}
\end{lemma}

\begin{proof}[Proof of Lemma~\ref{L:wedges approx 1}]
Let $\tilde G^k$ be as in Lemma~\ref{L:tilde Gk}.  For $k \in \scriptK'$, and $D = D(N,A)$ as large as we like,
$$
|\tilde \pi_1(G^k \setminus \tilde G^k)| \lessapprox (\alpha^{j_0k}_1)^{-1}|G^k \setminus \tilde G^k| < D^{-1}\eps^D (\alpha^{j_0k}_1)^{-1} |G^k| \lessapprox D^{-1}\eps^D.
$$
Thus if $D = D(N,A)$ is sufficiently large, 
$$
|\bigcup_{k \in \scriptK'} \tilde\pi_1(G^k \setminus \tilde G^k)| < \tfrac12 A^{-1}\eps^A,
$$
so $\bigcap_{k \in \scriptK'} \tilde\pi_1(\tilde G^k)$ is nonempty.  Thus there exists a point $y^0 \in \tilde G^{k_0}$ and times $\{t^{j_0k}\}$ such that $e^{t^{j_0k}X_1}(y^0) \in \tilde G^k$, $k \in \scriptK'$.  We may assume that $t^{j_0k_0}=0$, and we set $y^{k}:= e^{t^{j_0k}X_1}(y^0)$.  

By the Chain Rule and basic linear algebra, and then our Jacobian estimate on $\det D\Phi_{x^0}^I$,
\begin{align*}
&|\bigwedge_{j=1}^{n-1} D\tilde\pi_1(y^k)X_{w_j}(y^k)| \\
&\qquad\qquad = \alpha_1^{j_0k_0}(\alpha^{j_0k_0})^{-\deg I}|\det D(\Phi_{x^0}^I)^{-1}(y^k)||\det(X_{w_1}(y^k),\ldots,X_{w_n}(y^k))|\\
&\qquad\qquad\sim \alpha_1^{j_0k_0}(\alpha^{j_0k_0})^{-\deg I} \frac{|\lambda_I(y^k)|}{|\lambda_I(x^0)|}.
\end{align*}
By (ii) of Proposition~\ref{P:quasiex} and the definition of $\alpha^{jk}$, 
$$
(\alpha^{j_0k})^{\deg I} |\lambda_I(y^k)| \approx (\alpha^{j_0k})^b \approx |\Omega^{j_0k}| = \alpha^{j_0k}_1|E^{j_0}_1|,
$$
for all $k$, so
\begin{align*}
\alpha_1^{j_0k_0}(\alpha^{j_0k_0})^{-\deg I} \frac{|\lambda_I(y^k)|}{|\lambda_I(x^0)|} \approx \alpha_1^{j_0k}(\alpha^{j_0k})^{-\deg I}.
\end{align*}
Putting these inequalities together
$$
|\bigwedge_{i=1}^{n-1} D\tilde\pi_1(y^k) (\alpha^{j_0k})^{\deg w_i}X_{w_i}(y^k)| \approx 1,
$$
and by \eqref{E:Xw lesssim 1}, this is possible only if \eqref{E:wedges approx 1} holds.  
\end{proof}

Finally, we are ready to complete the proof of Lemma~\ref{L:M balls}.  

\begin{proof}[Proof of Lemma~\ref{L:M balls}]
Let 
$$
\gamma(t) := D\tilde\pi_1(e^{tX_1}(y^0))X_2(e^{tX_1}(y_0)) = D\tilde\pi_1(y^0)De^{-tX_1}(e^{tX_1}(y^0)) X_2(e^{tX_1}(y^0)).
$$
Then $\gamma$ is a polynomial, and 
$$
\gamma'(t) = D\tilde\pi_1(y^0)De^{-tX_1}(e^{tX_1}(y^0)) X_{12}(e^{tX_1}(y^0)) = D\tilde\pi_1(e^{tX_1}(y^0))X_{12}(e^{tX_1}(y_0)).
$$
Thus by Lemma~\ref{L:wedges approx 1},
\begin{equation} \label{E:contradicts gamma || gamma'}
|\gamma(t^k)| \approx_A (\alpha_2^{j_0k})^{-1}, \qquad |\gamma(t^k) \wedge \gamma'(t^k)| \approx_A |\gamma(t^k)||\gamma'(t^k)|.
\end{equation}
By the definition of $\scriptL$ and a bit of arithmetic,
$$
\alpha_2^{j_0k} \sim \eps \eta_1^{\frac1{p_1}} 2^{-j_0}\eta_2^{-\frac1{p_2'}}2^{k\frac{p_2}{p_2'}},
$$
and thus for $B$ sufficiently large, \eqref{E:contradicts gamma || gamma'} contradicts Lemma~\ref{L:gamma || gamma'}.
\end{proof}

\section{Adding up the torsion scales} \label{S:ST full}

%%%%%%%%%%%%%%%%%%%%%%%%%%%%%%%%%%%%%%%%
%%%%%%%%%%%%%%%%%%%%%%%%%%%%%%%%%%%%%%%%
%%%%%%%%%%%%%%%%%%%%%%%%%%%%%%%%%%%%%%%%

In this section, we add up the different torsion scales, $\rho \sim 2^{-m}$, thereby completing the proof of Theorem~\ref{T:main}.

As in the previous section, we consider functions
$$
f_i = \sum_k 2^k \chi_{E_i^k}, \qquad \|f_i\|_{p_i} \sim 1, \qquad i=1,2,
$$
with the $E^k_i$ pairwise disjoint (as $k$ varies) for each $i$.  For $m \in \Z$, we define $U_m := \{\rho \sim 2^{-m}\}$.  By rescaling Proposition~\ref{P:local strong type}, we know that
$$
\scriptB_m(f_1,f_2) := \int_{U_m} f_1 \circ \pi_1 (x) \, f_2 \circ \pi_2(x)\, \rho(x)\, dx \lesssim 1.
$$
For $0 < \delta \lesssim 1$, define
$$
\scriptM(\delta) := \{m : \scriptB_m(f_1,f_2)  \sim \delta\}.
$$

Define $\theta := (p_1^{-1}+p_2^{-1})^{-1}$.  Then $0 < \theta < 1$.  We will prove that for each $0 < \delta \lesssim 1$,
\begin{equation} \label{E:Bm theta}
\sum_{m \in \scriptM(\delta)} \scriptB_m(f_1,f_2)^\theta \lesssim (\log \delta^{-1})^C.
\end{equation}
Thus 
$$
\sum_{m \in \scriptM(\delta)} \scriptB_m(f_1,f_2) \lesssim \delta^{1-\theta}(\log\delta^{-1})^C,
$$
which implies Theorem~\ref{T:main}.

The remainder of this section will be devoted to the proof of \eqref{E:Bm theta} for some fixed $\delta > 0$.  We will use the notation $A \lessapprox B$ to mean that $A \leq C\delta^{-C} B$ for some $C$ depending on $N$.  

For $m \in \scriptM(\delta)$ and $\eps,\eta_1,\eta_2 \lesssim 1$, define
\begin{align*}
\scriptL_m(\eps,\eta_1,\eta_2):= \{(j,k) : &\scriptB_m(\chi_{E_1^j},\chi_{E_2^k}) \sim \eps |E_1^j|^{\frac1{p_1}} |E_2^k|^{\frac1{p_2}}, \\
&\qquad 2^{jp_1}|E^j_1| \sim \eta_1, \: 2^{kp_2} |E^k_2| \sim \eta_2\}.
\end{align*}
By \eqref{E:sum on L}, the sum over all $(j,k)$ lying in any $\scriptL_m(\eps,\eta_1,\eta_2)$ with $\eps$, $\eta_1$, or $\eta_2$ much smaller than $\delta^C$ contributes a negligible amount to $\scriptB_m(f_1,f_2)$:
$$
\sum_{\min\{\eps,\eta_1,\eta_2\}<c\delta^C} \sum_{(j,k) \in \scriptL_m(\eps,\eta_1,\eta_2)}2^{j+k}\scriptB_m(\chi_{E_1^j},\chi_{E_2^k}) <  c\delta < \tfrac12 \scriptB_m(f_1,f_2).
$$
Thus the majority of each $\scriptB_m(f_1,f_2)$ is contributed by the $C(\log\delta^{-1})^3$  parameters $\eps,\eta_1,\eta_2 \approx 1$.  By the triangle inequality and pigeonholing, there exists (at least) one such triple for which 
\begin{align*}
\sum_{m \in \scriptM(\delta)}\scriptB_m(f_1,f_2)^\theta %&\lesssim (\log \delta^{-1})^{3\theta} \sum_{m \in \scriptM(\delta)} \bigl(\sum_{(j,k) \in \scriptL_m(\eps,\eta_1,\eta_2)} \scriptB_m(2^j \chi_{E^j_1},2^k \chi_{E^k_2})\bigr)^\theta\\
&\lesssim (\log \delta^{-1})^3 \sum_{m \in \scriptM(\delta)} \sum_{(j,k) \in \scriptL_m(\eps,\eta_1,\eta_2)} \scriptB_m(2^j \chi_{E^j_1},2^k \chi_{E^k_2})^\theta.
\end{align*}
Henceforth, we will abbreviate $\scriptL_m := \scriptL_m(\eps,\eta_1,\eta_2)$, for this choice of $\eps,\eta_1,\eta_2$.  

For $m \in \scriptM(\delta)$ and $(j,k) \in \scriptL_m$, we set
$$
\Omega^{jkm}:= U_m \cap \pi_1^{-1}(E^j_1) \cap \pi_2^{-1}(E^k_2), \qquad \alpha^{jkm}:= \bigl(\tfrac{|\Omega^{jk}_m|}{|E^j_1|},\tfrac{|\Omega^{jk}_m|}{|E^k_2|}\bigr).  
$$
There exist finite sets $\scriptA^{jkm} \subseteq \Omega^{jkm}$, satisfying the conclusions of Proposition~\ref{P:quasiex}, appropriately rescaled.  We proceed under the assumption that the minimal $n$-tuple $I$ for all of these sets are the same; the general case follows by taking a sum over all possible minimal $n$-tuples.  We set
$$
\tilde\Omega^{jkm} := \bigcup_{x \in \scriptA^{jkm}} \bigcap_{\sigma \in S_n} B^{I_\sigma}(x;c\delta^C\alpha^{jkm}).
$$
We recall that on these balls,
$$
(\alpha^{jkm})^{\deg I}|\lambda_I| \approx (\alpha^{jkm})^b 2^{-m(|b|-1)} \approx |\Omega^{jkm}| \sim |\tilde\Omega^{jkm}|.
$$
(The factor $|b|-1$ in the exponent is due to the form of the weight $\rho$.)  

As in the preceding section, we let $q_i:=\theta^{-1}p_i$.  By the definition of $\Omega^{jkm}$, $|\tilde\Omega^{jkm}| \gtrsim |\Omega^{jkm}|$, the restricted weak type inequality \eqref{E:RWT}, and H\"older's inequality, 
\begin{align*}
&\sum_{m \in \scriptM(\delta)} \sum_{(j,k) \in \scriptL_m} \scriptB_m(2^j\chi_{E^j_1},2^k\chi_{E^k_2})^\theta\\
&\qquad \lesssim \sum_{m \in \scriptM(\delta)} \sum_{(j,k) \in \scriptL_m} (2^{j+k}|\pi_1(\tilde\Omega^{jkm})|^{\frac1{p_1}}|\pi_2(\tilde\Omega^{jkm})|^{\frac1{p_2}})^\theta\\
&\qquad \lesssim (\sum_{m \in \scriptM(\delta)} \sum_{(j,k) \in \scriptL_m} 2^{jp_1}|\pi_1(\tilde\Omega^{jkm})|)^{\frac1{q_1}} (\sum_{m \in \scriptM(\delta)} \sum_{(j,k) \in \scriptL_m} 2^{kp_2}|\pi_2(\tilde\Omega^{jkm})|)^{\frac1{q_2}} .
\end{align*}
Thus the inequalities
\begin{gather} \label{E:many scales pi1 disj}
\sum_{m \in \scriptM(\delta)} \sum_{k : (j,k) \in \scriptL_m} |\pi_1(\tilde\Omega^{j,k,m})| \lesssim (\log \delta^{-1})^C |E_1^j|, \qquad j \in \Z\\
\label{E:many scales pi2 disj}
\sum_{m \in \scriptM(\delta)} \sum_{j: (j,k) \in \scriptL_m} |\pi_2(\tilde\Omega^{j,k,m})| \lesssim (\log \delta^{-1})^C |E_2^k|, \qquad k \in \Z,
\end{gather}
together would imply \eqref{E:Bm theta}.  The rest of the section will be devoted to the proof of \eqref{E:many scales pi1 disj}, the proof of \eqref{E:many scales pi2 disj} being similar.  

The proof is similar to the proof of Lemma~\ref{L:exist disjoint}; so we will just review that argument, giving the necessary changes.  Let $\scriptK \subseteq \Z^2$ be a finite set such that $(j,k) \in \scriptL_m$ for all $(k,m) \in \scriptK$ and such that the following sets are all $(B \log \delta^{-1})$-separated for some $B = B(N)$ sufficiently large for later purposes:
\begin{gather*}
\{k : (k,m) \in \scriptK, \ctc{for some m$\}$,} \qquad \{m : (k,m) \in \scriptK, \ctc{for some k$\}$,} \\
 \{m+\tfrac{p_2}{p_2'}k : (k,m) \in \scriptK\}.
\end{gather*}
(In the case of the last set, we recall that $\tfrac{p_2}{p_2'}$ is rational.)  It suffices to prove that
\begin{equation} \label{E:km in K}
\sum_{(k,m) \in \scriptK} |\pi_1(\tilde\Omega^{jkm})| \lesssim |E_1^j|.
\end{equation}
By the proof of Lemma~\ref{L:exist disjoint}, failure of \eqref{E:km in K} implies that there exists a subset $\scriptK' \subseteq \scriptK$ of cardinality $\#\scriptK' \geq M$, with $M=M(N)$ sufficiently large for later purposes, and points $x^{jkm} \in \scriptA^{jkm}$ such that
\begin{equation} \label{E:M km balls}
|\bigcap_{(k,m) \in \scriptK'} \pi_1(\bigcap_{\sigma \in S_n} B^{I_\sigma}(x^{jkm},c\delta^C\alpha^{jkm}))| \gtrapprox |E^j_1|,
\end{equation}
with $I = (w_1,\ldots,w_n)$ minimal and $w_n = (1)$.  By rescaling Lemma~\ref{L:M balls} to torsion scale $\rho \sim 2^{-m}$, for each $m$, $\#(\Z \times \{m\}) \cap \scriptK \lesssim 1$.  Thus we may assume that 
$$
\scriptK' = \{(k_1,m_1),\ldots,(k_M,m_M)\},
$$
with the $m_i$ all distinct.  Set $\alpha^i:= \alpha^{jk_im_i}$.  

As in the proof of Lemma~\ref{L:M balls}, we can construct a submersion $\tilde\pi_1$ and find points $y^i = e^{t^i X_1}(y)$ such that
\begin{gather} \label{E:rho y^i}
2^{-m_i(|b|-1)} \sim \rho(y^i)^{|b|-1} \approx (\alpha^i)^{-b} \max_{I'}(\alpha^i)^{\deg I'} |\lambda_{I'}(y^i)|,\\
\label{E:wedge y^i}
|\bigwedge_{s=1}^L D\tilde\pi_1(y) De^{t^i X_1}(y^i)(\alpha^i)^{\deg w_{l_s}} X_{w_{i_l}}(y^i)| \approx 1,
\end{gather}
for all $1 \leq i \leq M$ and $1 \leq l_1 < \cdots < l_L \leq n-1$.

By construction, the $m_i$ are all $(B \log \delta^{-1})$-separated.  Thus by Lemma~\ref{L:lambda t sim lambda 0} and \eqref{E:rho y^i}, for $B$ sufficiently large, 
\begin{equation} \label{E:ts separated}
|t^i-t^{i'}| \gtrapprox \alpha^i_1 + \alpha^{i'}_1, \qquad \text{for each $i \neq i'$};  
\end{equation}
otherwise, two distinct balls would share a point in common, whence $2^{m_i} \approx 2^{m_{i'}}$, a contradiction.  
With $\gamma(t) := D\tilde\pi_1(y) De^{t X_1}(e^{tX_1}(y)) X_2(e^{tX_1}(y))$, \eqref{E:wedge y^i} gives
\begin{equation} \label{E:gamma approx alpha2}
|\gamma(t^i)| \approx (\alpha_2^i)^{-1}, \quad |\gamma'(t^i)| \approx (\alpha^i_1\alpha^i_2)^{-1}, \quad |\gamma'(t^i) \wedge \gamma'(t^i)| \approx |\gamma(t^i)||\gamma'(t^i)|.
\end{equation}
Since 
$$
\alpha_2^i \sim \eps \eta_1^{\frac1{p_1}}\eta_2^{-\frac1{p_2'}}2^{-j} 2^{m_i+k_i\frac{p_2}{p_2'}},
$$
and the set of values $m_i+k_i\frac{p_2}{p_2'}$ takes on is $(B \log \delta^{-1})$-separated, by Lemma~\ref{L:gamma || gamma'}, we may assume that $m_i+k_i \frac{p_2}{p_2'}$ is constant as $i$ varies.  Thus we may fix $\alpha_2$ so that $\alpha_2^i \sim \alpha_2$ for all $i$.  

We note that 
$$
\alpha_1^i \sim \eps \eta_1^{-\frac1{p_1'}}\eta_2^{\frac1{p_2}}2^{-j\frac{p_1}{p_1'}}2^{m_i-k_i}.  
$$
Since 
$$
m_i-k_i = -\tfrac{p_2'}{p_2}(m_i+k_i\tfrac{p_2}{p_2'}) + p_2' m_i= -\tfrac{p_2'}{p_2}(m_1+k_1\tfrac{p_2}{p_2'}) + p_2' m_i,
$$
our prior deductions imply that the $m_i-k_i$ are all distinct, $(B \log \delta^{-1})$-separated.  Reindexing, we may assume that $m_1-k_1 < \cdots < m_M-k_M$.  Thus $\alpha_1^1 < \cdots < \alpha_1^M$.  

By Lemma~\ref{L:gamma sim t^k} (after a harmless time translation), we may assume that all of the $t^i$ lie within a single interval $I \subseteq (0,\infty)$ on which
\begin{equation} \label{E:monomial dominates}
|\tfrac1{k!} \gamma^{(k)}(0)t^k| < c_N|\tfrac1{k_0!} \gamma^{(k_0)}(0)t^{k_0}|, \qquad k \neq k_0,
\end{equation}
with $c_N$ sufficiently small.  As we have seen, $|\gamma(t^i)| \approx \alpha_2^{-1}$, for all $i$.  On the other hand, for $c_N$ sufficiently small, and any subinterval $I' \subseteq I$,
$$
|\int_{I'} \gamma'(t)\, dt| \sim |I|\max_{t \in I} |\gamma'(t)|.
$$
(We can put the norm outside of the integral by \eqref{E:monomial dominates}.)  
Specializing to the case when $I'$ has endpoints $t_1,t_2$, and using \eqref{E:ts separated}, 
$$
\alpha_1^2 (\alpha_1^1 \alpha_2)^{-1} \lessapprox |t_1-t_2| |\gamma'(t_2)| \lesssim |\gamma(t_2)-\gamma(t_1)| \approx (\alpha_2)^{-1},
$$
i.e.\ $\alpha_1^2 \lessapprox \alpha_1^1$, which is impossible for $B$ sufficiently large.  Thus we have a contradiction, and tracing back, \eqref{E:many scales pi1 disj} must hold.  This completes the proof of Theorem~\ref{T:main}.

%%%%%%%%%%%%%%%%%%%%%%%%%%%%%%%%%%%%%%%%
%%%%%%%%%%%%%%%%%%%%%%%%%%%%%%%%%%%%%%%%
%%%%%%%%%%%%%%%%%%%%%%%%%%%%%%%%%%%%%%%%

\section{Nilpotent Lie algebras and polynomial flows} \label{S:Nilpotent}

%%%%%%%%%%%%%%%%%%%%%%%%%%%%%%%%%%%%%%%%
%%%%%%%%%%%%%%%%%%%%%%%%%%%%%%%%%%%%%%%%
%%%%%%%%%%%%%%%%%%%%%%%%%%%%%%%%%%%%%%%%

In the next section, we will generalize Theorem~\ref{T:main} by relaxing the hypothesis that the flows of the vector fields $X_j$ must be polynomial.  In this section, we lay the groundwork for that generalization by reviewing some results from Lie group theory.  In short, we will see that if $M$ is a smooth manifold and $\mathfrak g_M \subseteq \scriptX(M)$ is a nilpotent Lie algebra, then there exist local coordinates for $M$ in which the flows of the elements of $\mathfrak g$ are polynomial.  These results have the advantage over the analogous results in \cite{GressmanPoly} that the lifting of the vector fields is by a local diffeomorphism, rather than a submersion; this will facilitate the global results in the next section.  

Throughout this section, $M$ will denote a connected $n$-dimensional manifold, and $\mathfrak g_M \subseteq \scriptX(M)$ will denote a Lie subalgebra of the space $\scriptX(M)$ of smooth vector fields on $M$.  We assume throughout that $\mathfrak g_M$ is nilpotent, and we let $N:=\dim \mathfrak g_M$.  We further assume that the elements of $\mathfrak g_M$ span the tangent space to $M$ at every point.  We will say that a quantity is bounded if it is bounded by a finite, nonzero constant depending only on $N$, and our implicit constants will continue to depend only on $N$.  

For the moment, we will largely forget about the manifold $M$.  

Let $G$ denote the unique connected, simply connected Lie group with Lie algebra $\mathfrak g_M$.  For clarity, we denote the Lie algebra of right invariant vector fields on $G$ by $\mathfrak g$, and we fix an isomorphism $X \mapsto \hat X$ of $\mathfrak g_M$ onto $\mathfrak g$.  Under the natural identification of $G$ as a subgroup of $\rm{Aut}(G)$, $G = \exp(\mathfrak g)$, and the group law is given by $e^{\hat X} \cdot e^{\hat Y} = e^{\hat X} \circ e^{\hat Y} = e^{\hat X * \hat Y}$, where $X*Y$ a Lie polynomial in $X$ and $Y$, which is given explicitly by the Baker--Campbell--Hausdorff formula.  

Let $S$ be a Lie subgroup of $G$.  The Lie algebra $\mathfrak z$ of $S$ is a Lie subalgebra of $\mathfrak g$, and $Z:= \exp(\mathfrak z)$ is the connected component of $S$ containing the identity.  In addition, $Z$ is a normal subgroup of $S$.  Let $n:=N-\dim \mathfrak z$.  (Later on, we will set $\mathfrak z = \mathfrak z_{x_0} := \{\hat X \in \mathfrak g : X(x_0)=0\}$ and $S=S_{x_0} := \{e^{\hat X} : e^X(x_0)=x_0\}$.)  

Let $\Pi:G \to G/Z$ denote the quotient map.  For $g \in G$ and $s \in S$, left multiplication by $g$ and right multiplication by $s$ have well-defined pushforwards; in other words, there exist automorphisms $\Pi_*l_g$, $\Pi_*r_s$ on $G/Z$ such that
$$
(\Pi_*l_g)(hZ) = (gh)Z, \qquad (\Pi_* r_s)(hZ) = (hs)Z,
$$
for every $h \in G$.  

Our next task is to find good coordinates on $G$.  

\begin{lemma}[{\cite[Theorem 1.1.13]{Corwin}}] \label{L:exists malcev}
There exists an ordered basis $\{\hat X_1,\ldots,\hat X_N\}$ of $\mathfrak g$, such that for each $k$, the linear span $\mathfrak g_k$ of $\{\hat X_{k+1},\ldots,\hat X_N\}$ is a Lie subalgebra of $\mathfrak g$ and such that $\mathfrak g_n = \mathfrak z$.  
\end{lemma}

We will not replicate the proof.  

Such a basis is called a weak Malcev basis of $\mathfrak g$ through $\mathfrak z$. As we will see, the utility of weak Malcev bases is that they give coordinates for $G$ and $G/Z$ in which the flows of our vector fields are polynomial.  We will say that a function $q$ is a polynomial diffeomorphism on $\R^N$ if $q:\R^N \to \R^N$ is a polynomial having a well defined inverse $q^{-1}:\R^N \to \R^N$ that is also a polynomial.  Polynomial diffeomorphisms must have constant Jacobian determinant; we will say that they are volume-preserving if this constant equals 1.  

Fix a weak Malcev basis $\{\hat X_1,\ldots,\hat X_N\}$ for $\mathfrak g$ through $\mathfrak z$.  For convenience, we will use the notation $x \cdot \vec X := \sum_{j=1}^N x_j \hat X_j$, for $x \in \R^N$.  Define 
$$
\psi(x) := e^{x_1 \hat X_1} \cdots e^{x_N \hat X_N}.
$$

\begin{lemma}\label{L:malcev good} 
There exists a polynomial diffeomorphism $p$ on $\R^N$ such that $\psi(x) = \exp(p(x) \cdot \vec X)$.  In particular, $\psi$ is a diffeomorphism of $\R^N$ onto $G$.  In these coordinates, the right and left exponential maps are polynomial.  More precisely, for $x^1,x^2 \in \R^N$,
$$
e^{x^2 \cdot \vec X} \psi(x^1) = \psi(q(x^1,x^2)), \qquad \psi(x^1)e^{x^2 \cdot \vec X} = \psi(r(x^1,x^2)),
$$
where $q,r:\R^{2N} \to \R^N$ are polynomials, $q(\cdot,x^2)$ and $r(\cdot,x^2)$ are volume-preserving polynomial diffeomorphisms for each $x^2$, and for each $1 \leq i \leq N$, $q_i(x^1,x^2)$ only depends on $x_1^1,\ldots,x_i^1$, and $x^2$.
\end{lemma}

\begin{proof}
The assertion on $p$ is just Proposition~1.2.8 of \cite{Corwin}.  That $q$ and $r$ are polynomial just follows by taking compositions:  
$$
\exp(q(x^1,x^2)) = \exp(x^2 \cdot \vec X) \psi(x^1) = \exp((x^2 \cdot \vec X)*p(x^1)) = \psi(p^{-1}((x^2 \cdot \vec X)*p(x^1)));
$$
similarly for $r$.  

The inverse of $r(\cdot,x^2)$ is $r(\cdot,-x^2)$, also a polynomial.  Since $r(r(x^1,x^2),-x^2) \equiv x^1$, $\det (D_{x^1} r)(r(x^1,x^2),-x^2) \det D_{x^1} r(x^1,x^2) \equiv 1$, and since both determinants are polynomial in $x^1$ and $x^2$, both must be constant.  Finally, since $r(x^1,0)$ is the identity, this constant must be 1.  

We turn to the dependence of $q_i$ on $x^2$ and the first $i$ entries of $x^1$.  Set $G_k:= \exp(\mathfrak g_k)$ (in the notation of Lemma~\ref{L:exists malcev}).  Our coordinates $\psi$ on $G$ give rise to diffeomorphisms
$$
\phi_k:\R^k \to G/G_k, \qquad \phi_k(y) := \psi(y,0) G_k.
$$
In these coordinates, the projections $\Pi_k:G \to G/G_k$ may be expressed as coordinate projections:  $\phi_k^{-1} \circ \Pi_k \circ \psi(y,z) = y$.  Since left multiplication pushes forward via $\Pi_k$,
\begin{align*}
(q_1,\ldots,q_i)(y,z,x^2) &= \phi_k^{-1} \circ \Pi_k (l_{e^{x^2 \cdot \vec X}} \psi(y,z)) = \phi_k^{-1}((\Pi_k)_* l_{e^{x^2 \cdot \vec X}} \Pi_k \psi(y,z)) \\ &= \phi_k^{-1}((\Pi_k)_* l_{e^{x^2 \cdot \vec X}} \phi_k(y)),
\end{align*}
which is independent of $z$.  
\end{proof}

Recalling that $Z=G_n$, we set $\phi := \phi_n$.  The pushforwards $\Pi_* \hat X$, $\hat X \in \mathfrak g$, are well-defined and have polynomial flows; indeed, 
$$\exp(\Pi_*(x \cdot \vec X))(\phi(y)) = \phi(q_1((y,0),x),\ldots,q_n((y,0),x)).
$$
Furthermore, $\Pi_*$ is a Lie group homomorphism of $\mathfrak g$ onto a Lie subgroup of $\mathcal X(G/Z)$, and, since $\Pi_*$ is a submersion and $\mathfrak g$ spans the tangent space to $\R^N$ at every point, $\Pi_* \mathfrak g$ spans the tangent space to $\R^n$ at every point.  

Next we examine the pushforwards $\Pi_* r_s$ of right multiplication by $s \in S$.  First, a preliminary remark.  Since $Z$ is a normal subgroup of $S$, $S$ acts on $Z$ by conjugation.  Replacing $G$ with $Z$, Lemma~\ref{L:malcev good} implies that the pushforward $\psi_* dz$ of $(N-n)$-dimensional Hausdorff measure on $Z$ is a bi-invariant Haar measure on $Z$.  We may uniquely extend this to a bi-invariant Haar measure on $S$.  Both $Z$ and this Haar measure on $S$ are invariant under the conjugation action, so $\psi_*dz$ is invariant under the conjugation action of $S$.  

\begin{lemma} \label{L:rs volume preserving}
In the coordinates given by $\phi$, the pushfoward $\Pi_* r_s$ is a volume-preserving polynomial diffeomorphism.
\end{lemma}

\begin{proof}
By Lemma~\ref{L:malcev good}, for each $s \in S$, there exists a polynomial $r^s:\R^N \to \R^N$ such that $r_s(\psi(x)) = \psi(r^s(x))$.  From the definition of the pushforward,
$$
\Pi_* r_s(\phi(y)) = \Pi(r_s(\psi(y,0))) = \Pi(\psi(r^s(y,0))) = \phi((r^s_1,\ldots,r^s_n)(y,0)),
$$
and taking the composition with $\phi^{-1}$ yields a polynomial.  Since $(r^s)^{-1} = r^{-s}$, this is also a polynomial diffeomorphism.  It remains to verify that this diffeomorphism is volume-preserving.

For simplicity, we will use vertical bars to denote the pushforward by $\phi$ of Lebesgue measure on $\R^n$ to $G/Z$ and also the pushforwards by $\psi$ of Lebesgue measure on $\R^N$ to $G$ and Hausdorff measure on $\R^{N-n} \times \{0\}$ to $Z$.  Fix an open, unit volume set $B \subseteq Z$.  By the remarks preceding the statement of Lemma~\ref{L:rs volume preserving}, $|s^{-1}Bs| = |B| = 1$.  Let $U \subseteq G/Z$ be measurable, and let $\sigma:G/Z \to G$ denote the section $\sigma(u) := \psi(\phi^{-1}(u),0)$. By the coarea formula,
$$
|\Pi_*r_s U| = |\sigma(\Pi_* r_s U) (s^{-1} B s)|.
$$
Of course, $\sigma(\Pi_* r_s U) (s^{-1} B s) = (\sigma(U) B)s$, so using the fact that right multiplication by $s$ is volume-preserving, and using the coarea formula a second time,
$$
|\Pi_*r_s U| = |\sigma(U) B| = |U|.
$$
\end{proof}

Now we are ready to return to our $n$-dimensional manifold $M$ from the opening of this section.  Fix $x_0 \in M$, and set $\mathfrak z =\mathfrak z_{x_0} := \{\hat X \in \mathfrak g : X(x_0) = 0\}$ and $Z=Z_{x_0} := \exp(\mathfrak z)$.  

We consider the smooth manifold $H = H_{x_0} :=\R^n \times M$, and view $\mathfrak g\simeq \mathfrak g_H$ as a tangent distribution on $H$, with elements $(\phi^*\Pi_* \hat X) \oplus X \in \mathfrak g_H$.  By the Frobenius theorem, there exists a smooth submanifold $(0,x_0) \in L=L_{x_0} \subseteq H$ whose tangent space equals the span of the elements of $\mathfrak g_H$ at each point.  The dimension of this leaf equals $n$; indeed, the map $\phi^*\Pi_* \hat X(0) \mapsto X(x_0)$ is an isomorphism, so its graph, $T_{(0,x_0)}H$, has dimension $n$.  

We let $p_1:L \to \R^n$ and $p_2:L \to M$ denote the restrictions to $L$ of the coordinate projections of $H$ onto $\R^n$ and $M$, respectively.  These restrictions are smooth, because $\phi^*\Pi_*\mathfrak g$ and $\mathfrak g_M$ span the tangent spaces to $\R^n$ and $M$, respectively, at every point.  For this same reason, they are in fact submersions, and hence local diffeomorphisms.  Composition of $p_2$ with a local inverse for $p_1$ immediately yields the following.

\begin{lemma} \label{L:local poly flows}
Let $x_0 \in M$ and fix a weak Malcev basis $\{\hat X_1,\ldots,\hat X_N\}$ of $\mathfrak g$ through $\mathfrak z_{x_0}$.  Then there exist neighborhoods $V_{x_0}$ of $0$ in $\R^n$ and $U_{x_0}$ of $x_0$ such that the map
$$
\Phi_{x_0}(y):= e^{y_1 X_1} \cdots e^{y_n X_n}(x_0)
$$
is a diffeomorphism of $V_{x_0}$ onto $U_{x_0}$, and, moreover, the pullbacks $\tilde X := (\Phi_{x_0})^*X$, $X \in \mathfrak g_M$ may be extended to globally defined vector fields on $\R^n$ for which each exponentiation $(t,x) \mapsto e^{t\hat X}(x_0)$ is a polynomial of bounded degree.
\end{lemma}

We would like to remove the restriction to small neighborhoods of points in $M$ from the preceding.  

\begin{lemma}\label{L:p_2 covering}
The projection $p_2:L_{x_0} \to M$ is a covering map.
\end{lemma}

\begin{proof}[Proof of Lemma~\ref{L:p_2 covering}]
That $p_2$ is surjective follows from H\"ormander's condition and connectedness of $M$.  Indeed, any point of the form $e^{X_1} \cdots e^{X_K}(x_0)$ (here we assume that each of the exponentials is defined) is in the range of $p_2$, and the set of such points is both open and closed in $M$.  (This is Chow's theorem.)

Let $x \in M$.  Fix a weak Malcev basis $\{\hat W_1,\ldots,\hat W_N\}$ of $\mathfrak g$ through $\mathfrak z_x$.  Then there exist neighborhoods $0 \in V_x \subseteq \R^n$ and $x \in U_x \subseteq M$ such that 
$$
\Phi_x(w):= e^{w_1 W_1} \cdots e^{w_nW_n}(x)
$$
is a diffeomorphism of $V_x$ onto $U_x$, so 
$$
p_2^{-1}(U_x) = \bigcup_{y: (y,x) \in L_{x_0}} \{(e^{w_1\tilde W_1} \cdots e^{w_n \tilde W_n}(y), \Phi_x(w)):w \in V\},
$$
where $\tilde W_n:= \phi^*\Pi_* \hat W_n$, and the restriction of $p_2$ to each set in this union is a diffeomorphism.
\end{proof}

\begin{lemma} \label{L:p_1 injective}
Assume that the exponential $e^X(x_0)$ is defined for every $X \in \mathfrak g_M$.  Then the projection $p_1:L_{x_0} \to G/Z_{x_0}$ is one-to-one.  
\end{lemma}

\begin{proof}
The projection $p_1$ fails to be one-to-one if and only if there exist $\hat X_1,\ldots,\hat X_K \in \mathfrak g$ such 
$$
e^{X_1} \cdots e^{X_K}(x_0)
$$
is defined and not equal to $x_0$, but $\hat X_1 * \cdots * \hat X_K = 0$.  Thus it suffices to show that if $e^{X_1} \cdots e^{X_K}(x_0)$ is defined, it equals $e^{X_1*\cdots*X_K}(x_0)$.  By induction, it suffices to prove this when $K=2$.  

Assume that $e^Xe^Y(x_0)$ is defined, and let 
$$
E:= \{t \in [0,1]:e^{sX}e^Y(x_0) = e^{(sX)*Y}(x_0),\: s \in [0,t]\}.  
$$
Let $Y_t:=(tX)*Y$, $t \in [0,1]$.  It suffices to prove that there exists $\delta > 0$ such that for each $t \in [0,1]$ and $0 <s<\delta$, $e^{sX}e^{Y_t}(x_0) = e^{(sX)*Y_t}(x_0)$.  From our initial remark, $p_1$ is one-to-one on each of the sets
$$
\Gamma_t:= \{e^{s \hat Y_t}(0,x_0) : s \in [0,1]\}, \qquad \hat Y_t:=\tilde Y_t \oplus Y_t \in \mathfrak g_H.
$$
Since $p_1$ is a local diffeomorphism, the $\Gamma_t$ are compact, and $t$ varies in a compact interval, there exists $\delta > 0$ such that $p_1$ is a diffeomorphism on the neighborhoods 
$$
N_\delta(\Gamma_t):= \{e^{\hat Z}(z) : \hat Z \in \mathfrak g_H, \, |\hat Z| < \delta, \, z \in \Gamma_t\}.
$$
Since $p_2 \circ p_1^{-1}|_{N_\delta(\Gamma_t)}$ is a diffeomorphism, for $s$ sufficiently small (independent of $t$),
$$
e^{sX}e^{Y_t}(x_0) = p_2 \circ p_1^{-1}(e^{s\tilde X}e^{\tilde Y_t}(0)) = p_2 \circ p_1^{-1}(e^{(s\tilde X)*\tilde Y_t}(0)) = e^{(sX)*Y_t}(x_0).
$$
\end{proof}

Taking the composition $p_2 \circ p_1^{-1}$, we obtain the following.

\begin{proposition} \label{P:Rn covers M}
Let $x_0 \in M$, and assume that $e^X(x_0)$ is defined for each $X \in \mathfrak g_M$.  Fix a weak Malcev basis $\{\hat X_1,\ldots,\hat X_N\}$ of $\mathfrak g$ through $\mathfrak z_{x_0}$.  Then the map
$$
\Phi_{x_0}(y) := e^{y_1 X_1} \cdots e^{y_n X_n}(x_0)
$$
is a local diffeomorphism of $\R^n$ onto $M$, which is also a covering map.  For each $X \in \mathfrak g_M$, the flow $(t,x) \mapsto e^{t\tilde X}(x)$ of the pullback $\tilde X := \Phi^*_{x_0}X$ is polynomial.  Finally, the covering is regular, and elements of the deck transformation group are volume-preserving.  
\end{proposition}

Much of the proposition has already been proved; our main task is the following.

\begin{lemma}\label{L:deck transformations}
Let $S:=\{e^{\hat X} \in G : e^X(x_0) = x_0\}$.  Then the deck transformation group $\rm{Aut}(\Phi_{x_0})$ of $\Phi_{x_0}$ coincides with the group $\scriptS \subseteq \rm{Diff}(\R^n)$ whose elements are the pushforwards $\hat r_s := \phi^*\Pi_* r_s$ of right multiplication by elements of $S$.  
\end{lemma}

\begin{proof}[Proof of Lemma~\ref{L:deck transformations}]
Let $s = e^{\hat X} \in S$.  Then 
$$
\Phi_{x_0} \circ \phi^{-1} \circ (\Pi_* r_s) \circ \phi(y) = e^{y_1 Y_1} \cdots e^{y_nY_n}e^X(x_0) = \Phi_{x_0}(y),
$$
so $\scriptS \subseteq \rm{Aut}(\Phi_{x_0})$.  If $y_0 \in \Phi_{x_0}^{-1}(x_0)$, then we may write $y_0 = e^{\tilde X}(0)$, with $e^{\hat X} \in S$, so $\scriptS$ acts transitively on the fiber $\Phi_{x_0}^{-1}(x_0)$.  

Let $f \in \rm{Aut}(\Phi_{x_0})$, and set $y_0:=f(0)$.  By the preceding, there exists an element $r \in \scriptS$ such that $r(0) = y_0$.  We claim that $f=r$.  The set of points where the maps coincide is closed by continuity.  If $f(y)=r(y)$, then the maps must coincide on a neighborhood of $y$, because $\Phi_{x_0}$ is a covering map.  Thus the set of points where the maps coincide is also open.  Since $f(0) = r(0)$, $f \equiv r$.  
\end{proof}

\begin{proof}[Proof of Proposition~\ref{P:Rn covers M}]
It remains to prove that the covering $\Phi_{x_0}$ is regular, and that the elements of its deck transformation group are volume-preserving.  By Lemma~\ref{L:rs volume preserving}, the deck transformations are all volume-preserving, and as seen in the proof of Lemma~\ref{L:deck transformations}, $\rm{Aut}(\Phi_{x_0})$ acts transitively on $\Phi_{x_0}^{-1}(x_0)$, which is to say that $\Phi_{x_0}$ is regular.  
\end{proof}

\section{Generalizations of Theorem~\ref{T:main}} \label{S:Extensions}

%%%%%%%%%%%%%%%%%%%%%%%%%%%%%%%%%%%%%%%%
%%%%%%%%%%%%%%%%%%%%%%%%%%%%%%%%%%%%%%%%
%%%%%%%%%%%%%%%%%%%%%%%%%%%%%%%%%%%%%%%%

In \cite{GressmanPoly}, which sparked our interest in this problem, Gressman established unweighted, local, endpoint restricted weak type inequalities, subject to the hypotheses that the $\pi_j:\R^n \supseteq U \to \R^{n-1}$ are smooth submersions and that there exist smooth, nonvanishing vector fields $Y_1,Y_2$ on $U$ that are tangent to the fibers of the $\pi_j$ and generate a nilpotent Lie algebra.  Thus the results of \cite{GressmanPoly} are more general than Theorem~\ref{T:main} in two respects:  The hypotheses are made on vector fields parallel to the fibers, and these vector fields are only assumed to generate a nilpotent Lie algebra, not to have polynomial flows.  In this section, we address both of these generalizations.  

\subsection*{Changes of variables, changes of measure, and the affine arclengths}

The above mentioned generalizations will be achieved by using the results of the previous section, so we begin by observing how the weights $\rho_\beta$ transform under compositions of the $\pi_j$ with diffeomorphisms.  We note that the same computations also show how the $\rho_\beta$ transform under smooth changes of the measures on $M$ and the $N_j$. (Changes of measure change the vector fields associated to the maps $\pi_1,\pi_2$ by the coarea formula.)  

Let $F:\R^n \to \R^n$ be a diffeomorphism, and let $G_j:\R^{n-1} \to \R^{n-1}$ be a smooth map, $j=1,2$.  Define $\hat\pi_j := G_j \circ \pi_j \circ F$.  These maps give rise to associated vector fields $\hat X_j$, and a simple computation shows that
\begin{equation} \label{E:hat Xj coords}
\hat X_j = [(\det DG_j) \circ \pi_j \circ F] (\det DF)  F^*X_j, 
\end{equation}
where $F^*$ denotes the pullback $F^* X_j := (DF)^{-1} X_j \circ F$.  We continue to let $\Psi_{F(x_0)}(t)$ denote the map obtained by iteratively flowing along the $X_i$ and let $\hat\Psi_{x_0}(t)$ denote the map obtained by iteratively flowing along the $\hat X_i$.  

By naturality of the Lie bracket and the Chain Rule, we thus have for any multiindex $\beta$ that
\begin{equation} \label{E:rho beta coords}
\partial^\beta \det D\hat\Psi_{x_0}(0) = \sum_{\beta' \preceq \beta} G^\beta_{\beta'}(F(x_0)) \partial^{\beta'}\det D\Psi_{F(x_0)}(0).
\end{equation}
Here `$\preceq$' denotes the coordinate-wise partial order on multiindices, 
$$
G_\beta^\beta(F(x_0)) := (\det DF(x_0))^{b_1+b_2-1}(\det DG_1 \circ \pi_1 \circ F(x_0))^{b_1}(\det DG_2 \circ \pi_2 \circ F(x_0))^{b_2},
$$
and for $\beta' \prec \beta$, $G^\beta_{\beta'}$ is a smooth function involving derivatives of the Jacobian determinants $\det DF$, $\det DG_i$.  

This allows us to bound the weight associated to the maps $\hat\pi_1,\hat\pi_2$ and multiindex $\beta$:
\begin{equation} \label{E:bound rho beta hat}
\begin{aligned}
|\hat \rho_\beta| \leq &|\det DF| |(\det DG_1) \circ \pi_1 \circ F|^{\frac1{p_1}} |(\det DG_2)\circ \pi_2 \circ F|^{\frac1{p_2}} \rho_\beta \circ F \\
&\qquad + \sum_{\beta' \prec \beta} g_\beta^{\beta'} \rho_{\beta'}^{\frac{|b'|-1}{|b|-1}} \circ F,
\end{aligned}
\end{equation}
where the $g_\beta^{\beta'}$ are continuous and equal zero if $\det DF$, $\det DG_1$, and $\det DG_2$ are constant, and $b' = b(\beta')$ and $\rho_{\beta'}$ are as in \eqref{E:def b}, \eqref{E:def rho}, respectively, $p$ is as in \eqref{E:def p}, and vertical bars around $b$'s denote the $\ell^1$ norm.

We turn to an estimate for 
$$
\int \prod_{j=1}^2 |f_j \circ \hat \pi_j| \hat \rho_\beta\, a(x)\, dx,
$$
with $|a| \leq 1$ a cutoff function (possibly identically 1).  We begin with the contribution from the main term of \eqref{E:bound rho beta hat}.  Assuming \eqref{E:main}, the change of variables formula gives
\begin{equation} \label{E:main term after coord change}
\int \bigl(\prod_{j=1}^2 |f_j \circ \hat \pi_j|\, |\det DG_j \circ \pi_j \circ F|^{\frac1{p_j}}\bigr) |\det DF| \, \rho_\beta \circ F\, a \, dx \lesssim \prod_{j=1}^2 \|f_j\|_{p_j}.
\end{equation}

Now we turn to the error terms.  Fix $\beta' \prec \beta$ and assume that $a$ has compact support.  The analogue of \eqref{E:main}, with $\beta'$ in place of $\beta$, together with the change of variables formula, yields 
\begin{equation} \label{E:Lp pi hat}
\bigl|\int \bigl(\prod_{j=1}^2 f_j \circ \hat\pi_j\bigr)(g_\beta^{\beta'})^{1/\theta} \rho_{\beta'} \circ F \,a\, dx\bigr| \lesssim_{F,G_1,G_2} \prod_{j=1}^2 \|f_j\|_{q_j},
\end{equation}
where $q = p(b') = (\tfrac{|b'|-1}{b_1},\tfrac{|b'|-1}{b_2})$ and $\theta = \tfrac{|b'|-1}{|b|-1}$.  Provided that the $\pi_j$ are submersions on the support of $a$, H\"older's inequality gives
\begin{equation} \label{E:holder pi hat}
\bigl|\int \bigl(\prod_{j=1}^2 f_j \circ \hat\pi_j\bigr) \,a\, dx\bigr| \lesssim_{F,G_1,G_2,\pi_1,\pi_2} \diam(\supp a) \prod_{j=1}^2 \|f_j\|_{r_j},
\end{equation}
where $(r_1,r_2) = (\frac{|b|-|b'|}{b_1-b_1'},\frac{|b|-|b'|}{b_2-b_2'})$.  Since $(p_1^{-1},p_2^{-1}) = \theta(q_1^{-1},q_2^{-1})+(1-\theta)(r_1^{-1},r_2^{-1})$, complex interpolation gives
\begin{equation} \label{E:error term after coord change}
\bigl|\int \bigl(\prod_{j=1}^2 f_j \circ \hat\pi_j\bigr)g_\beta^{\beta'} (\rho_{\beta'} \circ F)^\theta \,a\, dx\bigr|  \lesssim_{F,G_1,G_2,\pi_1,\pi_2} \diam(a)^{1-\theta}\prod_{j=1}^2 \|f_j\|_{p_j},
\end{equation}
so the error terms are harmless for sufficiently local estimates \textbf{in the special case} that the $\pi_j$ are submersions on the support of $a$.  

\subsection*{Uniform local estimates}

For simplicity, we will give our local estimates in coordinates.  Let $U \subseteq \R^n$ be an open set, let $\pi_1,\pi_2:U \to \R^n$ be smooth maps, and let $X_1,X_2$ denote the vector fields associated to the $\pi_j$ by \eqref{E:def X}.  Assume that:\\
(i) For $j=1,2$ and a.e.\ $y \in \pi_j(U)$, $\pi_j^{-1}(y)$ is contained in a single integral curve of $X_j$;\\
(ii) The Lie algebra generated by $X_1,X_2$ spans the tangent space to $\R^n$ at every point of $U$;\\
(iii) There exist smooth, nonvanishing functions $h_1,h_2$ such that the vector fields $Y_j:= h_j X_j$, $j=1,2$, generate a nilpotent Lie algebra of step at most $N$.

We note that even if one knows that (i-iii) hold, it may be very difficult to find $h_1,h_2$.  Our next proposition allows one to use the ``wrong'' vector fields (the $X_j$), at least locally, and for certain $\beta$.  

\begin{proposition}\label{P:local version}
Fix $x_0 \in U$.  If $\beta$ is minimal in the sense that $\beta' \prec \beta$ implies $\rho_{\beta'} \equiv 0$, or if $d\pi_1(x_0)$ and $d\pi_2(x_0)$ both have full rank, then there exists a neighborhood $U_{x_0}$ of $x_0$, depending on $x_0$ and the $\pi_j$, such that for all $f_1,f_2 \in C^0(U)$,
\begin{equation} \label{E:local version}
|\int_{U_{x_0}} \prod_{j=1}^2 f_j \circ \pi_j(x) \, \rho_\beta(x)\, dx| \leq C_N \prod_{j=1}^2 \|f_j\|_{p_j};
\end{equation}
here $\rho_\beta$ is the weight \eqref{E:def rho}, defined using the $X_j$, not the $Y_j$.  
\end{proposition}

Proposition~\ref{P:local version} implies a uniform, strong type endpoint version of the restricted weak type result in \cite{GressmanPoly}.  We remark that uniform bounds are impossible if we define the weight $\rho_\beta$ using the $Y_i$.  This can be seen by replacing $Y_i$ with $\lambda Y_i$ and sending $\lambda \to \infty$.  In Section~\ref{S:Examples}, we will give a counter-example showing the impossibility of global bounds under these hypotheses in the case that $\beta$ is non-minimal. 

\begin{proof}[Proof of Proposition~\ref{P:local version}]
By Lemma~\ref{L:local poly flows}, we may find neighborhoods $U_{x_0}$ of $x_0$ and $V_{x_0}$ of $0$, and a diffeomorphism $\Phi_{x_0}:V_{x_0} \to U_{x_0}$ such that the pullbacks $\hat Y_j$ of the $Y_j$ with respect to $\Phi_{x_0}$ extend to global vector fields with polynomial flows.  Let $\hat Z_j$ denote the vector field associated to $\hat \pi_j := \pi_j \circ \Phi_{x_0}$, via the natural analogue of \eqref{E:def X}.  Then 
$$
\hat Z_j = (\det D\Phi_{x_0}) \Phi_{x_0}^* X_j  = \tfrac1{\hat h_j} \hat Y_j, \qquad \hat h_j = h_j \circ \Phi_{x_0}.
$$

\begin{lemma} \label{L:push forward hj}
There exist functions $g_j$ on $\pi_j(U_{x_0})$ such that $\hat h_j = g_j \circ \hat \pi_j$, a.e.\ on $U_{x_0}$.  
\end{lemma}

\begin{proof}[Proof of Lemma~\ref{L:push forward hj}]
Since $\hat Y_j$ is polynomial, it is divergence free, and since $\hat Z_j$ is defined by \eqref{E:def X}, it is also divergence free.  Since
$$
0=\rm{div}\, \hat Z_j = \tfrac1{\hat h_j} \rm{div}\,\hat Y_j + \hat Y_j (\tfrac1{\hat h_j}) = \hat Y_j (\tfrac1{\hat h_j}),
$$
$\hat h_j$ is constant on the integral curves of $\hat Y_j$.  By our hypothesis on the fibers of the $\pi_j$, the lemma follows.
\end{proof}

If $\Omega \subseteq V_{x_0}$,
\begin{align*}
|\Omega| &= \int_{\hat\pi_j(\Omega)} \int_{\pi_j^{-1}(y)}\chi_\Omega(t) \, |\hat Z_j(t)|^{-1}d\scriptH^1(t)\, dy\\
& = \int_{\hat\pi_j(\Omega)} \int_{\pi_j^{-1}(y)}\chi_\Omega(t) \, |\hat Y_j(t)|^{-1}d\scriptH^1(t)\,g_j(y)\, dy.
\end{align*}
Thus the change of variables formula and the proof of Theorem~\ref{T:main} (c.f.\ the argument leading to \eqref{E:upstairs}) imply that
\begin{equation} \label{E:Vx0 to Ux0}
\begin{aligned}
|\int_{U_{x_0}} \prod_{j=1}^2 f_j \circ \pi_j \, \hat \rho_\beta \circ \Phi_{x_0}^{-1}\,|\det D\Phi_{x_0}|^{-1} \, dx| 
&=|\int_{V_{x_0}} \prod_{j=1}^2 f_j \circ \hat\pi_j \,  \hat \rho_\beta \, dx| \\
& \lesssim \prod_{j=1}^2 \|f_j\|_{L^{p_j}(g_j\, dy)},
\end{aligned}
\end{equation}
where $\hat \rho_\beta$ is defined using $\hat Y_1$ and $\hat Y_2$.  We have seen that $\tfrac1{h_j} = g_j \circ \pi_j$, so computations similar to those leading up to \eqref{E:bound rho beta hat} give
\begin{align*}
|\rho_\beta| \leq |\det D\Phi_{x_0}|^{-1} |g_1 \circ \pi_1|^{\frac1{p_1}} |g_2 \circ \pi_2|^{\frac1{p_2}} \hat \rho_\beta \circ \Phi_{x_0}^{-1} + \sum_{\beta \prec \beta'} g_\beta^{\beta'} \rho_{\beta'}^{\frac{|b'|-1}{|b|-1}} \circ \Phi_{x_0}^{-1},
\end{align*}
where the $g_\beta^{\beta'}$ are continuous and involve derivatives of $\det D\Phi_{x_0}$, $g_1$, and $g_2$.  

Finally, \eqref{E:local version} follows from \eqref{E:main term after coord change} and \eqref{E:error term after coord change} in the case that $\beta$ is minimal or $d\pi_1(x_0)$ and $d\pi_2(x_0)$ both have full rank.  
%
%As an aside, we note that $h_1$ and $h_2$ are locally comparable to constants.  By rescaling, we may assume that $h_1(x_0) \sim h_2(x_0) \sim 1$, in which case \eqref{E:Vx0 to Ux0} gives
%$$
%|\int_{U_{x_0}} \prod_{j=1}^2 f_j \circ \pi_j \, \hat \rho_\beta \circ \Phi_{x_0}^{-1} \, dx| \lesssim \prod_{j=1}^2 \|f_j\|_{p_j}.
%$$
\end{proof}

\subsection*{A ``global'' version on manifolds} 

Let $M$ be a smooth $n$-dimensional manifold, let $P_1,P_2$ be smooth $(n-1)$-dimensional manifolds, and assume that $\pi_j:M \to P_j$ are smooth maps with a.e.\ surjective differentials.  Assume that we are given measures $\mu, \nu_1,\nu_2$ on $M, P_1,P_2$ that have smooth, nonvanishing densities in local coordinates.  

For instance, in the setting of \eqref{E:linear formulation} and \eqref{E:bilinear formulation}, we are given Riemannian manifolds $(P_1,h_1),(P_2,h_2)$ and a map 
$$
P_2 \times \R \ni (x,t) \mapsto \gamma_x(t) \in P_1;  
$$
here the measures $\nu_1,\nu_2$ are the Riemannian volume elements,  the manifold $M$ is simply $M = P_2 \times \R$, and $d\mu = |\gamma_{x_2}'(t)|_{h_1}\, d\nu_2\, dt$.  

 By \eqref{E:hat Xj coords}, we may define (up to a sign) vector fields $X_1,X_2 \in \scriptX(M)$ such that in any choice of local coordinates,
$$
X_j = d\pi_j^1 \wedge \cdots \wedge d\pi_j^{n-1}\,(\tfrac{d\nu_j}{dy} \circ \pi_j)\,(\tfrac{dx}{d\mu}).
$$
We observe that $\mu$ is invariant under the flow of the $X_j$, and hence is also invariant under the flow of elements of the Lie algebra $\mathfrak g_M$ generated by $X_1$ and $X_2$.  We assume that:  \\
(i) The Lie algebra $\mathfrak g_M$ generated by $X_1,X_2$ is nilpotent of step $N$, and the flows of its elements are complete;\\
(ii) For a.e.\ $y \in P_j$, $\pi_j^{-1}(y)$ is contained in a single integral curve of $X_j$.  

Let $M_0$ denote the (open) submanifold of $M$ on which $\mathfrak g_M$ spans the tangent space to $M$, and decompose $M_0$ into its connected components, $M_0 = \bigcup_k M_{0,k}$.  By the Frobenius Theorem, $\mathfrak g_M \subseteq \scriptX(M_{0,k})$ for each $k$.  We now put local coordinates on $M_0$ by fixing points $x_k \in M_{0,k}$ and letting $\Phi_k:=\Phi_{x_k}:\R^n \to M_{0,k}$ be the covering map guaranteed by Proposition~\ref{P:Rn covers M}.  

Fix $k$.  Then $\Phi_k$ is a local diffeomorphism, and the pullbacks of vector fields in $\mathfrak g_M$ by $\Phi_k$ have polynomial flows.  By composing $\Phi_k$ with an isotropic dilation, we may assume that for $U_k \subseteq \R^n$ open with $\Phi_k|_{U_k}$ one-to-one, $(\Phi_k|_{U_k})_*(dx) = d\mu$ on $\Phi_k(U_k)$.  (Such a dilation exists because $d\mu$ and the pushforward of $dx$ are both invariant under the flows of the $X_j$ and hence differ from one another by a constant by Chow's theorem.)  The vector fields $\hat X_j := \Phi_k^*X_j$ are divergence-free and tangent to the fibers of $\hat \pi_1:=\pi_1\circ \Phi_k$ and $\hat \pi_2 := \pi_2 \circ \Phi_k$, respectively.  For $\beta$ a multiindex, the $\hat X_j$ give rise to a measure $\hat \rho_\beta\, dx$ on $\R^n$.  If $r \in \rm{Aut}(\Phi)$ is an element of the deck transformation group, then $\hat\rho_\beta \circ r = \hat \rho_\beta$, and thus we can define a measure $\mu_\beta$ on $M_{0,k}$ by setting $\mu_\beta|_{\Phi_k(U_k)} := (\Phi_k)_*(\hat\rho_\beta\, dx|_{U_k})$ whenever $\Phi_k|_{U_k}$ is a diffeomorphism.  We extend this to a measure on $M$ by setting $\mu_\beta = 0$ on $M \setminus M_0$.  

The measure $\mu$ plays a slightly lesser role than the $\nu_j$ in the construction of the $\mu_\beta$.  The measure $\mu$ affects the definition of the vector fields $X_j$, and hence the nilpotency hypothesis, but in the minimal case (that $\rho_{\beta'} \equiv 0$ for all $\beta' \prec \beta$), all choices of $\mu$ lead to the same definition of $\mu_\beta$ by \eqref{E:rho beta coords}.  Moreover, in the case that $\beta$ is minimal, by \eqref{E:rho beta coords}, the analogous construction carried out with respect to any choice of local coordinates on $M$ would give rise to the same measure $\mu_\beta$.  When $\beta$ is non minimal, the measure depends on the choice of coordinates, but in any coordinates, the analogue of $\mu_\beta$ would vanish on $M \setminus M_0$.

\begin{theorem} \label{T:mfold version}
Under the notation and hypotheses above, let $V \subseteq M$ be an open set.  For each $k$, let $V_k:=V\cap M_{0,k}$, let $U_k \subseteq \R^n$ be an open set, and assume that for a.e.\ $x \in V_k$, 
$\#(\Phi_k^{-1}(x) \cap U_k) \geq A_k$, and for a.e.\ $y \in \pi_j(V_k)$, $U_k \cap \Phi_k^{-1}(\pi_j^{-1}\{y\})$ is contained in the union of at most $B_{j,k}$ integral curves of $\hat X_j$, with $0 < A_k,B_{1,k},B_{2,k} < \infty$.  Then 
\begin{equation} \label{E:mfold version}
|\int_V f_1 \circ \pi_1(x) f_2 \circ \pi_2(x)\, d\mu_\beta(x)| \lesssim \bigl(\sup_k \tfrac{B_{1,k}^{1/{p_1}}B_{2,k}^{1/{p_2}}}{A_k}\bigr) \|f_1\|_{L^{p_1}(P_1;\nu_1)}\|f_2\|_{L^{p_2}(P_2;\nu_2)}.
\end{equation}
Here the exponent pair $p = (p_1,p_2)$ is defined as in \eqref{E:def p}.  
\end{theorem}

The quantities $A_k,B_{1,k}, B_{2,k}$ count (in the absence of the polynomial hypothesis) oscillations naturally associated to the $\pi_j$.  The pre-image $\Phi_k^{-1}(x)$ is in one-to-one correspondence with $\rm{Aut}(\Phi_k)$ and may be viewed as the set of distinct paths of the form $t \mapsto e^{tX}(x)$, $0 \leq t \leq 1$, $X \in \mathfrak g_M$, that start and end at $x$.  Assume that $\pi_j(x)=y$.  The set of $\hat X_j$ integral curves containing $\Phi_k^{-1}(\pi_j^{-1}(y))$ equals the set of $\hat X_j$ integral curves containing $\Phi_k^{-1}(x)$, and thus is in one-to-one correspondence with the set of distinct paths of the form $t \mapsto \pi_j(e^{tX}(x))$, $0 \leq t \leq 1$, $X \in \mathfrak g_M$, that start and end at $y$.  Both of these sets are either singletons (the trivial loop) or are countable; intersecting with the set $U$ as in the hypothesis of the theorem makes other finite cardinalities possible.  As seen in the next section, the analogue of Theorem~\ref{T:mfold version} without some accounting for oscillations is false.

\begin{proof}[Proof of Theorem~\ref{T:mfold version}]
Let $P_{k,j} := \pi_j(M_{0,k})$.  By hypothesis (ii), for each $j=1,2$, the $P_{k,j}$ have measure zero (pairwise) intersection.  Therefore by H\"older's inequality and $p_1^{-1}+p_2^{-1}>1$, 
\begin{align*}
\sum_k \|f_1\|_{L^{p_1}(P_{k,1};\nu_1)}\|f_2\|_{L^{p_2}(P_{k,2};\nu_2)} 
&\leq \bigl(\sum_k\|f_1\|_{L^{p_1}(P_{k,1};\nu_1)}^{p_1} \bigr)^{\frac1{p_1}}\bigl(\sum_k \|f_2\|_{L^{p_2}(P_{k,2};\nu_2)}^{p_2}\bigr)^{\frac1{p_2}}\\
&=  \|f_1\|_{L^{p_1}(P_1;\nu_1)}\|f_2\|_{L^{p_2}(P_2;\nu_2)},
\end{align*}
so it suffices to prove \eqref{E:mfold version} when $V = V_k$ for some $k$.  This reduces matters to consideration of the special case when $M$ is connected and the elements of $\mathfrak g_M$ span the tangent space to $M$ at every point, and we may henceforth omit the subscript $k$ from the various objects in the setup of the theorem.

Define $\tilde P_j :=  \R^n/ [x \sim e^{t\hat X_j}(x)]$, i.e.\ the set of all $\hat X_j$ integral curves in $\R^n$, let $\tilde \pi_j:\R^n \to \tilde P_j$ denote the quotient map, and endow $\tilde P_j$ with the quotient topology.  The image $\tilde \pi_j(\R^n \setminus \{\hat X_j=0\})$ is then a smooth $(n-1)$-dimensional manifold.  Since $\pi_j \circ \Phi$ is constant along integral curves of $\hat X_j$ (and hence on the level sets of $\tilde \pi_j$), we may define a map $\tilde\Phi_j:\tilde P_j \to P_j$ by $\tilde\Phi_j(\tilde \pi_j(x)) := \pi_j(\Phi(x))$.  We observe that $\tilde\Phi_j$ is a local diffeomorphism from $\tilde \pi_j(\R^n \setminus \{\hat X_j=0\})$ onto $\pi_j(M \setminus \{X_j=0\})$. Our hypothesis on the $B_j$ is precisely the statement that for a.e.\ $y \in \pi_j(V)$, $\#(\tilde\Phi_j^{-1}(y) \cap \tilde\pi_j(U)) \leq B_j$.  

Because the flows of the $\hat X_j$ preserve Lebesgue measure, we may define Borel measures $\tilde \nu_j$ on the $\tilde P_j$ by setting $\tilde\nu_j(\tilde\pi_j(\{\hat X_j = 0\})) = 0$ and, for every finite measure $\Omega \subseteq \{\hat X_j \neq 0\}$,
\begin{align*}
|\Omega| 
&=\int_{\tilde \pi_j(\Omega)}\int_{\tilde\pi_j^{-1}(y)} \chi_\Omega(t)\, |\hat X_j(t)|^{-1} d\scriptH^1(t)\, d\tilde\nu_j(y).
\end{align*}
Equivalently, if $V$ is open and $\tilde\Phi_j|_V$ is a diffeomorphism, $\tilde\nu_j = (\tilde\Phi_j|_V)^{-1}_* \nu_j$.  

By the proof of Theorem~\ref{T:main}, which did not use the global Euclidean structure of $\R^{n-1}$, nor its algebraic properties, nor the specific measure $dx$, but only the local Euclidean structure of $\pi_j(\{X_j \neq 0\})$, 
\begin{equation} \label{E:upstairs}
\int_{\R^n} \prod_{j=1}^2 |\tilde f_j \circ \tilde \pi_j| \, \hat \rho_\beta(x)\, dx \lesssim \prod_{j=1}^2 \|\tilde f_j\|_{L^{p_j}(\tilde P_j; \tilde \nu_j)},
\end{equation}
for all pairs of continuous, compactly supported functions $\tilde f_j$ on $\tilde P_j$, $j=1,2$.  

Taking $V,U,A,B_1,B_2,\Phi$ as in the hypothesis of the theorem and $f_j$ a continuous function with compact support on $P_j$, $j=1,2$, we set $\tilde f_j:=(f_j \circ \tilde \Phi_j) \chi_{\tilde\pi_j(U)}$.  By construction,
\begin{gather*}
|\int_V \prod_{j=1}^2 f_j \circ \pi_j\, d\mu_\beta| \leq \tfrac1A \int_U \prod_{j=1}^2 |\tilde f_j \circ \pi_j \circ \Phi|\, \hat \rho_\beta(x)\, dx
= \tfrac1A \int_U \prod_{j=1}^2| \tilde f_j \circ \tilde\pi_j|\, \hat \rho_\beta(x)\, dx,\\
\qtq{and} \int_{\tilde P_j} |\tilde f_j|^{p_j}\, d\tilde\nu_j \leq B_j \int_{P_j} |f_j|^{p_j}\, d\nu_j.
\end{gather*}
Together with \eqref{E:upstairs}, the preceding two inequalities imply \eqref{E:mfold version}.  
\end{proof}

\section{Examples, counter-examples, and open questions} \label{S:Examples}

%%%%%%%%%%%%%%%%%%%%%%%%%%%%%%%%%%%%%%%%
%%%%%%%%%%%%%%%%%%%%%%%%%%%%%%%%%%%%%%%%
%%%%%%%%%%%%%%%%%%%%%%%%%%%%%%%%%%%%%%%%

%First, we note a simple extension that is afforded by a convexity argument and the fact that $p_1^{-1}+p_2^{-1} > 1$.  
%
%\begin{corollary}\label{C:disconnected}
%Define vector fields $X_1,X_2$ on $U_0$ by \eqref{E:Hodge}.  Then the conclusions of Theorem~\ref{T:main} continue to hold under the weaker hypotheses that \\
%(i)'  $X_w \equiv 0$ for each word $w$ of length greater than $N$ and $(t,x) \mapsto e^{tX_w}(x)$ is polynomial on each component of its domain\\
%(ii)'  For a.e.\ $y \in \R^{n-1}$, $\pi_j^{-1}(\{y\})$ is contained in a union of at most $N$ integral curves, lying in $U_0$, of $X_j$.
%\end{corollary}
%
%\begin{proof}...
%\end{proof}

\subsection*{The translation invariant case}  We begin with a concrete example.  The weights $\rho_\beta$ were originally conceived in \cite{BSapde} as a generalization of the affine arclength measure associated to curves, and the results of this article include, as a special case, results on translation invariant averages on curves with affine arclength measure.  Let $\gamma:\R \to \R^d$ be a polynomial of degree at most $N$.  Consider the maps $\pi_j:\R^{d+1} \to \R^d$ given by
$$
\pi_1(x,t) := x, \quad \pi_2(x,t) := x-\gamma(t).
$$
The vector fields associated to these maps are
$$
X_1 = \tfrac{\partial}{\partial t}, \quad X_2 = \tfrac{\partial}{\partial t} + \gamma'(t) \cdot \nabla_x,
$$
and $X_1,X_2$ generate a nilpotent Lie algebra on $\R^{d+1}$ whose elements have polynomial flows.  As discussed in Section~\ref{S:TW}, it is slightly easier to compute determinants of vector fields arising as iterated Lie brackets of the $X_i$, rather than derivatives of Jacobian determinants, so we look to Theorem~\ref{T:TW version}.  Provided that the polytope $\scriptP$ associated to $X_1,X_2$ via \eqref{E:big P} is nonempty, 
$$
\scriptP = \rm{ch}\bigl\{\bigl((d,1+\tfrac{d(d-1)}2)+[0,\infty)^2\bigr) \cup \bigl((1+\tfrac{d(d-1)}2,d)+[0,\infty)^2\bigr)\bigr\}.
$$
Thus minimal elements $b$ of $\scriptP$ lie on the line segment joining $(d,1+\tfrac{d(d-1)}2)$ and $(1+\tfrac{d(d-1)}2,d)$.  The corresponding Lebesgue exponents are those $(p_1,p_2)$ with $(p_1^{-1},p_2^{-1})$ lying on the line segment joining
$$
\bigl(\tfrac{2d}{d(d+1)},\tfrac{2+d(d-1)}{d(d+1)}\bigr), \qquad \bigl(\tfrac{2+d(d-1)}{d(d+1)},\tfrac{2d}{d(d+1)}\bigr),
$$
and the corresponding weights are all equal:
$$
|\lambda_I|^{\frac1{b_1+b_2-1}} = |\det(\gamma'(t),\ldots,\gamma^{(d)}(t))|^{\frac2{d(d+1)}}.
$$
Theorem~\ref{T:TW version} thus states that
$$
|\iint g(x) f(x-\gamma(t))\, |\det(\gamma'(t),\ldots,\gamma^{(d)}(t))|^{\frac2{d(d+1)}}\, dt\, dx| \lesssim_N \|g\|_{p_1}\|f\|_{p_2},
$$
for all $p_1,p_2$ as above, which is precisely the main theorem of \cite{BSjfa}.  One may analogously obtain the main result of \cite{DSjfa}, which considered the X-ray transform restricted to polynomial curves, as a special case of Theorem~\ref{T:TW version}.

\subsection*{Independence and necessity of Hypotheses (i) and (ii)}
Hypothesis (ii) of Theorem~\ref{T:main} certainly does not imply (i); nor does (i) imply (ii), as can be seen by considering, on the domain $U := (1,\infty) \times \R \times \R$, the maps
\begin{equation} \label{E:rotational}
\pi_1(x,y,z) := (y,z), \qquad \pi_2(x,y,z) := (x \cos(y+\frac zx), x \sin(y+\frac zx)),
\end{equation}
for which $X_1 = \partial_x$ and $X_2 = \partial_y-x\partial_z$.  

Hypothesis (ii) can be weakened to the assumption that a bounded number of integral curves constitute each fiber; this can be carried out by factoring the $\pi_j$ through the quotients of $\R^n$ by the integral curves of the $X_j$, as in the proof of Theorem~\ref{T:mfold version}.  The necessity of some hypothesis in this direction follows from the example \eqref{E:rotational} above.  Indeed, with this choice of $\pi_1,\pi_2$, \eqref{E:main} would suggest
$$
|\int_U f_1 \circ \pi_1(x) \, f_2 \circ \pi_2(x)\, dx| \lesssim \|f_1\|_{3/2}\|f_2\|_{3/2},
$$
which can be seen to fail for $f_1 := \chi_{\{|y| < R\}}$, $f_2 := \chi_{\{1 < |y| < 2\}}$, as $R \to \infty$.  

We expect that hypothesis (i) can weakened substantially, though at a cost of losing some uniformity (as will be seen momentarily).  Indeed, in the translation invariant case, this has been done \cite{OberlinLogCvx, DSjfa2, GressSbLvl}.  That being said, the conclusions of the theorem are false if we completely omit this hypothesis.  To see this, we consider first Sj\"olin's \cite{Sjolin} counter-example 
\begin{gather*}
\pi_1(x) := (x_1,x_2) \qquad \pi_2(x) := (x_1,x_2) - (x_3,\phi(x_3)), \\ \phi(x_3) := \sin(x_3^{-k})e^{-1/x_3}, \quad x \in \R^2 \times (0,\infty),
\end{gather*}
for $k$ sufficiently large.  Inequality \eqref{E:main} would suggest
$$
\bigl|\int_{\{0 < x_3 < 1\}} f_1\circ \pi_1(x) \, f_2 \circ \pi_2(x)\, |\phi''(x_3)|^{1/3}\, dx\bigr| \lesssim \|f_1\|_{3/2}\|f_2\|_{3/2},
$$
but this can be seen to fail for the characteristic functions $f_j = \chi_{E_j}$,
\begin{align*}
E_1 &:= \{y \in \R^2 : \delta < y_1 < \delta + \delta^2, \quad |y_2| \leq e^{-1/\delta}\}\\
E_2 &:= \{y \in \R^2 : |y_1| \lesssim \delta^2, \quad |y_2| \lesssim e^{-1/\delta}\},
\end{align*}
as $\delta \searrow 0$.

\subsection*{Malcev coordinates and the linear operator} For simplicity, we consider the Euclidean case.  We recall that we were initially interested in bilinear forms arising in the study of averages on curves, $\scriptB(f_1,f_2) = \langle f_1,Tf_2 \rangle$, where 
$$
Tf_2(x) := \int f_2(\gamma_x(t))\, d\mu_{\gamma_x}(t).
$$
Thus we are particularly interested in the case when $\pi_1$ is a coordinate projection, and dualizing the linear operator corresponds to changing variables so that $\pi_2$ is a coordinate projection.  As we will see, weak Malcev coordinates are sometimes useful in carrying this out.  

Fix a nilpotent Lie algebra $\mathfrak g$ generated by vector fields $X_1,X_2 \in \mathfrak g$.  Let $N$ denote the dimension of $\mathfrak g$, and let $\mathfrak z$ denote an $(N-n)$-dimensional Lie subalgebra of $\mathfrak g$.  As we have seen, there exists a weak Malcev basis $\{W_1,\ldots,W_N\}$ for $\mathfrak g$, with $\{W_{n+1},\ldots,W_N\}$ a basis for $\mathfrak z$, and, in the coordinates
$$
(x_1,\ldots,x_n,z_{n+1},\ldots,z_N) \mapsto e^{x_1 W_1} \cdots e^{x_n W_n} e^{z_{n+1}W_{n+1}} \cdots e^{z_NW_N}
$$
for the associated Lie group $G := \exp(\mathfrak g)$, the flows of the elements of $\mathfrak g$ are polynomial, and, moreover, the projection map $(x,z) \mapsto x$ defines a Lie group isomorphism of $\mathfrak g$ onto a Lie subgroup of $\mathfrak \R^n$, in which $\mathfrak z$ pushes forward to $\mathfrak z_0$, the algebra consisting of vector fields in (the pushforward of) $\mathfrak g$ that vanish at 0.  Thus we may identify $\R^n$ with $G/Z$, where $Z := \exp(\mathfrak z)$.  

If $W_1 = X_1$, then we define $\pi_1(x) := (x_2,\ldots,x_n)$.  (Alternately, there are local coordinates in which $\pi_1$ may be written in this form.)  If there exists another weak Malcev basis $\{\tilde W_1,\ldots,\tilde W_N\}$ for $\mathfrak g$ through $\mathfrak z$ with $\tilde W_1 = X_2$, then the map $F:x \mapsto \tilde x$ is a polynomial diffeomorphism, and so $\pi_2(x) := (\tilde x_2,\ldots,\tilde x_n)$ is also a polynomial.  The map $F$, being a polynomial diffeomorphism, has constant Jacobian determinant.  By scaling the $\tilde W_j$, we may assume that this constant is 1.  Our bilinear form is
\begin{align*}
B(f_1,f_2) &= \int f_1 \circ \pi_1(x) \, f_2 \circ \pi_2(x) \, \rho_\beta(x)\, dx \\
&= \int f_1 \circ \pi_1\circ F^{-1}(x) \, f_2 \circ \pi_2 \circ F^{-1}(x) \, \rho_\beta \circ F^{-1}(x)\, dx.
\end{align*}
Thus the associated linear  and adjoint operators are
$$
Tf(y) = \int f(\pi_2(t,y))\, \rho_\beta(t,y)\, dt, \quad T^*g(y) = \int g(\pi_1 \circ F^{-1}(t,y))\, \rho_\beta \circ F^{-1}(t,y)\, dt,
$$
averages along curves parametrized by polynomials.

It is therefore natural to ask when it is possible to find a weak Malcev basis of $\mathfrak g$ through $\mathfrak z$ whose first element is $X_1$.  

Initially fix any weak Malcev basis $\{W_1,\ldots,W_N\}$.  Let $\mathfrak g^{(2)} := [\mathfrak g,\mathfrak g]$, and let $\mathfrak h := \mathfrak g^{(2)}+\mathfrak z$.  Then $\mathfrak h$ is an ideal in $\mathfrak g$.  In fact, it is a proper ideal, because the linear span $\R W_2 + \cdots + \R W_N$ is an ideal (being a codimension 1 subalgebra) of $\mathfrak g$ that contains both $\mathfrak g^{(2)}$ and $\mathfrak z$.  Since $\mathfrak h \subseteq \R W_2+\cdots+\R W_N$, if $X_1 \in \mathfrak h$, we cannot take $W_1 = X_1$.  If $X_1 \notin \mathfrak h$, then there exists a weak Malcev basis $\{W_k,\ldots,W_N\}$ of $\mathfrak h$ through $\mathfrak z$, which we may complete to a basis $\scriptB := \{X_1,W_2,\ldots,W_N\}$ of $\mathfrak g$.  For each $2 \leq j \leq k$, the linear span $\R W_j + \cdots + \R W_N$ is an ideal in $\mathfrak g$, so $\scriptB$ is a weak Malcev basis of $\mathfrak g$ through $\mathfrak z$.  

Since $X_1,X_2$ generate $\mathfrak g$, both cannot lie in the proper subideal $\mathfrak h$, and so there does exist a weak Malcev basis with either $X_1$ or $X_2$ as the first element.

Malcev coordinates aside, we can ask when it is possible to express $\pi_1$ as a coordinate projection and $\pi_2$ as a polynomial, without changing the Lie algebra.  The authors have not strenuously endeavored to determine necessary and sufficient conditions, but it is clear that it is not possible in general, even locally around points where both maps are submersions.  Indeed, local polynomial maps extend to global ones generating the same Lie algebra, and a necessary condition for $\pi_1$ to be a coordinate projection is that 
$$
X_1 \notin \{X*Z*(-X) : Z \in \mathfrak z_0, \: X \in \mathfrak g\},
$$
since $\mathfrak z_{e^{X}(0)} = \{X * Z * (-X): Z \in \mathfrak z_0\}$.

\subsection*{Optimality of the weight}  It is proved in \cite{BSapde} that if $b$ is an extreme point of the Newton polytope $\scriptP$ defined in \eqref{E:big P}, then the corresponding weight $\rho_\beta$ is (up to summing weights corresponding to the same degree) the largest possible weight for which \eqref{E:main} can hold.  It is also shown that if $b$ is not on the boundary of $\scriptP$, then it is not possible to establish a pointwise bound on weights $\rho$ for which \eqref{E:main} might hold.  

\subsection*{Changes of speed and failure of global bounds}
The analogue of Proposition~\ref{P:local version} with $U_{x_0}$ replaced by the full region $U$ can fail if $\beta$ is not minimal and the Hodge-star vector fields are not themselves nilpotent, even when the Hodge-star vector fields are real analytic and have flows satisfying natural convexity hypotheses.  We see this by considering the example $U:=\{x \in \R^3 : x_3 > 0\}$ and 
$$
\pi_1(x) := (x_1,x_2), \qquad \pi_2(x) := (x_1,x_2) - (\log x_3,(\log x_3)^2).
$$
The Hodge-star vector fields are
$$
X_1 = \partial_3, \qquad X_2 = \frac1{x_3}\partial_1 + \frac 2{x_3} \log x_3 \partial_2 + \partial_3.
$$
Taking $Y_1 := x_3 X_1$ and $Y_2 := x_3 X_2$, we have $Y_{12} = 2 \partial_2$, and all higher order commutators are zero.  Taking $\beta = (0,2,0)$,
$$
\partial_t^\beta|_{t=0} \det D_t \,e^{t_3X_1} \circ e^{t_2X_2} \circ e^{t_1 X_1}(x) = -\frac3{x_3^4}.
$$
Thus \eqref{E:local version} would suggest the bound
\begin{equation} \label{E:log example}
|\int_U f_1 \circ \pi_1(x) \, f_2 \circ \pi_2(x) \, x_3^{-1} \, dx| \lesssim \|f_1\|_2 \|f_2\|_{4/3}.
\end{equation}
Changing variables, \eqref{E:log example} becomes
$$
|\int_{\R^3} f_1(x_1,x_2) \, f_2(x_1-t,x_2-t^2)  \, dt\, dx| \lesssim \|f_1\|_2\|f_2\|_{4/3},
$$
which is easily seen to be false by scaling.

It is still conceivable that global bounds are possible in the real analytic case when $\beta$ is minimal and some convexity/non-oscillation assumption is made.

\subsection*{Failure of strong type bounds in dimension 2}
The hypothesis $n \geq 3$ in Theorem~\ref{T:main} cannot be omitted.  Indeed, consider $\pi_1(x_1,x_2) := x_1$, $\pi_2(x_1,x_2) := x_2^k$.  Then $X_1 = \frac{\partial}{\partial x_2}$, $X_2 = k x_2^{k-1}\frac{\partial}{\partial x_1}$, which together generate a nilpotent Lie algebra with polynomial flows.  Moreover, if we take $\beta = (k-1,0)$, then the corresponding weight is $\rho_\beta \sim 1$, so \eqref{E:main} would suggest
$$
|\int f_1(x_1) f_2(x_2^k)\, dx_1 \, dx_2| \lesssim \|f_1\|_1\|f_2\|_k,
$$
which is false in general (take e.g.\ $f_2(y) = (y^{\frac1k}\log y)^{-1} \chi_{(0,1]}$).  

We recall, however, that the argument in \cite{GressmanPoly} (and also the proof of Proposition~\ref{P:RWT}) did not require the hypothesis $n \geq 3$ to obtain the restricted weak type inequality on the single scale $\rho_\beta \sim 1$.  

\subsection*{Failure of global estimates for an analogue on manifolds}  An interesting question that we do not investigate is whether there are natural, simple hypotheses leading to global estimates in Theorem~\ref{T:mfold version}.  Without further hypotheses, such a result is false, as can be seen in the following counterexample.  

Let $M = \R/\Z \times \R/\Z \times \R$ and $P_1 = P_2 = \R/\Z \times \R/\Z$, all equipped with Lebesgue measure.  Define projections
$$
\pi_1(\theta_1,\theta_2,t) = (\theta_1,\theta_2), \qquad \pi_2(\theta_1,\theta_2,t) = (\theta_1+t,\theta_2+t^2).
$$
Then the $\pi_j$ naturally give rise to the vector fields 
$$
X_1 = \tfrac\partial{\partial t}, \qquad X_2 = \tfrac\partial{\partial t} - \tfrac\partial{\partial \theta_1}-2t\tfrac\partial{\partial \theta_2}.
$$
These generate a nilpotent Lie algebra obeying the H\"ormander condition, and, moreover, each point of $P_j$ has its $\pi_j$-preimage contained in a unique integral curve of $X_j$.  A naive analogue of Theorem~\ref{T:main} might suggest
$$
\int_M |f_1 \circ \pi_1 \, f_2 \circ \pi_2| \lesssim \|f_1\|_{\frac3 2}\|f_2\|_{\frac 3 2},
$$
but this is obviously false, as can be seen by taking $f_1 \equiv f_2 \equiv 1$.  

\subsection*{Multilinear averages on curves}  
In the multilinear case considered in \cite{BSrevista, BSapde}, the natural generalization of the map $\Psi_x$ used to define $\rho$ involves iteratively exponentiating the vector fields in some specified order, and the single-scale restricted weak type inequality is known to hold under the natural analogue of the hypotheses of Theorem~\ref{T:main}.  Indeed, the proof in Section~\ref{S:RWT} readily generalizes.  Unfortunately, the analogy breaks down in Section~\ref{S:Quasiex}, where we need to use the gain coming from nonzero entries of the multiindex $\beta$.  To rule out such examples in the multilinear case would require rather more complicated hypotheses, particularly if we want a theory that includes examples such as the perturbed Loomis--Whitney inequality, where the endpoint bounds are known to hold \cite{BCW}.  

We record here two multilinear examples that may be of interest in future explorations of this topic.  

The first is a Loomis--Whitney inspired variant on the above two-dimensional example.  Define
$$
\pi_i(x) := (x_1,\ldots,\widehat{x_i},\ldots,x_n), \quad 1 \leq i \leq n-1, \qquad \pi_n(x) := (x_1^k,x_2,\ldots,x_{n-1}).
$$
Our vector fields are $X_i = \frac{\partial}{\partial x_i}$, $1 \leq i \leq n-1$, and $X_n = kx_1^{k-1}\frac{\partial}{\partial x_n}$, and the endpoint inequality
$$
|\int \prod_i f_i \circ \pi_i \, dx| \lesssim \prod_{i=1}^n \|f_i\|_{p_i}, \qquad p_1 = \tfrac{n+k-2}{k}, \: p_i = n+k-2, \: i=2,\ldots,n
$$
is false for $k > 1$, as can be seen by considering $f_1 := \chi_{B_1}$ and $f_i(x_1,x') := |x_1|^{-\frac1{n+k-2}} |\log|x_1||^{-\frac1{n-1}} \chi_{B_1}(x)$, where $B_1$ denotes the unit ball.  

The second is a hybrid of a well-studied convolution operator with this example.  For $(x,t,s) \in \R^{n+1+1}$, let $\pi_1(x,t,s) := (x,s)$, $\pi_2(x,t,s) := (x-\gamma(t),s)$, $\pi_3(x,t,s) := (x,t^k)$, where $\gamma(t) := (t,t^2,\ldots,t^n)$.  Our vector fields are $X_1 = \frac{\partial}{\partial t}$, $X_2 = \frac{\partial}{\partial t} - \gamma'(t) \cdot \nabla_x$, $X_3 = kt^{k-1}\frac{\partial}{\partial s}$.  From the preceding examples, we might guess that the endpoint inequality
\begin{gather*}
|\int \prod_{i=1}^3 f_i \circ \pi_i\, dx| \lesssim \prod_{i=1}^3 \|f_i\|_{p_i}, \\
p_1:= \frac{2k+n(n+1)}{2k+n(n-1)}, \qquad p_2:= \frac{2k+n(n+1)}{2n}, \qquad p_3:= \frac{2k+n(n+1)}2
\end{gather*}
fails.  In fact, this inequality is true, as can be seen from H\"older's inequality and Theorem~2.3 of \cite{DSjfa2}.

%%%%%%%%%%%%%%%%%%%%%%%%%%%%%%%%%%%%%%%%
%%%%%%%%%%%%%%%%%%%%%%%%%%%%%%%%%%%%%%%%
%%%%%%%%%%%%%%%%%%%%%%%%%%%%%%%%%%%%%%%%

\section{Appendix:  Polynomial lemmas} \label{S:Polynomials}

%%%%%%%%%%%%%%%%%%%%%%%%%%%%%%%%%%%%%%%%
%%%%%%%%%%%%%%%%%%%%%%%%%%%%%%%%%%%%%%%%
%%%%%%%%%%%%%%%%%%%%%%%%%%%%%%%%%%%%%%%%

In this section, we collect together a number of lemmas on the size and injectivity of polynomials.

The next lemma shows that if a polynomial bounds a monomial, then the monomial must in fact be bounded by two terms of the polynomial; this facilitates a complex interpolation argument used in the deduction of Theorem~\ref{T:TW version} from Theorem~\ref{T:main}.  

\begin{lemma}\label{L:extract 2 terms}
Let $p(t) = \sum_{n=0}^N a_n t^n$ be a polynomial with nonnegative coefficients, and let $k \in \Z_{\geq 0}$. If $t^k \leq p(t)$ for all $t > 0$, then $a_k \gtrsim 1$ or there exist $n_1 < k < n_2$ such that $(a_{n_1})^{\frac{n_2-k}{n_2-n_1}}(a_{n_2})^{\frac{k-n_1}{n_2-n_1}} \gtrsim 1$.  Conversely, if $a_k \geq 1$ or $(a_{n_1})^{\frac{n_2-k}{n_2-n_1}}(a_{n_2})^{\frac{k-n_1}{n_2-n_1}} \geq 1$ for some $n_1 < k < n_2$, then $t^k \leq p(t)$ for all $t > 0$.  
\end{lemma}

\begin{proof}
If $t^k \leq p(t)$ for all $t \geq 0$, but $a_k \leq \tfrac12$, then $t^k \leq 2(p(t)-a_k t^k)$, so we may as well assume that $a_k = 0$.  

Let $p_{lo}(t) := \sum_{n < k} a_n t^n$ and $p_{hi}(t) := \sum_{n > k} a_n t^n$.  By considering small $t$, we see that $p_{lo} \not\equiv 0$, and by considering large $t$, we see that $p_{hi} \not \equiv 0$.    By a routine application of the Intermediate Value Theorem, there exists a unique $t_0 > 0$ such that $p_{lo}(t_0) = p_{hi}(t_0)$.  We may choose $n_1 < k < n_2$ such that $p_{lo}(t_0) \sim a_{n_1}t_0^{n_1}$ and $p_{hi}(t_0) \sim a_{n_2}t_0^{n_2}$.  Thus $t_0^k \lesssim a_{n_1}t_0^{n_1} \sim a_{n_2}t_0^{n_2}$, from which we learn that $t_0 \sim (\frac{a_{n_1}}{a_{n_2}})^{\frac{1}{n_2-n_1}}$ and, consequently, $1 \lesssim (a_{n_1})^{\frac{n_2-k}{n_2-n_1}}(a_{n_2})^{\frac{k-n_1}{n_2-n_1}}$.  In the converse direction, if $(a_{n_1})^{\frac{n_2-k}{n_2-n_1}}(a_{n_2})^{\frac{k-n_1}{n_2-n_1}} \geq 1$, then at $t_0 := (\frac{a_{n_1}}{a_{n_2}})^{\frac{1}{n_2-n_1}}$, $t_0^k \leq a_{n_1} t_0^{n_1} = a_{n_2}t_0^{n_2}$, so $t^k \leq a_{n_1} t^{n_1}$ for all $t \leq t_0$ and $t^k \leq a_{n_2} t^{n_2}$ for all $t \geq t_0$.  
\end{proof}

A lemma in \cite{DW} states that if $\scriptP$ is a finite collection of polynomials on $\R$, each of degree at most $N$, then there exists a decomposition $\R = \bigcup_{j=1}^{C(\#\scriptP,N)} I_j$, with each $I_j$ an interval, such that on $I_j$, each $p$ has roughly the same size as some fixed monomial, centered at a point that depends only on $j$, not $p$:
$$
|p(t)| \sim a_{p,j}(t-b_j)^{k_{p,j}}, \qquad a_{p,j} \in [0,\infty), \quad b_j \notin I_j, \quad k_{p,j} \geq 0.  
$$
Our next lemma strengthens this to show that we may take each monomial to be an entry of the Taylor polynomial centered at $b_j$ of the polynomial $p$ and ensures that the other entries of that Taylor polynomial are as small as we like.

\begin{lemma} \label{L:p sim t^k}
Let $\scriptP$ denote a finite collection of polynomials on $\R$, each having degree at most $N$, and let $\eps > 0$.  There exist a collection of nonoverlapping open intervals $I_1,\ldots,I_{N'}$, with $N' = N'(N,\#\scriptP,\eps)$ and $\R = \bigcup_j \overline{I_j}$, and centers $b_1,\ldots,b_{N'}$, with $b_j \notin I_j$, such that for each $j$ and $p \in \scriptP$, there exists an integer $k_{j,p}$ such that
\begin{equation} \label{E:p sim t^k L}
|\tfrac1{k!} p^{(k)}(b_j)(t-b_j)^k| \leq \eps \tfrac1{k_{j,p}!}| p^{(k_{j,p})}(b_j)(t-b_j)^{k_{j,p}}|, \qquad k \neq k_{j,p}, \: t \in I_j.
\end{equation}
\end{lemma}

In particular, provided we take $\eps < \frac1{2N}$, 
\begin{equation}\label{E:p sim t^k}
|p(t)| \sim A_{j,p} |t-b_j|^{k_{j,p}}, \: \text{for} \: t \in I_j, \: \text{where}\: A_{j,p} := |\tfrac1{k_{j,p}!} p^{(k_{j,p})}(b_j)|.
\end{equation}

\begin{proof}[Proof of Lemma~\ref{L:p sim t^k}]
We modify the approach from \cite{DW}.  We will allow the integer $N'$ to change from line to line, subject to the constraint $N' = N'(N,\#\scriptP,\eps)$.  

Without loss of generality, all elements of $\scriptP$ are nonconstant, and $\scriptP$ contains all nonconstant derivatives of its elements.  Let $\{z_1,\ldots,z_{N'}\}$ denote the union of the (complex) zero sets of the elements in $\scriptP$.  Set
$$
S_i := \{t \in \R : |t-z_i| \leq |t-z_j|, \: j \neq i\}.
$$
Then $\R = \bigcup_i S_i$.  We will further decompose each $S_i$, and by reindexing, it suffices to further decompose $S_1$.  Reindexing, we may assume that $|z_1-z_2| \leq \cdots \leq |z_1-z_N|$.  Define
\begin{gather*}
T_j := \{t \in S_1 : \tfrac12|z_1-z_j| \leq |t-z_1| < \tfrac12|z_1-z_{j+1}|\}, \qquad j=1,\ldots,N'-1,\\
 T_{N'}:= S_1 \setminus T_{N'-1}.
\end{gather*}
If $t \in T_j$, then by the triangle inequality,
$$
|t-z_1| \leq |t-z_{j'}| \leq 3|t-z_1|, \quad j' \leq j, \qquad \tfrac12|z_1-z_{j'}| \leq |t-z_{j'}| < \tfrac32|z_1-z_{j'}|, \quad j' > j.
$$
Writing $p(t) =  A_p \prod_{j' \in \scriptJ_p} (t-z_j)^{n_{j',p}}$, where $\scriptJ_p$ denotes the set of indices corresponding to zeros of $p$,
$$
|p(t)| \sim |A_p \prod_{j < j' \in \scriptJ_p}(z_1-z_{j'})^{n_{j',p}} \prod_{j > j' \in \scriptJ_p}(t-z_1)^{n_{j',p}}|.
$$ 
Thus $p$ is comparable to a complex monomial.  Let $b_1 := \rm{Re}\, z_1$, $c_1 := |\rm{Im}\, z_1|$.  On $\{|t-b_1| \leq c_1\}$, $|t-z_1| \sim c_1$, and on $\{|t-b_1| \geq c_1\}$, $|t-z_1| \sim |t-b_1|$, so subdividing one more time, we obtain intervals on which each polynomial is comparable to a real monomial.  

More precisely, at this point, we have simply reproved the lemma from \cite{DW}:  There exists a decomposition $\R = \bigcup_{j=1}^{N'} I_j$ such that $|p(t)| \sim a_{p,j}|t-b_j|^{n_{p,j}}$, $p \in \scriptP$ and $t \in I_j$.  We want a bit more, which requires us to subdivide further.  Reindexing, it suffices to subdivide $I_1$.  Translating, we may assume that $b_1 = 0$, and by symmetry, we may assume that $I_1 = (\ell,r) \subseteq (0,\infty)$.  To fix our notation, 
\begin{equation}\label{E:def ap}
|p(t)| \sim a_p |t|^{n_p}, \qquad t \in I:=I_1.
\end{equation}

If $I = (0,\infty)$, then each $p$ must in fact be a monomial, and we are done.  Otherwise, by rescaling, we may assume that either $r=1$ or that $\ell = 1, r=\infty$.  

Case I:  $I = (\ell,1)$.  By construction, $z_1$, which is purely imaginary, is no further from 1 than any zero of any nonzero derivative of any element of $\scriptP$.  Thus no element of $\scriptP$ (nor any nonzero derivative of any element) has a zero inside the disk $\{|z-1| < 1\}$.  Therefore for each $p \in \scriptP$, $|p|$ is monotone on $(0,2)$.  If $|p|$ is decreasing, by equivalence of norms,
$$
|p(0)| = \|p\|_{C^0(0,2)} \sim \|p\|_{C^0(1,2)} = |p(1)|.
$$
Thus either $|p|$ is increasing or $|p| \sim c_p$ on all of $(0,1)$.  In either case, for $t \in (0,1)$,
\begin{equation} \label{E:p inc}
|p(t)| \sim \|p\|_{L^\infty((0,t))} \sim \sum_{j=0}^N \tfrac1{j!}|p^{(j)}(0)|t^j \sim \sum_{j=0}^N \tfrac1{j!}\|p^{(j)}\|_{L^\infty((0,t))}t^j.
\end{equation}

Let $\eps^{-1} \leq j < \eps^{-2}$, and let $n \geq 1$.  By \eqref{E:p inc}, $|p^{(n)}(j\eps^2)|(j\eps^2)^n \lesssim |p(j\eps^2)|$, so $|p^{(n)}(j\eps^2)|(\eps^2)^n \lesssim \eps^n |p(j\eps^2)|$.  Therefore \eqref{E:p sim t^k L} holds for $t \in [j\eps^2,(j+1)\eps^2]$ with $b_j = j\eps^2$ and $k_{j,p} = 0$.  

It remains to decompose $(\ell,\eps)$, supposing this interval is nonempty.  Evaluating \eqref{E:p inc} at $t=1$, and recalling \eqref{E:def ap},
$$
\sum_{n >n_p} \tfrac1{n!} |p^{(n)}(0)| \lesssim a_p.
$$
Thus for $0<t < \eps$ and $n > n_p$, $|p^{(n)}(0)|t^n < \eps a_pt^{n_p}$.  Evaluating \eqref{E:p inc} at $t=\ell$,
$$
\sum_{n < n_p} |\tfrac1{n!} p^{(n)}(0)|\ell^n \lesssim a_p \ell^{n_p},
$$
so for $t > \eps^{-1}\ell$ and $n < n_p$, $|p^{(n)}(0)|t^n < \eps a_p t^{n_p}$.  Therefore \eqref{E:p sim t^k L} holds on $(\eps^{-1}\ell,\eps)$ with $b_j = 0$ and $k_{j,p} = n_p$.  This leaves us to decompose $(\ell,\eps^{-1}\ell)$.  By scaling, this is equivalent to decomposing $(\eps,1)$, which we have already shown how to do.  

Case II:  $I = (1,\infty)$.  By construction, $z_1$ is nearer to each $t > 1$ than any zero of any derivative of any element of $\scriptP$.  Thus no element of $\scriptP$ has a zero with positive real part, so each $|p(t)|$ is nonvanishing with nonvanishing derivative on $(0,\infty)$, and thus must be increasing on $(0,\infty)$.  Therefore \eqref{E:p inc} holds for each $t \in (0,\infty)$.    Taking limits, we see that for $0 \neq p \in \scriptP$, $n_p = \deg p$ and $a_p = \tfrac1{n_p!}p^{(n_p)}(0)$.  Evaluating at 1, $\sum_n \tfrac1{n!}|p^{(n)}(0)| \lesssim \tfrac1{n_p!}|p^{(n_p)}(0)|$, so for $t > \eps^{-1}$ and $n < n_p$, $|p^{(n)}(0)| t^n < \eps |p^{(n_p)}(0)|t^{n_p}$.  This leaves us to decompose $(1,\eps^{-1})$, which rescales to $(\eps,1)$, so the proof is complete.  
\end{proof}

The next lemma applies Lemma~\ref{L:p sim t^k} to make precise the heuristic that products of polynomials must vary at least as much as the original polynomials.  

\begin{lemma}\label{L:p sim 1/q}
Let $p_1$ and $p_2$ be polynomials on $\R$ of degree at most $N$, and let $a_1,a_2$ be positive integers.  The number of integers $k$ for which there exists $t_k \in \R$ such that 
\begin{equation} \label{E:p sim 1/q}
|p_1(t_k)| \sim 2^{a_1 k}, \qquad |p_2(t_k)| \sim 2^{-a_2 k}
\end{equation}
is bounded by a constant depending only on $N$.  
\end{lemma}

\begin{proof}[Proof of Lemma~\ref{L:p sim 1/q}]
The conclusion is trivial for monomials, so by Lemma~\ref{L:p sim 1/q}, it follows for arbitrary polynomials.  
\end{proof}

The next lemma extends Lemma~\ref{L:p sim t^k} to polynomial curves $\gamma:\R \to \R^n$, allowing us to closely approximate a given polynomial curve by a constant vector multiple of a monomial. 

\begin{lemma} \label{L:gamma sim t^k}
Let $N$ be an integer and let $\eps > 0$.  Let $\gamma:\R \to \R^n$ be a polynomial of degree at most $N$.  There exist nonoverlapping open intervals $I_1,\ldots,I_{N'}$, with $N'= N'(N,n,\eps)$ and $\R = \bigcup_j \overline{I_j}$, and centers $b_1,\ldots,b_{N'}$, with $b_j \notin I_j$, such that for each $j$, there exists an integer $k_j$ such that
\begin{equation} \label{E:gamma t^k}
|\tfrac1{k!} \gamma^{(k)}(b_j)(t-b_j)^k| \leq \eps |\tfrac1{k_j!} \gamma^{(k_j)}(b_j)(t-b_j)^{k_j}|, \qquad k \neq k_j, t \in I_j.
\end{equation}
\end{lemma}

\begin{proof}[Proof of Lemma~\ref{L:gamma sim t^k}]
By Lemma~\ref{L:p sim t^k}, it suffices to decompose an interval $I \subseteq \R$ for which there exists a point $b \notin I$ and integers $0 \leq k_1,\ldots,k_n \leq N$ such that the coordinates of $\gamma$ satisfy
\begin{equation} \label{E:gamma_i sim t^k}
|\tfrac1{k!} \gamma_i^{(k)}(b)(t-b)^k| < \eps |\tfrac1{k_i!} \gamma_i^{(k_i)}(b)(t-b)^{k_i}|, \qquad t \in I, \: k \neq k_i.
\end{equation}

Making a finite decomposition of $I$ and reindexing our coordinates if needed, we may assume that 
$$
|\tfrac1{k_1!}\gamma_1^{(k_1)}(b)(t-b)^{k_1}| \geq |\tfrac1{k_i!}\gamma_i^{(k_i)}(b)(t-b)^{k_i}|, \qquad i=2,\ldots,n, \: t \in I.
$$
Thus for $t \in I$,
\begin{equation} \label{E:gamma k1 big}
|\tfrac1{k!} \gamma^{(k)}(b)(t-b)^k| \lesssim |\tfrac1{k_1!} \gamma^{(k_1)}(b)(t-b)^{k_1}|, \qquad 
 |\gamma(t)| \sim |\tfrac1{k_1!}\gamma^{(k_1)}(b)(t-b)^{k_1}|.
\end{equation}
Translating, reflecting, and rescaling, we may assume that $b=0$ and that  $I=(\ell,r)$.

Define curves
$$
\gamma_{lo}(t) := \sum_{k < k_1} \tfrac1{k!} \gamma^{(k)}(0) t^k, \qquad \gamma_{hi}(t) := \sum_{k > k_1} \tfrac1{k!} \gamma^{(k)}(0) t^k, \qquad \gamma_{k_1}(t) := \tfrac1{k_1!} \gamma^{(k_1)}(0) t^{k_1}
$$
and intervals
\begin{gather*}
I_{lo}^0 := (\ell \eps^{-1},r), \qquad I_{lo}^j := (\ell \eps^2 j,\ell\eps^2(j+1)), \quad \eps^{-2} \leq j < \eps^{-3}\\
I_{hi}^0 := (\ell ,r\eps), \qquad I_{hi}^j := (r \eps^2 j,r\eps^2(j+1)), \quad \eps^{-1} \leq j < \eps^{-2}.
\end{gather*}
Since $I$ is a bounded union of `$lo$' intervals and also a bounded union of `$hi$' intervals, $I$ may be written as a bounded union of intersections of one `$lo$' interval with one `$hi$' interval.  We will show that such intersections have the properties claimed in the lemma.  

Evaluating \eqref{E:gamma k1 big} at $\ell$, $\tfrac1{k!}|\gamma^{(k)}(0)|\ell^k \lesssim \tfrac1{k_1!}|\gamma^{(k_1)}(0)| \ell^{k_1}$.  Therefore
$$
|\gamma_{lo}(t)| \lesssim \eps |\gamma_{k_1}(t)|, \qquad t \in I_{lo}^0.
$$
By \eqref{E:gamma k1 big} and a Taylor expansion of $\gamma^{(m)}$ about 0, 
$$
|\gamma^{(m)}(t)|t^m \lesssim |\gamma^{(k_1)}(0)| t^{k_1} \lesssim |\gamma(t)|, \qquad t \in I.
$$
Thus for $m \geq 1$, $\eps^{-2} \leq j \leq \eps^{-3}$ and $t \in I_{lo}^j$, 
$$
|\gamma^{(m)}(\ell \eps^2 j)|(t-\ell \eps^2 j)^m \leq j^{-m}|\gamma^{(m)}(\ell \eps^2 j)|(\ell \eps^2 j)^m \lesssim \eps |\gamma(\ell \eps^2 j)|.
$$
Arguing analogously, 
\begin{gather*}
|\gamma_{hi}(t)| \lesssim \eps |\gamma_{k_1}(t)|, \qquad t \in I_{hi}^0,\\
|\gamma^{(m)}(r\eps^2 j)|(t-r\eps^2 j)^m  \lesssim \eps|\gamma(r\eps^2j)|, \qquad m \geq 1,\: t \in I^j_{hi}, \: \eps^{-1} \leq j < \eps^{-2}. 
\end{gather*}
Putting these inequalities together, \eqref{E:gamma t^k} holds:\\
-On $I^0_{lo} \cap I^0_{hi}$ with center $b_0 = 0$ and $k_0 = k_1$\\
-On $I_{lo}^{j_1} \cap I_{hi}^{j_2}$, for $(j_1,j_2) \neq (0,0)$, with center $b_j = \ell \eps^2 j$ and $k_0 = 0$.\\
\end{proof}

The next lemma applies Lemma~\ref{L:gamma sim t^k} to make precise the heuristic that, for $\gamma:\R \to \R^n$, since the derivative $\gamma'$ drives the curve forward, $\gamma$ and $\gamma'$ are typically almost parallel.  This result is crucial to proving Proposition~\ref{P:local strong type}.  

\begin{lemma} \label{L:gamma || gamma'}
There exists $M=M(N)$ sufficiently large that for all $\eps > 0$, there exists $\delta > 0$ such that if 
\begin{equation} \label{E:gamma many scales}
|\gamma(t_i)| < \delta|\gamma(t_{i+1})|, \qquad i=1,\ldots,M-1,
\end{equation}
then
\begin{equation} \label{E:gamma || gamma'}
|\gamma(t_i) \wedge \gamma'(t_i)| < \eps |\gamma(t_i)| |\gamma'(t_i)|,
\end{equation}
for some $1 \leq i \leq M$.
\end{lemma}

\begin{proof}[Proof of Lemma~\ref{L:gamma || gamma'}]
Performing a harmless translation and applying Lemma~\ref{L:gamma sim t^k}, it suffices to prove that there exists $M$ such that \eqref{E:gamma || gamma'} holds whenever \eqref{E:gamma many scales} holds with all $t_i$ lying in some interval $I$ on which
\begin{equation} \label{E:k0 1-dominates}
|\tfrac1{k!} \gamma^{(k)}(0) t^k| \leq |\tfrac1{k_0!} \gamma^{(k_0)}(0) t^{k_0}| \sim |\gamma(t)|,
\end{equation}
for all $0 \leq k \leq N$ and some $0 \leq k_0 \leq N$.  Moreover, by \eqref{E:gamma many scales}, we may assume that $k_0 \neq 0$.  

For $\delta>0$ sufficiently small, and each $k_0 \neq k$, by \eqref{E:gamma many scales} and \eqref{E:k0 1-dominates} the inequality
$$
\eps|\tfrac1{k_0!} \gamma^{(k_0)}(0) t_i^{k_0}| < |\tfrac1{k!} \gamma^{(k)}(0) t_i^k| 
$$
can only hold for a bounded number of $t_i$, so we may assume further that
$$
|\tfrac1{k!}\gamma^{(k)}(0)t_i^k| < \eps |\tfrac1{k_0!} \gamma^{(k_0)}(0) t_i^{k_0}|, \qquad k \neq k_0.
$$
Therefore
\begin{align*}
|\gamma(t_i) \wedge \gamma'(t_i)| &\leq |(\sum_{k \neq k_0} \tfrac1{k!}\gamma^{(k)}(0)t_i^k) \wedge \tfrac{k_0}{k_0!}\gamma^{(k_0)}(0)t_i^{k_0-1}| \\
&\qquad\qquad+ |\tfrac1{k_0!}\gamma^{(k_0)}(0)t_i^{k_0} \wedge (\sum_{k \neq k_0}\tfrac{k}{k!}\gamma^{(k)}(0)t_i^k)|\\
&\qquad\qquad + |(\sum_{k \neq k_0} \tfrac1{k!}\gamma^{(k)}(0)t_i^k) \wedge (\sum_{k \neq k_0}\tfrac{k}{k!}\gamma^{(k)}(0)t_i^k)|\\
&\lesssim \eps |\gamma(t_i)||\gamma'(t_i)|.
\end{align*}
\end{proof}

Next, we use basic facts from algebraic geometry to prove several lemmas about polynomials of $n$ variables of degree at most $N$.  We say that a quantity is bounded if it is bounded from above by a constant depending only on the dimension $n$ and the degree $N$, not on the particular polynomials in question.

Our main tool for lemmas below is the following theorem from algebraic geometry.

\begin{theorem}[\cite{Ful98}] \label{T:quantitative ag}
Let $f_1,\ldots,f_k:\C^n \to \C$ be polynomials of degree at most $N$ and let $Z \subseteq \C^n$ be the associated variety, i.e.\
$$
Z := \{z \in \C^n : f_1(z) = \cdots = f_k(z) = 0\}.
$$
Then we may decompose
\begin{equation} \label{E:decompose Z}
Z = \bigcup_{i=1}^{C(k,n,N)} Z_i,
\end{equation}
where each $Z_i$ is an irreducible variety. 
\end{theorem}

In particular, the decomposition in \eqref{E:decompose Z} involves a bounded number of dimension zero irreducible subvarieties.  We recall, and will repeatedly use the fact that the irreducible subvariety containing an isolated point of $Z$ must be a singleton.

Theorem~\ref{T:quantitative ag} follows from the refined version of Bezout's Theorem, Example 12.3.1 of \cite{Ful98}, which implies that $\sum_{i=1}^s \deg(Z_i) \leq \prod_{i=1}^k \deg(f_i)$.  Since $\deg Z_i \geq 1$ for each $i$, this suffices.  

\begin{lemma} \label{L:finite to one}
Let $P:\R^n \to \R^n$ be a polynomial.  Then, with respect to Lebesgue measure on $\R^n$, almost every point in $P(\R^n)$ has a bounded number of preimages.
\end{lemma}

\begin{proof}
It suffices to show that if $y \notin P(\{\det DP \neq 0\})$, then $y$ has a bounded number of preimages.  For such a point $y$, define
$$
Z_y := \{z \in \C^n : P(z) -y = 0\}.
$$
By the Inverse Function Theorem and our hypothesis on $y$, real points $x \in Z_y \cap \R^n$ are isolated (complex) points of $Z_y$.  By Theorem~\ref{T:quantitative ag} and the fact that dimension zero irreducible varieties are singletons, $Z_y$ contains a bounded number of isolated points.  
\end{proof}

\begin{lemma}\label{L:intersect cylinder}
Let $T$ denote the tube
$$
T:= \{x=(x',x_n) \in \R^{n} : |x'| < 1\}.
$$
Let $P:\R^n \to \R^n$ be a polynomial of degree at most $N$ and assume that $\det DP$ is nonvanishing on $\overline T$.  If $\gamma:\R \to \R^n$ is a polynomial of degree at most $N$, then $\gamma^{-1}[P(T)]$ is a union of a bounded number of intervals.  
\end{lemma}

\begin{proof}
Consider the complex varieties 
\begin{align*}
C &:= \{v \in \C^n : \sum_{i = 1}^{n-1} v_i^2 = 1\}\\
Z &:= \{(u,v) \in \C^{1+n} : \gamma(u) = P(v), \: v \in C\}.
\end{align*}
Suppose that $(t,x) \in Z \cap \R^n$ is a regular point of some subvariety $Z' \subseteq Z$, with $\dim Z' > 0$.  If $\det DP(x) \neq 0$, then by the implicit function theorem, $Z'$ can have complex dimension at most one, and, moreover, if the dimension of $Z'$ is one, then there exists a complex neighborhood $U$ of $t$ such that $\gamma(U) \subseteq P(C)$.  Shrinking $U$ if necessary, and again using $\det DP(x) \neq 0$, $\gamma(U \cap \R) \subseteq P(C \cap \R^n)=P(\partial T)$.  Thus a boundary point of $\gamma^{-1}[P(T)]$ must be a regular point of a dimension zero subvariety $Z' \subseteq  Z$, so by Theorem~\ref{T:quantitative ag}, the number of boundary points is bounded.  
\end{proof}

\begin{lemma}\label{L:proj gamma}
Let $P:\R^n \to \R^n$, $\gamma:\R \to \R^n$, and $Q:\R^n \to \R^n$ be polynomials of degree at most $N$, and assume that $P^{-1}$ is defined and differentiable on a neighborhood of the image $\gamma(I)$, for some open interval $I$.  Then no coordinate of the vector $[(DP^{-1})\circ \gamma](Q \circ \gamma)$ can change sign more than a bounded number of times on $I$.
\end{lemma}

\begin{proof}  
Multiplying the vector $[(DP^{-1})\circ \gamma](Q \circ \gamma)$ by $(\det DP)^2$ and using Cramer's rule, it is enough to prove that if $R:\R^n \to \R^n$ is a polynomial of bounded degree, then 
\begin{equation} \label{E:RPQ}
(R \circ P^{-1} \circ \gamma)\cdot(Q \circ \gamma)
\end{equation}
changes sign a bounded number of times on $I$.  

We consider the complex variety
$$
Z:= \{(u,v,w) \in \C \times \C^n \times \C^n : \gamma(u) = P(v) = w, \: R(v) \cdot Q(w) = 0\}.
$$
If \eqref{E:RPQ} vanishes at $t \in I$, then $(t,P^{-1}(\gamma(t)),\gamma(t))=:(t,x,y) \in Z$ and $\det DP(x) \neq 0$.  

Let $Z' \subseteq Z$ denote an irreducible subvariety from the decomposition \eqref{E:decompose Z} for which $(t,x,y)$ is a regular point.  By the Implicit Function Theorem and $\det DP(x) \neq 0$, either $Z'$ has dimension zero, or $Z'$ has (complex) dimension one and \eqref{E:RPQ} vanishes on a (complex) neighborhood of $t$.  Only a bounded number of points can lie on dimension zero subvarieties, and if \eqref{E:RPQ} vanishes on a neighborhood of $t$, then it vanishes on all of $I$ by analyticity.  Either way, the number of sign changes is bounded.  
\end{proof}

%%%%%%%%%%%%%%%%%%%%%%%%%%%%%%%%%%%%%%%%%%%%%%%
%%%%%%%%%%%%%%%%%%%%%%%%%%%%%%%%%%%%%%%%%%%%%%%
%%%%%%%%%%%%%%%%%%%%%%%%%%%%%%%%%%%%%%%%%%%%%%%

%%%%%%%%%%%%%%%%%%%%%%%%%%%%%%%%%%%%%%%%%%%%%%%
%%%%%%%%%%%%%%%%%%%%%%%%%%%%%%%%%%%%%%%%%%%%%%%
%%%%%%%%%%%%%%%%%%%%%%%%%%%%%%%%%%%%%%%%%%%%%%%

\end{document}